\documentclass[letterpaper]{article} 
\usepackage{aaai25}  
\usepackage{times}  
\usepackage{helvet}  
\usepackage{courier}  
\usepackage[hyphens]{url}  
\usepackage{graphicx} 
\usepackage{natbib}  
\usepackage{caption} 
\usepackage{algorithm}

\usepackage{newfloat}
\usepackage{listings}
\DeclareCaptionStyle{ruled}{labelfont=normalfont,labelsep=colon,strut=off} 
\floatstyle{ruled}
\newfloat{listing}{tb}{lst}{}
\floatname{listing}{Listing}
\usepackage[utf8]{inputenc} 
\usepackage[T1]{fontenc}    
\usepackage{url}            
\usepackage{booktabs}       
\usepackage{amsfonts}       
\usepackage{nicefrac}       
\usepackage{microtype}      
\usepackage{xcolor}         
\usepackage{amsmath,amsfonts,amssymb,amsthm,array}
\usepackage{algorithm}
\usepackage[font=small]{caption}
\usepackage{subcaption}
\usepackage{algpseudocode}
\usepackage{bm}
\usepackage{color}
\usepackage{graphicx}
\usepackage{enumitem}
\usepackage{mathtools}
\usepackage{appendix}
\usepackage{cleveref}
\usepackage[makeroom]{cancel}
\usepackage{multirow}
\usepackage{threeparttable}
\usepackage{makecell}
\usepackage{colortbl}
\usepackage{natbib}

\usepackage{indentfirst}
\allowdisplaybreaks

\newtheorem{theorem}{Theorem}
\newtheorem{proposition}{Proposition}
\newtheorem{lemma}{Lemma}
\newtheorem{corollary}{Corollary}
\newtheorem{definition}{Definition}
\newtheorem{assumption}{Assumption}

\title{Accelerated Methods with Compressed Communications for Distributed Optimization Problems under Data Similarity}
\author {
    Dmitry Bylinkin\textsuperscript{\rm 1,\rm 2},
    Aleksandr Beznosikov\textsuperscript{\rm 2, \rm 1, \rm 3, \rm 4}
}
\affiliations {
    \textsuperscript{\rm 1}Moscow Institute of Physics and Technology\\
    \textsuperscript{\rm 2}Ivannikov Institute for System Programming of the Russian Academy of Sciences\\
    \textsuperscript{\rm 3}Sber AI Lab\\
    \textsuperscript{\rm 4}Skoltech\\
    da.bylinkin@gmail.com, anbeznosikov@gmail.com
}

\newcommand*{\E}{\mathbb{E}}

\newcommand{\R}{\mathbb{R}^d}

\usepackage{graphicx} 
\usepackage{subcaption}

\begin{document}
\maketitle

\begin{abstract}
    In recent years, as data and problem sizes have increased, distributed learning has become an essential tool for training high-performance models. However, the communication bottleneck, especially for high-dimensional data, is a challenge. Several techniques have been developed to overcome this problem. These include communication compression and implementation of local steps, which work particularly well when there is similarity of local data samples. In this paper, we study the synergy of these approaches for efficient distributed optimization. We propose the first theoretically grounded accelerated algorithms utilizing unbiased and biased compression under data similarity, leveraging variance reduction and error feedback frameworks. In terms of communication time our theory gives $\tilde{\mathcal{O}} \left(  1+\left[ M^{-\nicefrac{1}{4}} + \omega^{-\nicefrac{1}{2}} \right]\sqrt{\nicefrac{\delta}{\mu}}  \right)$ complexity for unbiased compressors and $\tilde{\mathcal{O}}\left(1+\beta^{\nicefrac{1}{4}}\sqrt{\nicefrac{\delta}{\mu}}\right)$ for biased ones, where $M$ is the number of computational nodes, $\beta$ is the compression power, $\delta$ is the similarity measure and $\mu$ is the parameter of strong convexity of the objective. Our theoretical results are of record and confirmed by experiments on different average losses and datasets.
\end{abstract}

\section{Introduction}
Conventional machine/deep learning algorithms often struggle to handle the scale and complexity of modern datasets, resulting in long training times and limited scalability. Distributed optimization \citep{verbraeken2020survey} has witnessed remarkable progress, driven by rising demand for efficient methods across diverse applications such as medical image analysis \citep{intro_medical}, chemical physics \citep{intro_chem}, predictive maintenance \citep{intro_maintenance} and natural language processing \citep{intro_nlp}. Distributed learning addresses emerging challenges by spreading the training process across multiple nodes, allowing scientists and engineers to process and analyze data that would be impractical to handle on a single machine \citep{intro_distr_comparison}. In terms of optimization, we have the following problem statement:
\begin{align}\label{prob_form}
    \min_{x\in\R}&\left[f(x) = \frac{1}{M}\sum_{m=1}^Mf_m(x)\right]\\&\text{with}\quad f_m(x)=\frac{1}{n_m} \sum_{j=1}^{n_m} \ell(x, z_j^m)\notag,
\end{align}
where $M$ refers to the number of nodes/devices/clients/agents/machines, $n_m$ is the size of the local dataset on $m$-th machine, $x$ is the vector representation of the model using $d$ features, $z_j^m$ is the $j$-th data point on the $m$-th node and $\ell$ is the loss function. $z^m_j$ is an ordered pair of feature description $a^i_j$ and label $b^m_j$. We consider an architecture with a star-network topology, that is, a server, represented by $f_1$, plays a crucial role as the main computational hub. The other nodes act as users that can communicate with the server but not directly with each other. This hierarchical structure allows for easy coordination through the first node within the system. When training modern distributed models, the bottleneck is often the cost of communicating information \citep{konevcny2016federated, intro_bottle}. It is a known fact that client-to-server communication is much more resource-intensive than server-to-client \citep{diana, kairouz2021advances}. As a result, considerable effort has been devoted to development of distributed optimization methods that would be efficient in terms of nodes to server communication. Significant success has been achieved with the development of compression techniques, that help to reduce the amount of data transferred between nodes during training. In this paper, we deal with two main classes of compressors: unbiased and biased \citep{EF_proof_1}. 
\begin{definition}\label{unbiased_def}
    We call the mapping $Q:\R\rightarrow\R$ an unbiased compressor, if there exists a constant $\omega>1$ such that
    \begin{align*}
        \E_Q\left[Q(x)\right]=x,\quad \E_Q\left[\|Q(x)-x\|^2\right]\leq\omega\|x\|^2,\quad\forall x\in\R.
    \end{align*}
\end{definition}
\begin{definition}\label{biased_def}
    We call the mapping $C:\R\rightarrow\R$ a biased compressor, if there exists a constant $\beta>1$ such that
    \begin{align*}
        \E_C\left[\|C(x)-x\|^2\right]\leq\left(1-\frac{1}{\beta}\right)\|x\|^2,\quad\forall x\in\R.
    \end{align*}
\end{definition}
Let us denote by $\gamma_{\omega}$ the value showing how much the operator $Q(z)$ compresses the input vector on average. Let $b$ be the number of bits needed to represent a single float, and $\|\cdot\|_{bits}$ be the number of bits needed to represent the input vector. We define $\gamma_{\omega}^{-1}=\frac{1}{bd}\E\|Q(z)\|_{bits}$. Similarly, we introduce the compressive force $\gamma_{\beta}$ for $C(z)$. For practical compressors, we have $\gamma_{\omega}\geq\omega$ and $\gamma_{\beta}\geq\beta$ \citep{ef_empirical, EF_proof_1}. While methods with unbiased compressors may theoretically provide more optimistic estimates \citep{gorbunov2020unified, marina, adiana}, biased ones show a notable advantage in practical applications \citep{biased_superior}. Consequently, there is a growing interest in biased compression techniques. However, analyzing biased compressors is challenging, and there is limited understanding of their behavior \citep{EF_proof_1, ef21, EF_proof_2}.\\
Another technique that addresses the communication bottleneck in distributed optimization is utilizing local steps under the Hessian similarity condition \citep{Sim_ohad, intro_sim_decrease, aeg}. In this scenario, a single node represents the average nature of the data across all ones.

\begin{definition}\label{sim_def}
    We say that there is the Hessian similarity ($\delta$-relatedness) between $f_i$ and $f$, if there exists a constant $\delta>0$ such that
    \begin{align*}
        \lVert \nabla^2 f_i(x) - \nabla^2 f(x) \rVert \leq \delta, \quad \forall x\in\R.
    \end{align*}
\end{definition}

Consider $f$ to be $L$-smooth. As the size $n$ of data, distributed uniformly across the nodes, increases, the losses become more statistically similar. In the quadratic case, it is shown that $\delta\thicksim\nicefrac{L}{n}$. Otherwise, we have $\delta\thicksim\nicefrac{L}{\sqrt{n}}$ for any non-quadratic losses \citep{intro_sim_decrease}. As a result, the server gets to contact the clients less often to compute the full gradient. Despite the fact that research on distributed optimization via compression or similarity has been going on for quite some time, there are still open questions:
\begin{enumerate}
\item \textit{How to build accelerated methods that use both compression and similarity?}
\item \textit{Would such methods be superior to existing SOTAs in terms of communication efficiency?}
\end{enumerate}

\section{Notation}
When we talk about the communication efficiency of an algorithm, it is important to choose the right definition of its communication complexity. This can be done in several ways.
\begin{enumerate}
    \item \textit{Number of communication rounds (CC-1).} The number of times the server initiates communication with clients is used as a complexity measure. This does not take into account the number of involved machines. Even if the server has communicated with all the devices within a communication round, this counts the same as if it only has communicated with one.
    \item \textit{Number of client-server communications (CC-2).} It arises when we recognize that the number of rounds of communication is not sufficient to adequately compare distributed methods. For example, if the nodes operate asynchronously. In this case, the more appropriate metric is the total number of communications rather than the number of rounds. Using this approach, we begin to experience the superiority of methods that turn out to be bad in the sense of \textit{CC-1}.
    \item \textit{Communication time (CC-3).} We consider a synchronous setup. Let all the devices and their communication channels to the server be equivalent. Let transmission time of one unit of information from the devices to the server take $\tau$ time units. Also, if the clients start communicating with the server at same time, the channel between them is initialized in negligible time. Then, the total time of one such round of communication is $\tau K$, where $K$ is the amount of information transmitted from each machine to the server. This definition allows us to see the strengths of methods with compressed communications, since they can change $K$.
\end{enumerate}

\section{Related Works}

\subsection{Distributed Learning via Hessian Similarity}
Now that the communication complexity is suitably defined, let us move on to the survey on effective communication techniques. The first work around similarity was the Newton-type \texttt{DANE} method, developed for quadratic strongly convex functions \citep{Sim_ohad}. For this class of problems, a \textit{CC-1} lower bound was proved in \citep{sim_lower_bound}. \texttt{DANE} did not reach it. This raised the question of how to fill this gap. History of work on this problem counts many papers. Nevertheless, all of them either did not reach exactly the bound or considered special cases \citep{zhang2015disco, lu2018relatively, yuan2020convergence, beznosikov2021distributed, tian2022acceleration}. Recently, \texttt{Accelerated ExtraGradient} enjoying optimal \textit{CC-1} communication complexity under the Hessian similarity condition was constructed in \citep{aeg}. Exploiting the similarity is not the only approach to effective communication. There are also the compression techniques, to which we devote the next subsection.

\subsection{Distributed Learning via Compression}
There has been a significant amount of research done on communication compression. Especially on unbiased operators, due to their ease of analysis. The first family of compression schemes with convergence guarantees was proposed in one-device setup by \citet{QSGD}. It was suggested to perform gradient descent steps using compressed gradients. An extension of the proposed approach to multiple nodes was made a year later, when the first method \texttt{DQGD} for the distributed setup appeared \citep{DQGD}. The following was suggested: the node gets the current point, calculates the local gradient, then compresses and sends it to the server to perform the gradient step. This scheme is quite simple and easy to analyze, but it has a number of drawbacks, which is discussed below. A comprehensive review of unbiased compression can be found in \citep{ADIANA_strong}. As noted before, the nature of unbiased compression is quite simple and one can design methods simply by replacing the stochastic gradient with a compressed one. At the same time, biased compression shows better empirical results \citep{ef_empirical}, but its nature is non-trivial and requires a special approach. In 2014, a framework for constructing methods with biased compression was proposed \citep{EF_first}. The key idea is that each node remembers the "error" it made when compressing the local gradient, and then takes it into account in a special way in the next rounds. This approach is called Error Feedback. There were no theoretically proven results without unnatural assumptions until \citep{EF_proof_2, EF_proof_1}. In all mentioned papers both for unbiased and biased compression a similar problem arisen -- none of the named algorithms converged to the true optimum. The reason is uncontrolled variance of compressed gradient difference, which results in its approximation not tending to zero as the algorithm runs. After reaching some neighborhood of the solution, the method loses the ability to take small enough steps to avoid “overshooting” the optimum.
This could be solved by the variance reduction (VR) technique. 

\subsection{Variance Reduction}
Originally, variance reduction was designed to solve the convergence problem of \texttt{SGD} \citep{robbins1951stochastic}. Classical stochastic optimization methods such as \texttt{SGD} face the same problem as the simplest distributed methods with compression: approximation of the gradient does not tend to zero when searching for the optimum. Thus, there is convergence to its neighborhood only. The variance reduction framework allows to correct this drawback of naive methods. There are two main approaches: \texttt{SAGA} \citep{saga} and \texttt{SVRG} \citep{johnson2013accelerating}. The first stores the history of evoked gradients by mini-batches, resulting in a more accurate approximation of the full gradient. The second has no "memory", but occasionally recalculates the full gradient. The optimal stochastic optimization algorithm with variance reduction is \texttt{KATYUSHA} \citep{katyusha}. Its various modifications are of great interest \citep{kovalev2020don, katyushaX}. These ideas could be developed to solve problems arising in the analysis of methods utilizing unbiased compression. For example, see \texttt{DIANA} \citep{diana} and its accelerated version \texttt{ADIANA} \citep{adiana}. In the non-convex case, the current state-of-the-art \texttt{MARINA} is also based on the VR idea \citep{marina}. Variance reduction can be exploited in biased compression methods as well, see \texttt{EF21} \citep{ef21}, \texttt{ECLK} \citep{eclk}. It is also known how to construct variance reduction schemes for a more general class of problems than minimization -- variational inequalities (VIs) \citep{alacaoglu2022stochastic}. This approach  is useful for developing methods that combine compression with other approaches, such as local steps and similarity \citep{pillars, masha}. Since our goal is to utilize the synergy of multiple techniques to develop methods for effective communication, we devote the next subsection to a brief overview and comparison of the best existing methods.

\subsection{Effective Communications via Combination of Techniques}
One of the modern methods combining different techniques to communicate efficiently is \texttt{LoCoDL} \citep{condat2024locodl}. Using compression and local steps, it is possible to construct a method that, in the case of $\omega>M$, repeats the result of \texttt{ADIANA} in the sense of \textit{CC-1}, and outperforms it in the sense of \textit{CC-2} and \textit{CC-3}. The authors of \citep{condat2023tamuna} managed to add client sampling to this combination of techniques. Their \texttt{TAMUNA} gives stronger results in some special cases. The main weakness of the above methods is that they do not exploit data similarity. Thus, if the ratio $\nicefrac{\delta}{L}$ is small enough, these methods lose significantly to SOTAs that use similarity. \texttt{Accelerated Extragradient} \citep{aeg} is unbeatable in the sense of \textit{CC-1} and superior to \texttt{LoCoDL} and \texttt{TAMUNA} in the sense of \textit{CC-2} and \textit{CC-3}. These results are achieved by combining similarity and local steps. \texttt{AccSVRS} \citep{svrs}, which uses client sampling in addition to these techniques, loses in the sense of \textit{CC-1}. However, in the sense of \textit{CC-2}, this method is optimal. One can note the current best methods in terms of \textit{CC-3} are \texttt{Accelerated ExtraGradient} and \texttt{Three Pillars Algorithm}. The first one is accelerated, but does not use compression. The second one uses unbiased compression but does not use acceleration. Thus, which of these two algorithms performs better depends on the ratio of the similarity constant to the number of machines. 

\section{Our Contributions}
In this paper, we investigate whether it is possible to construct communication-efficient (in terms of \textit{CC-3}) algorithms for distributed optimization problems. By taking state-of-the-art techniques for handling similarity, compression and local steps into a single method, we bridge the existing gap and design a SOTA in the sense of \textit{CC-3}.

\begin{enumerate}
\item \textbf{Combination of compression and similarity.} It can be seen from Table \ref{tab:comparison0} that there are methods that combine similarity with either compression or acceleration -- but not both techniques at the same time. However, there was no success in combining all three approaches into a single algorithm. \textit{We propose \textbf{the} \textbf{first} accelerated method for \textbf{both} unbiased and biased compression under similarity condition}. 

\item \textbf{Best communication time}. To the best of our knowledge, in the case of $\delta\ll L$, depending on the number of machines, either \texttt{Accelerated ExtraGradient} or \texttt{Three Pillars Algorithm} \citep{pillars} gives the top \textit{CC-3} result. It can be seen from Table \ref{tab:comparison0} that our \texttt{OLGA} \textit{significantly \textbf{dominates} both methods}. 

\item \textbf{Numerical experiments}. The constructed theory is verified experimentally on different problems. Experiments confirm the superiority of our method and show its robustness to changes in the number of computational nodes and different values of the problem constants.
\end{enumerate}

\newcommand{\myred}[1]{{\color{red}#1}}
\definecolor{bgcolor}{rgb}{0.8,1,1}
\definecolor{bgcolor2}{rgb}{0.8,1,0.8}
\newcommand{\myblue}[1]{{\color{blue}#1}}
\begin{table*}[h]
    \centering
    \small
    \scriptsize
    \resizebox{\linewidth}{!}{
  \begin{threeparttable}
    \begin{tabular}{|c|c|c|c|c|c|}
    \cline{1-6}
     Method & \textbf{\quad\quad\quad\quad\quad\quad Approach \quad\quad\quad\quad\quad\quad} & \textbf{$\#$ of client-server communications rounds (\textit{CC-1})} & \textbf{$\#$ of client-server communications (\textit{CC-2})} & \textbf{Communication (\textit{CC-3})} & \textbf{Weaknesses} \\
    \hline
    \texttt{ADIANA} &
    \makecell{Unbiased compression\\Acceleration}
    & $\mathcal{O} \left( \left( \omega + \left(1 + M^{-\nicefrac{1}{2}}\omega\right)\sqrt{\frac{L}{\mu}} \right) \log \frac{1}{\varepsilon} \right)$ & $\mathcal{O} \left( \left( M\omega + \left(M + M^{\nicefrac{1}{2}}\omega\right)\sqrt{\frac{L}{\mu}} \right) \log \frac{1}{\varepsilon} \right)$ & $\mathcal{O} \left( \left( 1 + \left(\omega^{-1} + M^{-\nicefrac{1}{2}}\right)\sqrt{\frac{L}{\mu}} \right) \log \frac{1}{\varepsilon} \right)$ & \makecell{{Only unbiased compressor}\\{All nodes are involved in communication round}\\{Does not account for similarity}}
    \\ \cline{1-6} 
    \texttt{ECLK} & \makecell{Biased compression\\Acceleration} & $\mathcal{O} \left( \left( \beta + \beta^{\nicefrac{3}{2}}\sqrt{\frac{L}{\mu}} \right) \log \frac{1}{\varepsilon} \right)$ & $\mathcal{O} \left( \left( M\beta + M\beta^{\nicefrac{3}{2}}\sqrt{\frac{L}{\mu}} \right) \log \frac{1}{\varepsilon} \right)$ & $\mathcal{O} \left( \left( 1 + \beta^{\nicefrac{1}{2}}\sqrt{\frac{L}{\mu}} \right) \log \frac{1}{\varepsilon} \right)$ & \makecell{{Bad constants in estimation}\\{All nodes are involved in communication round}\\{Does not account for similarity}}
    \\ \cline{1-6}
    \texttt{LoCoDL} & \makecell{Unbiased compression\\Local steps\\Acceleration} & $\mathcal{O}\left(\left( \omega + (1+\omega^{\nicefrac{1}{2}})\sqrt{\frac{L}{\mu}} \right)\log\frac{1}{\varepsilon}\right)$
    & $\mathcal{O}\left(\left( M\omega + (M+M\omega^{\nicefrac{1}{2}})\sqrt{\frac{L}{\mu}} \right)\log\frac{1}{\varepsilon}\right)$ & $\mathcal{O}\left(\left( 1 + (\omega^{-1}+\omega^{-\nicefrac{1}{2}})\sqrt{\frac{L}{\mu}} \right)\log\frac{1}{\varepsilon}\right)$ & \makecell{Does not account for similarity\\All nodes are involved in communication round}
    \\ \cline{1-6}
    \texttt{TAMUNA} & \makecell{Unbiased compression\\Client sampling\\Local steps} & $\mathcal{O}\left(\left( 
M+\sqrt{\frac{L}{\mu}} \right)\log\frac{1}{\varepsilon}\right)$ & $\mathcal{O}\left(\left( 
M+\sqrt{\frac{L}{\mu}} \right)\log\frac{1}{\varepsilon}\right)$ & $\mathcal{O}\left(\left( 
M+\sqrt{\frac{L}{\mu}} \right)\log\frac{1}{\varepsilon}\right)$ & \makecell{Does not account for similarity}\\ \cline{1-6}
    \texttt{AccExtraGradient} & \makecell{Similarity\\Local steps\\Acceleration} & $\mathcal{O}\left( \sqrt{\frac{\delta}{\mu}}\log\frac{1}{\varepsilon} \right)$ & $\mathcal{O}\left( M\sqrt{\frac{\delta}{\mu}}\log\frac{1}{\varepsilon} \right)$ & $\mathcal{O}\left( \sqrt{\frac{\delta}{\mu}}\log\frac{1}{\varepsilon} \right)$ & \makecell{All nodes are involved in communication round\\Communication is inefficient} \\
    \cline{1-6}
    \texttt{AccSVRS} & \makecell{Similarity\\Client sampling\\Local steps\\Acceleration} & $\mathcal{O}\left(\left( M + M^{\nicefrac{3}{4}}\sqrt{\frac{\delta}{\mu}} \right)\log\frac{1}{\varepsilon}\right)$ & $\mathcal{O}\left(\left( M + M^{\nicefrac{3}{4}}\sqrt{\frac{\delta}{\mu}} \right)\log\frac{1}{\varepsilon}\right)$ & $\mathcal{O}\left(\left( M + M^{\nicefrac{3}{4}}\sqrt{\frac{\delta}{\mu}} \right)\log\frac{1}{\varepsilon}\right)$ & Communication is inefficient
    \\
    \cline{1-6}
    \texttt{Optimistic MASHA} 
    & \makecell{Similarity\\Unbiased compression}
    & $\mathcal{O} \left( \left( 
M+\frac{L}{\mu}+M^{\nicefrac{1}{2}}\frac{\delta}{\mu} \right)\log\frac{1}{\varepsilon}\right)$ & $\mathcal{O} \left( \left(
M^2+M\frac{L}{\mu}+M^{\nicefrac{3}{2}}\frac{\delta}{\mu} \right)\log\frac{1}{\varepsilon}\right)$ & $\mathcal{O} \left( 1 + \left(
M^{-1}\frac{L}{\mu}+M^{-\nicefrac{1}{2}}\frac{\delta}{\mu} \right)\log\frac{1}{\varepsilon}\right)$ & \makecell{{All nodes are involved in communication round}\\{Lipschitz constant is included in the estimation} \\ {No acceleration}\\{Only permutation compressor}}
    \\ \cline{1-6} \texttt{Three Pillars Algorithm}& \makecell{Unbiased compression\\Similarity\\Local steps}  & $\mathcal{O} \left(\left(M+M^{\nicefrac{1}{2}}\frac{\delta}{\mu}\right)\log\frac{1}{\varepsilon}\right)$ & $\mathcal{O} \left(\left(M^2+M^{\nicefrac{3}{2}}\frac{\delta}{\mu}\right)\log\frac{1}{\varepsilon}\right)$ & $\mathcal{O} \left(\left(1+M^{-\nicefrac{1}{2}}\frac{\delta}{\mu}\right)\log\frac{1}{\varepsilon}\right)$ & \makecell{{All nodes are involved in communication round} \\ {No acceleration}\\{Only permutation compressor}}
    \\ \cline{1-6}
    \cellcolor{bgcolor2}{\texttt{OLGA}} & \cellcolor{bgcolor2}\colorbox{bgcolor2}{\makecell{Similarity\\Unbiased compression\\Local steps\\Acceleration}}  & \cellcolor{bgcolor2}{$\mathcal{O} \left( \left( \omega+\left[ \omega M^{-\nicefrac{1}{4}} + \omega^{\nicefrac{1}{2}} \right]\sqrt{\frac{\delta}{\mu}} \right) \log \frac{1}{\varepsilon} \right)$} & \cellcolor{bgcolor2}{$\mathcal{O} \left( \left( M\omega+\left[ \omega M^{\nicefrac{3}{4}} + M\omega^{\nicefrac{1}{2}} \right]\sqrt{\frac{\delta}{\mu}} \right) \log \frac{1}{\varepsilon} \right)$} & \cellcolor{bgcolor2}{$\mathcal{O} \left( \left( 1+\left[ M^{-\nicefrac{1}{4}} + \omega^{-\nicefrac{1}{2}} \right]\sqrt{\frac{\delta}{\mu}} \right) \log \frac{1}{\varepsilon} \right)$} & \cellcolor{bgcolor2}\colorbox{bgcolor2}{\makecell{Only unbiased compressor\\All nodes are involved in communication round}}
    \\ \cline{1-6}
    \cellcolor{bgcolor2}{\texttt{EF-OLGA}} & \cellcolor{bgcolor2}\colorbox{bgcolor2}{\makecell{Similarity\\Biased compression\\Local steps\\Acceleration}}  & \cellcolor{bgcolor2}{$\mathcal{O} \left( \left( \beta + \beta^{\nicefrac{5}{4}}\sqrt{\frac{\delta}{\mu}} \right) \log \frac{1}{\varepsilon} \right)$} & \cellcolor{bgcolor2}{$\mathcal{O} \left( \left( M\beta + M\beta^{\nicefrac{5}{4}}\sqrt{\frac{\delta}{\mu}} \right) \log \frac{1}{\varepsilon} \right)$} & \cellcolor{bgcolor2}{$\mathcal{O} \left( \left( 1 + \beta^{\nicefrac{1}{4}}\sqrt{\frac{\delta}{\mu}} \right) \log \frac{1}{\varepsilon} \right)$} & \cellcolor{bgcolor2}{All nodes are involved in communication round}
    \\\hline
    \end{tabular}    
    \end{threeparttable}
    }
    \captionof{table}{Summary of different communication complexities of SOTAs for distributed optimization.\\\textit{Notation:} $\omega,\beta=$ compression constants, $M=$ number of computational nodes, $\delta=$ similarity (relatedness) constant, $\mu=$ constant of strong convexity of the objective $f$, $L=$ Lipschitz constant of the gradient of $f$.}
    \label{tab:comparison0} 
\end{table*}

\section{Problem Formulation and Assumptions}
Real problems arising in practice are known to behave better than in theory. Therefore, constructing methods for idealized setups can be useful \citep{woodworth2023two}. Our work relies on the strong convexity of the mean risk and the homogeneity of the data on all nodes. 
\begin{assumption}\label{assumption}
    Every $f_m$ is $\delta$-related to $f$ (Definition \ref{sim_def}) and $f\colon \mathbb{R}^d \rightarrow \mathbb{R}$ is $\mu$-strongly convex on $\mathbb{R}^d$:
    \begin{align}\label{strong_convexity}
        f(x)\geq f(y)+\langle\nabla f(y),x-y\rangle + \frac{\mu}{2}&\|x-y\|^2,\\ &\forall x,y\in\R.\notag
    \end{align}
\end{assumption}
Note that Assumption \ref{assumption} allows local functions to be non-convex. The only requirement imposed on them is the Hessian similarity. We consider the distributed optimization problem of the form (\ref{prob_form}) , where the data is drawn from a single distribution. It is also worth noting that Definition \ref{sim_def} implies $\delta$-smoothness of  $f_m-f$ for all $m\in[1,M]$:
\begin{align}\label{sim_cor}
    \|\nabla(f_m-f)(x)-\nabla(f_m-f)(y)\|^2\leq\delta^2&\|x-y\|^2,\\&\forall x,y\in\R\notag.
\end{align}
In fact, the Hessian similarity assumption does not dramatically reduce the generality of our analysis. Indeed, if the data is heterogeneous, it is sufficient to put $\delta=L$ in our results.

\section{Unbiased Compression via \texttt{OLGA}}
As mentioned above, compression can be implemented through variance reduction \citep{eclk, marina, masha}. The approach proposed in \citep{svrs} provides a way to construct an optimal method for the variance reduction technique under the similarity condition. The idea is to try to naturally generalize this approach to schemes with compression. We first consider an unbiased compressor (Definition \ref{unbiased_def}).
\begin{algorithm}[H]
	\caption{}
	\begin{algorithmic}[1]
		\State {\bf Input:} $x_0 \in \R, p\in(0,1), \theta>0$
            \State Set $N\sim$ Geom$(p)$\label{unb:sample_N}
            \State Send $x_0$ and $\nabla f_1(x_0)$ to each device 
            \State Collect $\nabla f(x_0)=\frac{1}{M}\sum_{m=1}^M\nabla f_m(x_0)$ on server\label{unb_epoch:grad_eval}
            \For{$k=0,1,2,\ldots, N-1$}:
                \For{each device $m$ in parallel}:\label{unbiased:line_6}
                    \State Calculate $\hat{g}_k^m$ using formula:
                    \begin{align*}
                        \hat{g}_k^m=\nabla f_m(x_k)-\nabla f_1(x_k) - \nabla f_m(x_0) + \nabla f_1(x_0)
                    \end{align*}\label{unbiased:line_7}
                    \State 
                        Send $g_k^m=Q\left(\hat{g}_k^m\right)$ to server\label{unbiased:line_8}
                \EndFor
                \State Collect $g_k = \frac{1}{M}\sum_{m=1}^Mg_k^m$ on server
                \State Calculate $t_k=g_k - \nabla f_1(x_0) + \nabla f(x_0)$\label{unbiased:line_11}
                \State Update $x_{k+1} = \arg\min_{x\in\R}q(x),$\label{unbiased_epoch:task} where
                \begin{align*}q(x)= \langle t_k,x-x_k \rangle + \frac{1}{2\theta}\|x-x_k\|^2 + f_1(x)\end{align*}
                    \State Send $x_{k+1}$ and $\nabla f_1(x_{k+1})$ to each device
            \EndFor
            \State{\bf Output:} $x_N$
	\end{algorithmic}
        \label{alg:unbiased_compr_epoch}
\end{algorithm}
In Algorithm \ref{alg:unbiased_compr_epoch}, the full gradient is called only once per iteration before entering the loop (Line \ref{unb_epoch:grad_eval}). At each iteration, it is required to communicate with each node and solve the subproblem (Line \ref{unbiased_epoch:task} of Algorithm \ref{alg:unbiased_compr_epoch}). The number of iterations of Algorithm \ref{alg:unbiased_compr_epoch} is a random variable (see Line \ref{unb:sample_N}) that is not bounded from above, and hence the effectiveness of communication is important. Algorithm \ref{alg:unbiased_compr_epoch} has no acceleration. We build it to use as an integral part of more efficient Algorithm \ref{alg:unbiased_compr}. From the point of view of theory, we are only interested in the descent lemma for Algorithm \ref{alg:unbiased_compr_epoch}.
\begin{lemma}\label{lem:4}
    Consider an epoch of Algorithm \ref{alg:unbiased_compr_epoch}. Let $h(x)=f_1(x)-f(x)+\frac{1}{2\theta}\|x\|^2$, where $\theta\leq\min\left\{\frac{\sqrt{p}\sqrt{M}}{8\delta\sqrt{\omega}}, \frac{1}{2\delta}\right\}$. Then the following inequality holds for every $x\in\R$:
    \begin{align}\label{lem:unbiased_estimation}
        \begin{split}\E \left[f(x_N)-f(x)\right] \leq& \E \Big[\langle x-x_0, \nabla h(x_N)-\nabla h(x_0) \rangle\\ &- \frac{p}{2}D_{h}(x_0,x_N)-\frac{\mu}{2}\|x_N-x\|^2\Big].\end{split}
    \end{align}
\end{lemma}
\begin{algorithm}[H]
	\caption{\texttt{OLGA}}
	\begin{algorithmic}[1]
		\State {\bf Input:} $z_0=y_0 \in \R, p\in(0,1), \theta>0, \tau\in(0,1), \alpha>0$
            \For{$k=0,1,2,\ldots, K-1$}:
                \State Update $x_{k+1}=\tau z_k + (1-\tau)y_k$\label{unbiased:3}
                \State Update $y_{k+1} =\text{Alg.1}(x_{k+1},\theta,p)$
                \State Send $y_{k+1}$ to each device
                \State Collect $\nabla f(y_{k+1})=\frac{1} {M}\sum_{m=1}^M\nabla f_m(y_{k+1})$ on server
                \State Calculate $t_k=\nabla(f_1-f)(x_{k+1}) - \nabla(f_1-f)(y_{k+1})$\label{unbiased_full:line_6}
                \State Calculate $G_{k+1}=p\left(t_k + \frac{x_{k+1}-y_{k+1}}{\theta}\right)$\label{unbiased_G}
                \State Update $z_{k+1}= \arg\min_{z\in\R}q(z)$, where  \label{unbiased:12}
                \begin{align*}q(z)=\frac{1}{2\alpha}\|z-z_k\|^2 + \langle G_{k+1},z\rangle +\frac{\mu}{2}\|z-y_{k+1}\|^2\end{align*}
            \EndFor
            \State{\bf Output:} $y_K$
	\end{algorithmic}
        \label{alg:unbiased_compr}
\end{algorithm}

Next, we employ interpolation framework motivated by \texttt{KatyushaX} \citep{katyushaX} to obtain the final version of proposed algorithm. The key difference between our approach and \texttt{KatyushaX} is the choice of a suitable Bregman divergence as a metric function instead of the Euclidean distance used in the mentioned method.

Note that Line \ref{unbiased:12} of Algorithm \ref{alg:unbiased_compr} can be solved analytically and does not require expensive computations. The full gradient is computed twice per iteration, whereas without compression it could be invoked potentially infinitely often, resulting in significant communication costs. Algorithm \ref{alg:unbiased_compr} calls the full gradient a second time on the outer iteration. \texttt{OLGA} calls the full gradient a constant rather than potentially infinite number of times. For the sake of brevity of description, we introduce two potential functions:
\begin{gather*}
    Y_k = \frac{\alpha}{\tau}[f(y_k)-f(x_*)],\quad
    Z_k = \frac{1+\mu\alpha}{2}\|z_k-x_*\|^2,
\end{gather*}
where $x_*$ is the solution of the problem (\ref{prob_form}). 

\begin{theorem}\label{th:unbiased}
    Let the problem (\ref{prob_form}) be solved by Algorithm \ref{alg:unbiased_compr} with $\theta\leq\min\left\{\frac{\sqrt{p}\sqrt{M}}{8\delta\sqrt{\omega}}, \frac{1}{2\delta}\right\}$ and tuning parameters such that $4\alpha p\tau \leq \theta$. Then the following inequality holds:
    \begin{align*}
        \E\left[ Y_{k+1} + Z_{k+1} \right] \leq& \E\left[ (1-\tau)Y_k + (1+\mu\alpha)^{-1}Z_k  \right].
    \end{align*}
\end{theorem}
As discussed above, the iteration of Algorithm \ref{alg:unbiased_compr} invokes the full gradient twice and the compressed gradient another $N$ times. $N$ is a random variable depending on the parameter $p$. On average, we have $\mathcal{O}\left(\nicefrac{1}{\gamma_{\omega}}+p\right)$ \textit{CC-3} for a single iteration of \texttt{OLGA}. It is obvious that one should choose $p=\nicefrac{1}{\gamma_{\omega}}$. It has been discussed above that for practical compressors $\gamma_{\omega}\geq\omega$ holds. Let us select $\alpha,\tau$ values more carefully and formulate the following corollary.

\begin{corollary}\label{cor_unbiased}
    Let the problem (\ref{prob_form}) be solved by Algorithm \ref{alg:unbiased_compr}. Choose $$p=\frac{1}{\gamma_{\omega}},\quad\theta\leq\min\left\{\frac{\sqrt{p}\sqrt{M}}{8\delta\sqrt{\omega}}, \frac{1}{2\delta}\right\},$$ $$\tau=\min\left\{ \frac{\sqrt{\mu}\theta^{\nicefrac{1}{2}}p^{-\nicefrac{1}{2}}}{4},\frac{1}{4} \right\},\quad\alpha=\frac{\theta p^{-1}}{8\tau},$$ then Algorithm \ref{alg:unbiased_compr} has
    $$\tilde{\mathcal{O}}\left( \gamma_{\omega}+\sqrt{\frac{\delta}{\mu}}\left[ \gamma_{\omega}M^{-\nicefrac{1}{4}} + \gamma_{\omega}^{\nicefrac{1}{2}} \right] \right) \textit{ CC-1},$$
    $$\tilde{\mathcal{O}}\left( M\gamma_{\omega}+\sqrt{\frac{\delta}{\mu}}\left[ \gamma_{\omega}M^{\nicefrac{3}{4}} + M\gamma_{\omega}^{\nicefrac{1}{2}} \right] \right) \textit{ CC-2},$$
    and
    $$\tilde{\mathcal{O}}\left( 1+\sqrt{\frac{\delta}{\mu}}\left[ M^{-\nicefrac{1}{4}} + \gamma_{\omega}^{-\nicefrac{1}{2}} \right] \right) \textit{ CC-3}.$$
\end{corollary}

It is worth noting that in \textit{CC-3} if $\gamma_{\omega}$ is too large, the term with $M$ dominates and the communication time result stops improving. If $\gamma_{\omega}$ is too small, the compression effect is not as strong as it could be. Therefore, the best effect is given by $\gamma_{\omega}=\Theta\left(\sqrt{M}\right)$. To simplify the appearance of the obtained results, the complexities are shown in Table \ref{tab:comparison0} under consideration that $\gamma_{\omega}\sim\omega$. This holds for a number of common used compressors \citep{EF_proof_1, alistarh2018convergence}.

\subsection{Discussion}
Let us compare Algorithm \ref{alg:unbiased_compr} with distributed optimization SOTAs. It makes sense to make comparisons only with methods that utilize similarity, because superiority over other ones depends mainly on how much $\delta$ is less than $L$. The optimal choice $\gamma_{\omega}=\Theta\left(\sqrt{M}\right)$ is assumed below.
\begin{enumerate}
    \item \textbf{\texttt{Accelerated ExtraGradient}} outperforms our method in the sense of \textit{CC-1} and \textit{CC-2}. However, \texttt{OLGA} has $\tilde{\mathcal{O}} \left( 1+ M^{-\nicefrac{1}{4}}\sqrt{\nicefrac{\delta}{\mu}}\right)$ \textit{CC-3} vs. $\tilde{\mathcal{O}}\left( \sqrt{\nicefrac{\delta}{\mu}}\right)$ for its competitor and hence turns out to be better by a factor $M^{\nicefrac{1}{4}}$. Thus, the difference in the running time of methods with a substantially large number of machines is enormous.
    \item \textbf{\texttt{AccSVRS}} loses to our method in terms of \textit{CC-1}: $\tilde{\mathcal{O}} \left( M^{\nicefrac{1}{2}}+M^{\nicefrac{1}{4}}\sqrt{\nicefrac{\delta}{\mu}} \right)$ vs. $\tilde{\mathcal{O}}\left( M + M^{\nicefrac{3}{4}}\sqrt{\nicefrac{\delta}{\mu}} \right)$; and \textit{CC-3}: $\tilde{\mathcal{O}} \left( 1+M^{-\nicefrac{1}{4}}\sqrt{\nicefrac{\delta}{\mu}} \right)$ vs. $\tilde{\mathcal{O}}\left( M + M^{\nicefrac{3}{4}}\sqrt{\nicefrac{\delta}{\mu}} \right)$; but wins by \textit{CC-2} due to client sampling.
    \item \textbf{\texttt{Three Pillars Algorithm}} loses to our method by \textit{CC-1}: $\tilde{\mathcal{O}} \left( M^{\nicefrac{1}{2}}+M^{\nicefrac{1}{4}}\sqrt{\nicefrac{\delta}{\mu}} \right)$ vs. $\mathcal{O} \left(M+M^{\nicefrac{1}{2}}\cdot\nicefrac{\delta}{\mu}\right)$; and \textit{CC-2}: $\tilde{\mathcal{O}} \left( M^{\nicefrac{3}{2}}+M^{\nicefrac{5}{4}}\sqrt{\nicefrac{\delta}{\mu}} \right)$ vs. $\mathcal{O} \left(M^2+M^{\nicefrac{3}{2}}\cdot\nicefrac{\delta}{\mu}\right)$. In terms of \textit{CC-3}, it has a better $M^{-\nicefrac{1}{2}}$ factor. However, \texttt{Three Pillars Algorithm} is not accelerated method. Thus, it loses to \texttt{OLGA} in a wide range of practical problems.
\end{enumerate}

\section{Biased Compression via \texttt{EF-OLGA}}
In this section, we consider a biased compressor (Definition \ref{biased_def}). As noted above, biased compressors are difficult to analyze and hence compressing gradient differences without introducing additional sequences will not yield results. Here we have to add "error" terms $e^m_k$. 
\begin{algorithm}[H]
	\caption{}
	\begin{algorithmic}[1]
		\State {\bf Input:} $x_0 \in \R, p\in(0,1), \theta>0, e_0^m=0$
            \State Set $N\in$ Geom$(p)$
            \State Send $x_0$ and $\nabla f_1(x_0)$ to each device 
            \State Collect $\nabla f(x_0)=\frac{1}{M}\sum_{m=1}^M\nabla f_m(x_0)$ on server
            \For{$k=0,1,2,\ldots, N-1$}:
                \For{each device $m$ in parallel}:
                    \State Calculate $\hat{g}_k^m$ using formula: \begin{align*}\hat{g}_k^m = \nabla f_m(x_k)-\nabla f_1(x_k) - \nabla f_m(x_0) + \nabla f_1(x_0)\end{align*}
                    \State 
                        Send $g_k^m=C\{e_k^m + \theta\hat{g}_k^m\}$ to server\label{epoch:stoch_grad}
                        \State Update $e_{k+1}^m = e_k^m - g_k^m + \theta\hat{g}_k^m$\label{epoch:error}
                        \If{$k=N-1$}:
                            \State Send $e_{k+1}^m=e_N^m$ to server
                        \EndIf
                \EndFor
                \State Collect $g_k = \frac{1}{M}\sum_{m=1}^Mg_k^m$ on server
                \State Calculate $t_k=\frac{1}{\theta}g_k - \nabla f_1(x_0) + \nabla f(x_0)$
                \State
                    Update $x_{k+1} = \arg\min_{x\in\R}q(x)$, where \begin{align*}
                        q(x)=\langle t_k,x-x_k \rangle + \frac{1}{2\theta}\|x-x_k\|^2 + f_1(x)
                    \end{align*}\label{epoch:task}
                    \State Send $x_{k+1}$ and $\nabla f_1(x_{k+1})$ to each device
            \EndFor
            \State{\bf Output:} $x_N$, $\frac{1}{M}\sum_{m=1}^Me^m_N$
	\end{algorithmic}
        \label{alg:biased_compr_epoch}
\end{algorithm}
Algorithm \ref{alg:biased_compr_epoch} is obtained by combining Algorithm \ref{alg:unbiased_compr_epoch} and error feedback framework \citep{EF_first}. It is proposed to introduce an additional sequence $e_k^m$ at each node $m$ that will “remember” how much the gradient sent to the server differs from the true gradient computed on the machine (Line \ref{epoch:error}). Unlike the analysis in the previous section, it is not possible to introduce the appropriate metric immediately. After extensive analysis of virtual sequences, the Bregman divergence of the smooth function on "real" arguments and the Euclidean distance on "virtual" ones appear independently. This requires some manipulation to analyze correctly.
\begin{algorithm}[H]
	\caption{\texttt{EF-OLGA}}
	\begin{algorithmic}[1]
		\State {\bf Input:} $z_0=y_0 \in \R, p\in(0,1), \theta>0, \tau\in(0,1), \alpha>0$ and $e_m^0 \in \R$ for every $m\in[1,M]$
            \For{$k=0,1,2,\ldots, K-1$}:
                \State Update $x_{k+1}=\tau z_k + (1-\tau)y_k$\label{biased:3}
                \State Update $y_{k+1}, e_{k+1} =\text{Alg.3}(x_{k+1},\theta,p)$
                \State Send $y_{k+1}$ to each device
                \State Collect $\nabla f(y_{k+1})=\frac{1}{M}\sum_{m=1}^M\nabla f(y_{k+1})$ on server
                \State Calculate $t_k=\nabla(f_1-f)(x_{k+1}) - \nabla(f_1-f)(y_{k+1})$
                \State Calculate $G_{k+1}=p\left(t_k + \frac{x_{k+1}-\tilde{y}_{k+1}}{\theta}\right)$, where \begin{align*}
                    \tilde{y}_{k+1}=y_{k+1}-e_{k+1}
                \end{align*}\label{biased:10}
                \State Update $z_{k+1} = \arg\min_{z\in\R}q(z)$, where \begin{align*}
                    q(z)=\frac{1}{2\alpha}\|z-z_k\|^2 + \langle G_{k+1},z\rangle +\frac{\mu}{2}\|z-y_{k+1}\|^2
                \end{align*} \label{biased:12}
            \EndFor
            \State{\bf Output:} $y_K$
	\end{algorithmic}
        \label{alg:biased_compr}
\end{algorithm}

Note that Line \ref{biased:10} of Algorithm \ref{alg:biased_compr} utilizes "error" terms from the last iteration of Algorithm \ref{alg:biased_compr_epoch}. Without such a modification, variance reduction can not be performed. Thus, in the biased case, a stronger connection between the inner and outer algorithms is formed. As in the case of Algorithm \ref{alg:unbiased_compr}, we also do not sample the stochastic gradient at the outer iteration because it does not affect the result and complicates the analysis. Therefore, \texttt{EF-OLGA} calls the full gradient twice per iteration. 

\begin{theorem}\label{th:biased}
    Let the problem (\ref{prob_form}) be solved by Algorithm \ref{alg:biased_compr} with $\theta\leq\frac{p^{\nicefrac{3}{2}}}{24\delta}\leq\frac{1}{6\delta}$ and tuning parameters such that $28\alpha p \tau \leq\theta$. Then the following inequality holds:
    \begin{align*}
        \E\left[ Y_{k+1} + Z_{k+1} \right] \leq& \E\left[ (1-\tau)Y_k + (1+\mu\alpha)^{-1}Z_k  \right].
    \end{align*}
\end{theorem}
Again, the \textit{CC-3} of the iteration is $\mathcal{O}\left(\frac{1}{\gamma_{\beta}}+p\right)$. Thus, $p=\nicefrac{1}{\gamma_{\beta}}$.

\begin{corollary}
    Let the problem (\ref{prob_form}) be solved by Algorithm \ref{alg:biased_compr}. Choose
    $$p=\frac{1}{\gamma_{\beta}}, \quad \tau=\min\left\{\frac{\theta^{\nicefrac{1}{2}}p^{-\nicefrac{1}{2}}}{18}, \frac{1}{18}\right\},$$
    $$\theta\leq\frac{p^{\nicefrac{3}{2}}}{12\delta},\quad \alpha=\frac{\theta p^{-1}}{36\tau},$$
    then Algorithm \ref{alg:biased_compr} has
    $$\tilde{\mathcal{O}}\left( \gamma_{\beta}+\gamma_{\beta}^{\nicefrac{5}{4}}\sqrt{\frac{\delta}{\mu}}\right) \textit{ CC-1}, \quad  \tilde{\mathcal{O}}\left( M\gamma_{\beta}+M\gamma_{\beta}^{\nicefrac{5}{4}}\sqrt{\frac{\delta}{\mu}} \right) \textit{ CC-2},$$
    and
    $$\tilde{\mathcal{O}}\left( 1+\gamma_{\beta}^{\nicefrac{1}{4}}\sqrt{\frac{\delta}{\mu}}\right) \textit{ CC-3}.$$
\end{corollary}
This corollary repeats the proof of Corollary \ref{cor_unbiased} with other constants. 

\section{Numerical Experiments}
Our theoretical findings are confirmed numerically on various tasks. 
In particular, we consider the ridge regression problem \citep{mcdonald2009ridge}: 
\begin{equation} 
    \label{eq:quadr}
    f(x) = \frac{1}{M}\sum_{m=1}^M \frac{1}{n}\sum_{j=1}^n \left(\langle x, a^m_j\rangle -b^m_j\right)^2+ \lambda\|x\|^2,
\end{equation}
and the logistic regression problem:
\begin{equation}
    \label{eq:logloss}
    f(x) = \frac{1}{M}\sum_{m=1}^M \frac{1}{n}\sum_{j=1}^n \ln\left(1+e^{-b^m_j \left\langle x, a^m_j \right\rangle}\right) + \lambda\|x\|^2.
\end{equation}
We set the penalty parameter $\lambda$ to $\nicefrac{L}{100}$, where $L$ is a Lipschitz constant of the gradient of the main objective. Also we consider different datasets from LibSVM \citep{chang2011libsvm}: \texttt{a9a} and \texttt{mushrooms} (Appendix). Since we consider illustrative experiments with linear models, it is not difficult to calculate $\mu$, $L$, $\delta$.  Therefore, the parameters of algorithms are chosen in the same way as in the theory without any tuning. We also vary the number of workers $M$. As competitors we take state-of-the-art methods from Table \ref{tab:comparison0}: \texttt{AccSVRS}, \texttt{Accelerated ExtraGradient}, \texttt{ADIANA}, \texttt{LoCoDL}. For algorithms with compression, we use a random sparsification operators Rand$K$, where we vary the number of coordinates $K$ (and hence $\omega$ since $\omega = \nicefrac{d}{K}$).

See the experiments with \texttt{OLGA} on \texttt{a9a} dataset on Figures \ref{fig:1} -- \ref{fig:6}. For additional experiments, see Appendix \ref{app:add_exp}. 


\begin{figure}[h!] 
   \begin{subfigure}{0.155\textwidth}
       \centering
       \includegraphics[width=\linewidth]{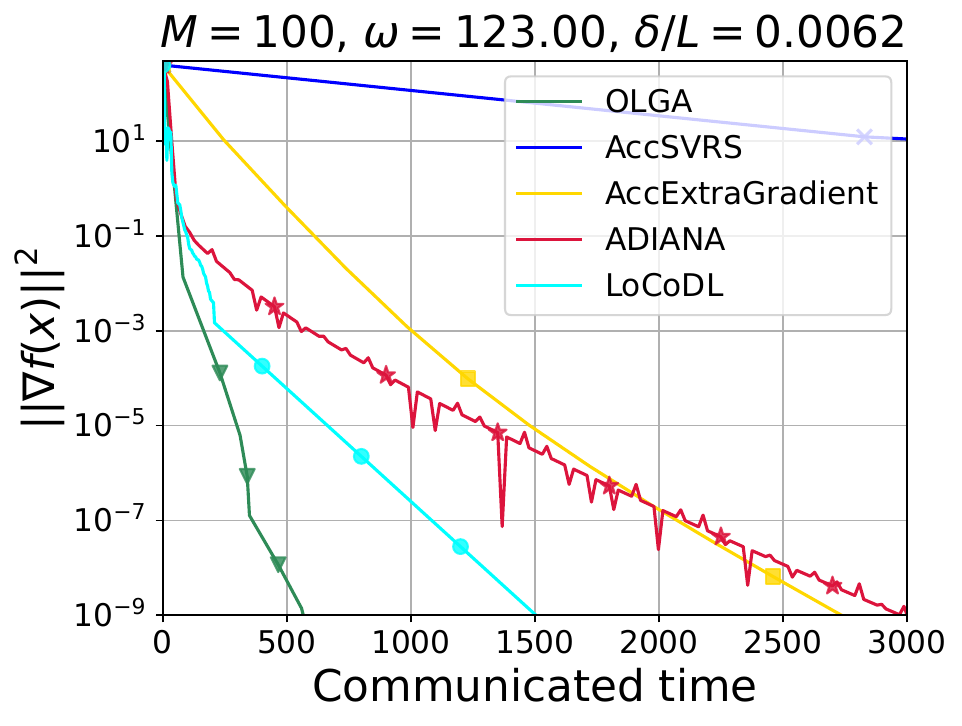}
   \end{subfigure}
   \begin{subfigure}{0.155\textwidth}
        \centering
       \includegraphics[width=\linewidth]{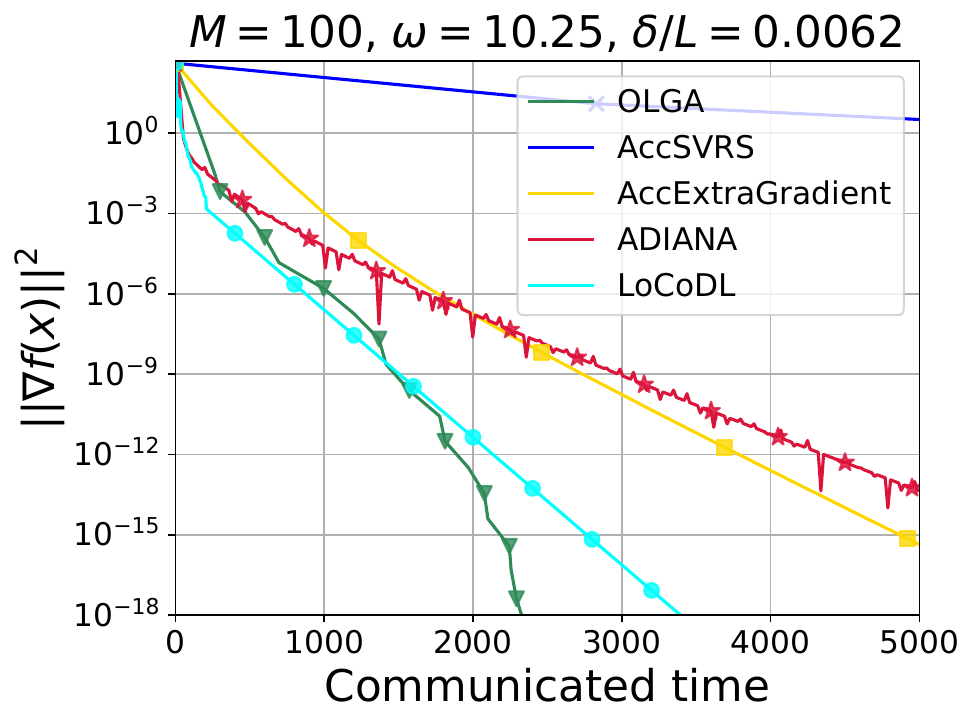}
   \end{subfigure}
   \begin{subfigure}{0.155\textwidth}
   \centering
       \includegraphics[width=\linewidth]{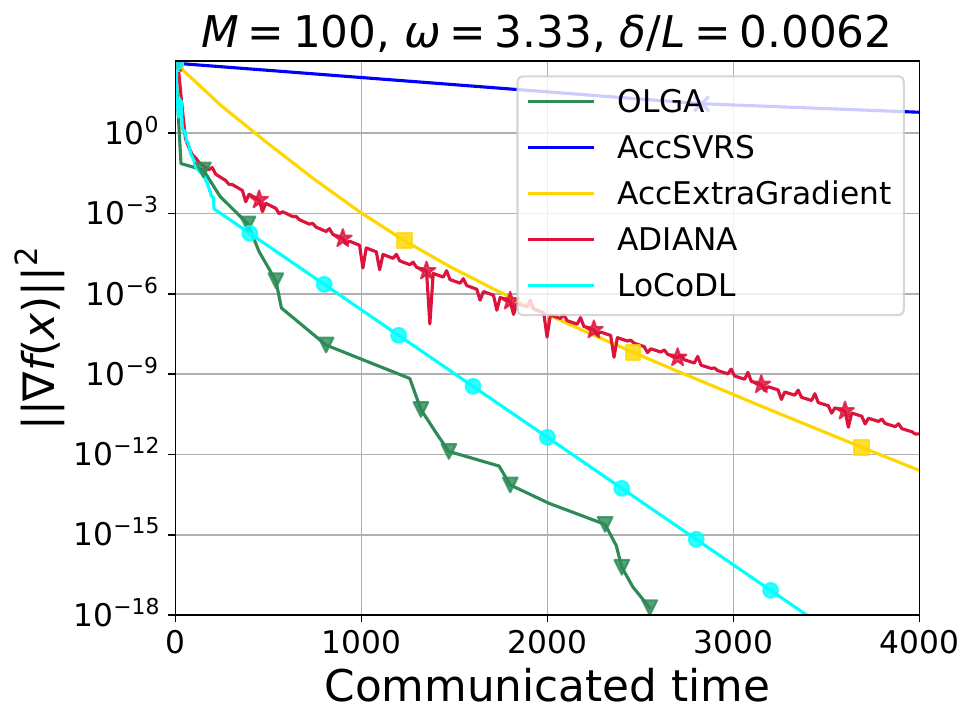}
   \end{subfigure}
   \begin{subfigure}{0.155\textwidth}
       \centering
       (a) $\omega = 123.00$
   \end{subfigure}
   \begin{subfigure}{0.155\textwidth}
       \centering
       (b) $\omega = 10.25$
   \end{subfigure}
   \begin{subfigure}{0.155\textwidth}
       \centering
       (c) $\omega = 3.33$
   \end{subfigure}
   \caption{Comparison of state-of-the-art distributed methods. The comparison is made on \eqref{eq:quadr} with $M=100$ and \texttt{a9a} dataset. The criterion is the communication time (\textit{CC-3}). For methods with compression we vary the power of compression $\omega$.}
   \label{fig:1}
\end{figure}
\begin{figure}[h!] 
   \begin{subfigure}{0.155\textwidth}
       \includegraphics[width=\linewidth]{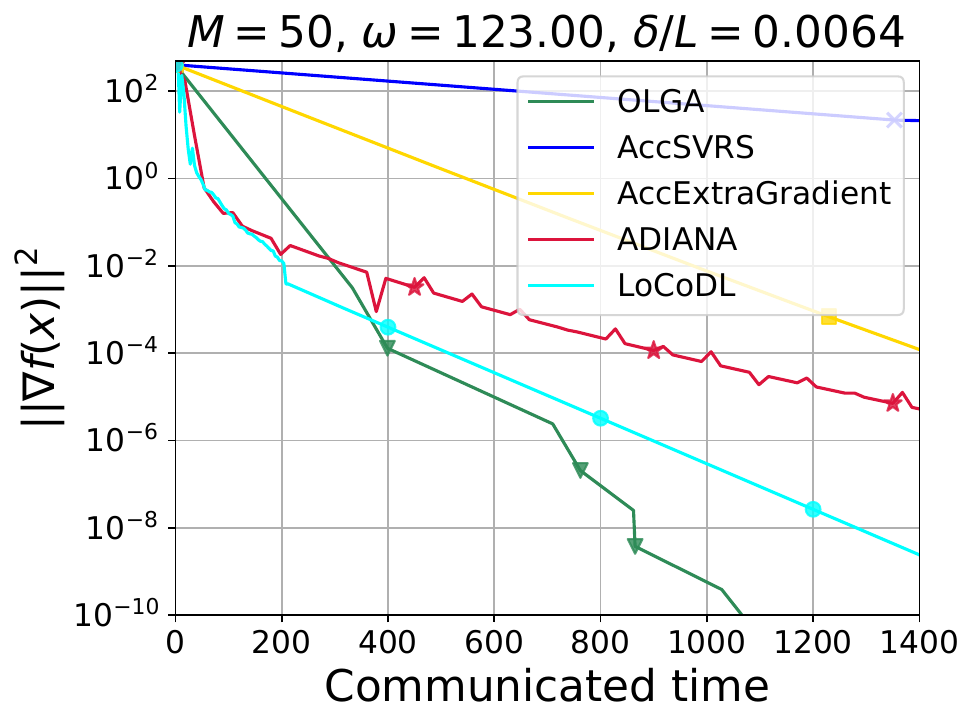}
   \end{subfigure}
   \begin{subfigure}{0.155\textwidth}
       \includegraphics[width=\linewidth]{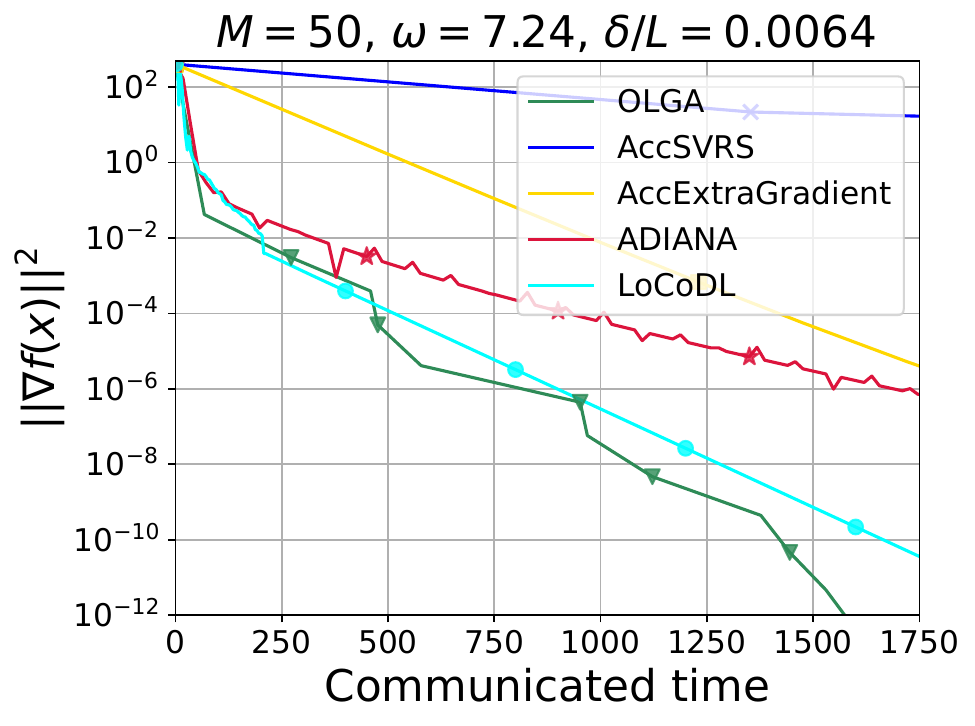}
   \end{subfigure}
   \begin{subfigure}{0.155\textwidth}
       \includegraphics[width=\linewidth]{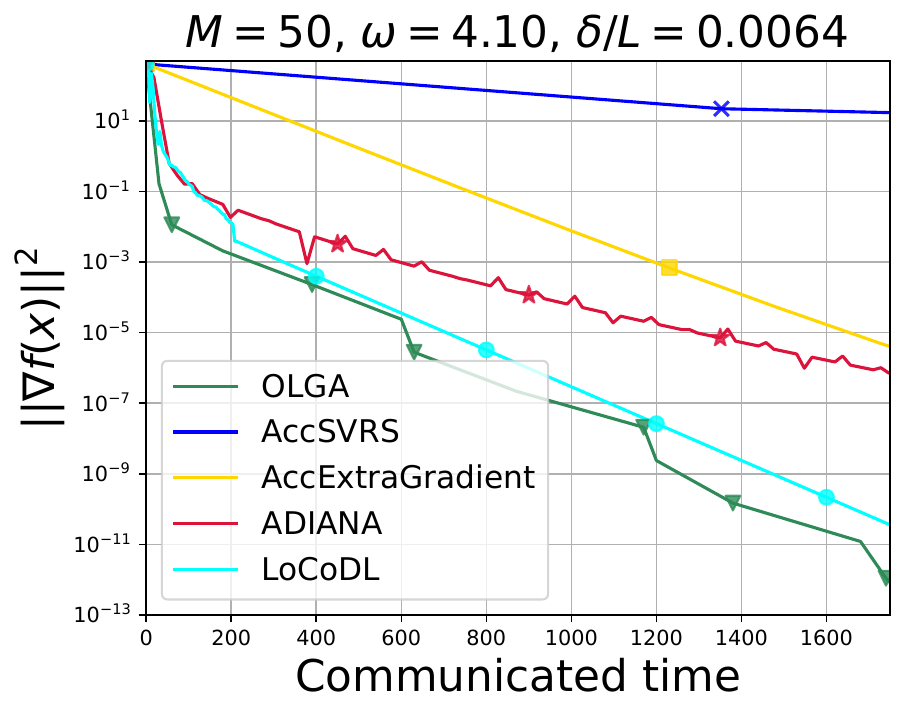}
   \end{subfigure}
   \begin{subfigure}{0.155\textwidth}
       \centering
       (a) $\omega = 123.00$
   \end{subfigure}
   \begin{subfigure}{0.155\textwidth}
       \centering
       (b) $\omega = 7.24$
   \end{subfigure}
   \begin{subfigure}{0.155\textwidth}
       \centering
       (c) $\omega = 4.10$
   \end{subfigure}
   \caption{Comparison of state-of-the-art distributed methods. The comparison is made on \eqref{eq:quadr} with $M=50$ and \texttt{a9a} dataset. The criterion is the communication time (\textit{CC-3}). For methods with compression we vary the power of compression $\omega$.}
\end{figure}
\begin{figure}[h!] 
   \begin{subfigure}{0.155\textwidth}
       \includegraphics[width=\linewidth]{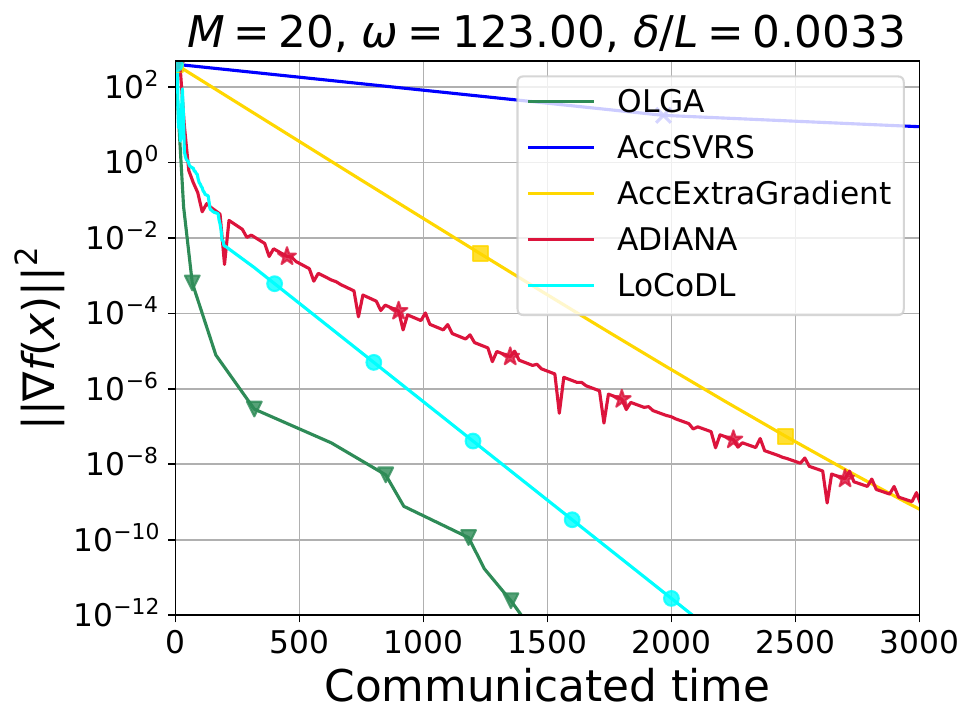}
   \end{subfigure}
   \begin{subfigure}{0.155\textwidth}
       \includegraphics[width=\linewidth]{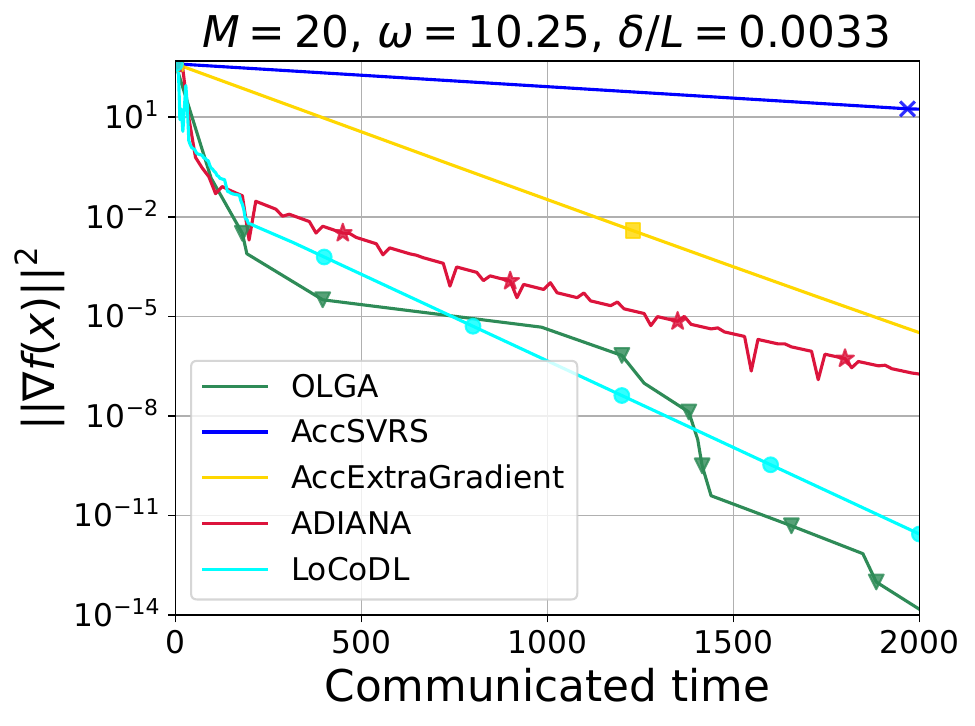}
   \end{subfigure}
   \begin{subfigure}{0.155\textwidth}
       \includegraphics[width=\linewidth]{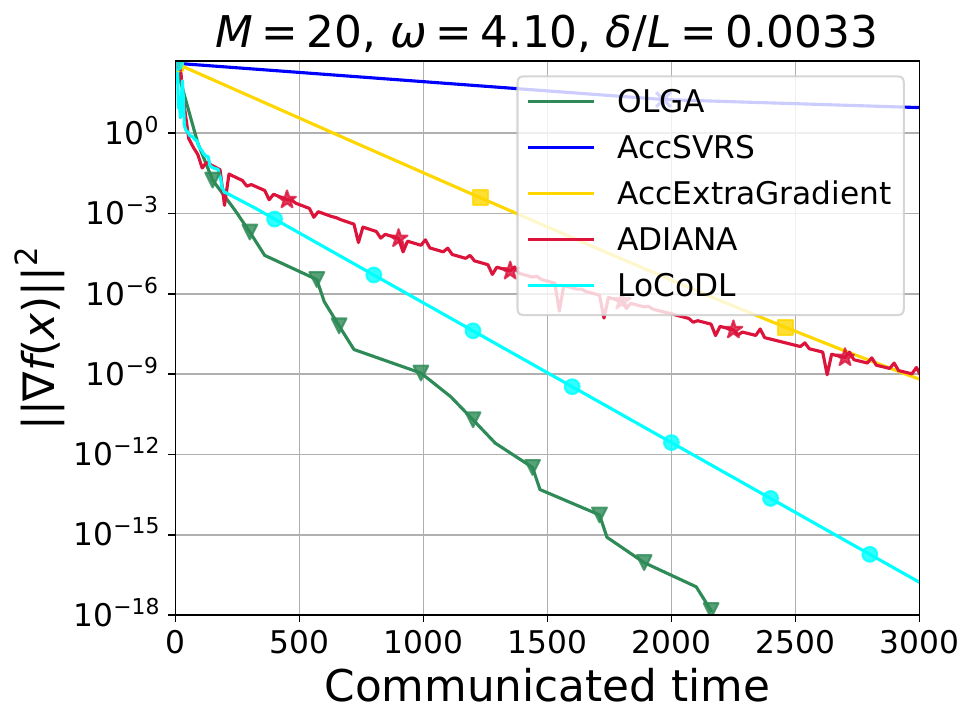}
   \end{subfigure}
   \begin{subfigure}{0.155\textwidth}
       \centering
       (a) $\omega = 123.00$
   \end{subfigure}
   \begin{subfigure}{0.155\textwidth}
       \centering
       (b) $\omega = 10.25$
   \end{subfigure}
   \begin{subfigure}{0.155\textwidth}
       \centering
       (c) $\omega = 4.10$
   \end{subfigure}
   \caption{Comparison of state-of-the-art distributed methods. The comparison is made on \eqref{eq:quadr} with $M=20$ and \texttt{a9a} dataset. The criterion is the communication time (\textit{CC-3}). For methods with compression we vary the power of compression $\omega$.}
\end{figure}

\begin{figure}[h!] 
   \begin{subfigure}{0.155\textwidth}
       \includegraphics[width=\linewidth]{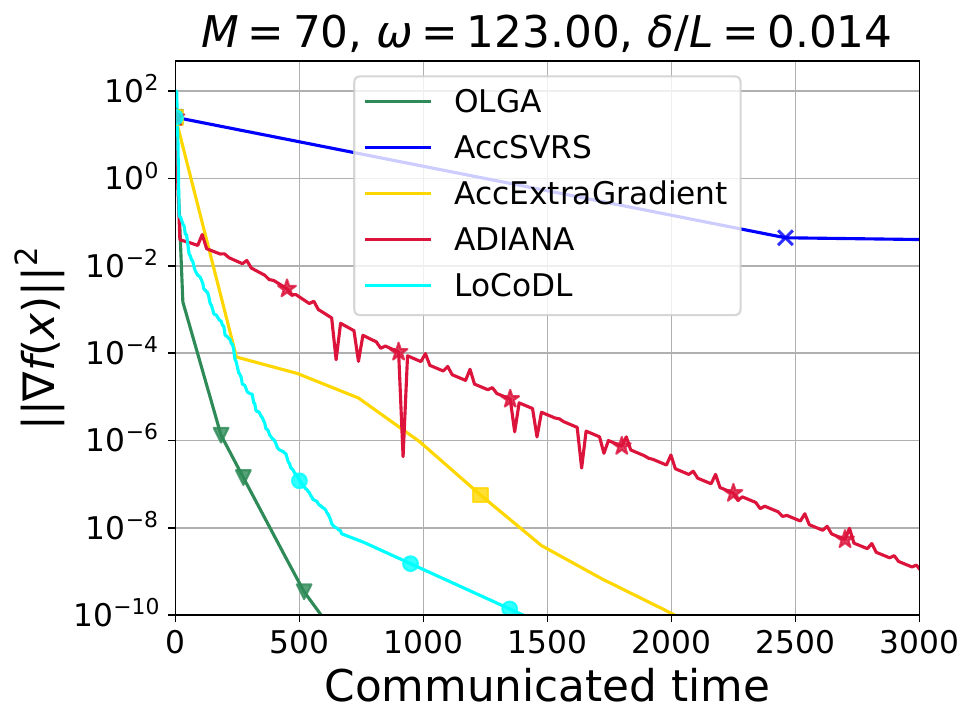}
   \end{subfigure}
   \begin{subfigure}{0.155\textwidth}
       \includegraphics[width=\linewidth]{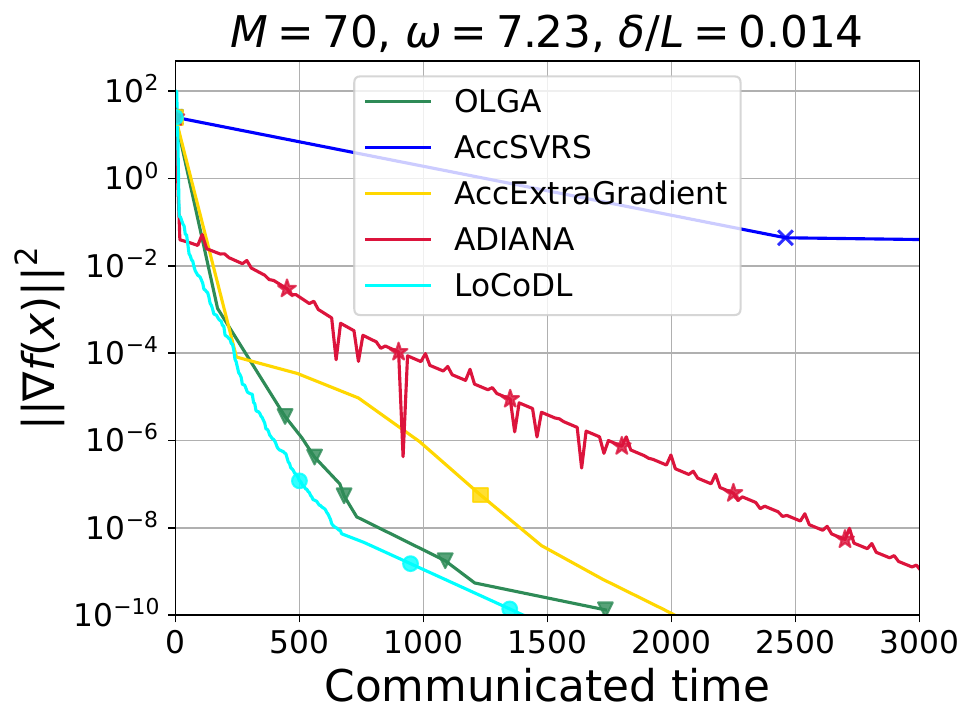}
   \end{subfigure}
   \begin{subfigure}{0.155\textwidth}
       \includegraphics[width=\linewidth]{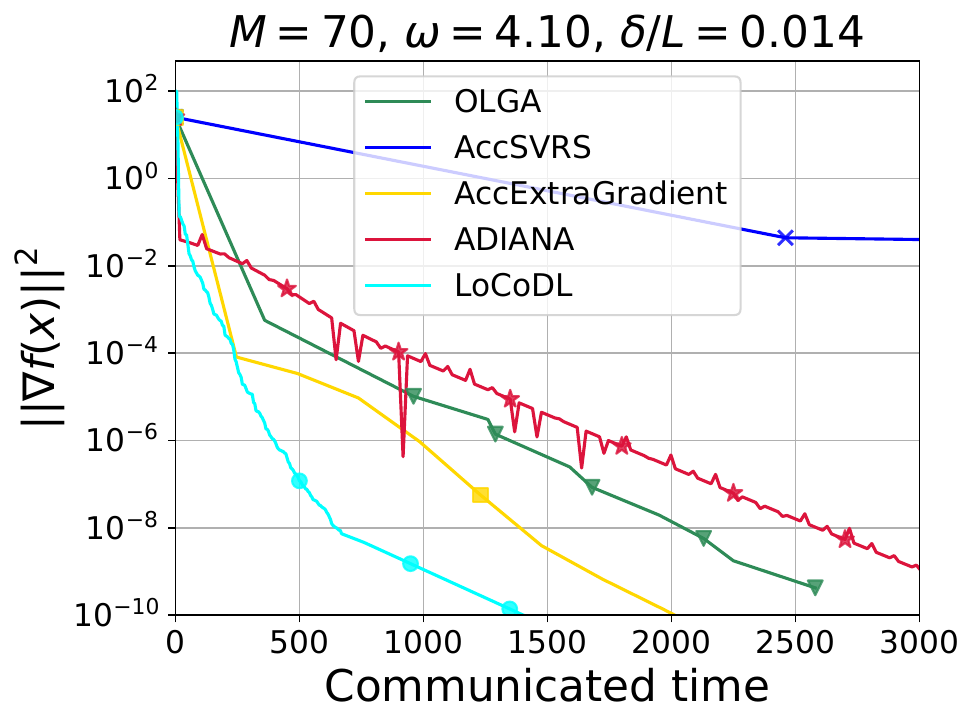}
   \end{subfigure}
   \begin{subfigure}{0.155\textwidth}
       \centering
       (a) $\omega = 123.00$
   \end{subfigure}
   \begin{subfigure}{0.155\textwidth}
       \centering
       (b) $\omega = 7.23$
   \end{subfigure}
   \begin{subfigure}{0.155\textwidth}
       \centering
       (c) $\omega = 4.10$
   \end{subfigure}
   \caption{Comparison of state-of-the-art distributed methods. The comparison is made on \eqref{eq:logloss} with $M=70$ and \texttt{a9a} dataset. The criterion is the communication time (\textit{CC-3}). For methods with compression we vary the power of compression $\omega$.}
\end{figure}
\begin{figure}[h!] 
   \begin{subfigure}{0.155\textwidth}
       \includegraphics[width=\linewidth]{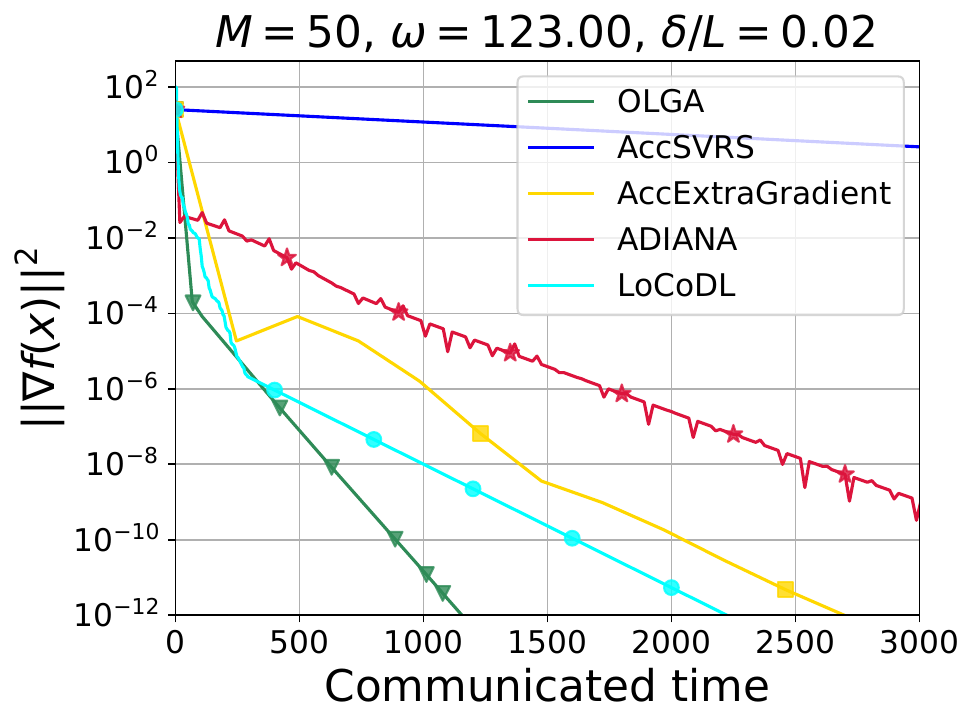}
   \end{subfigure}
   \begin{subfigure}{0.155\textwidth}
       \includegraphics[width=\linewidth]{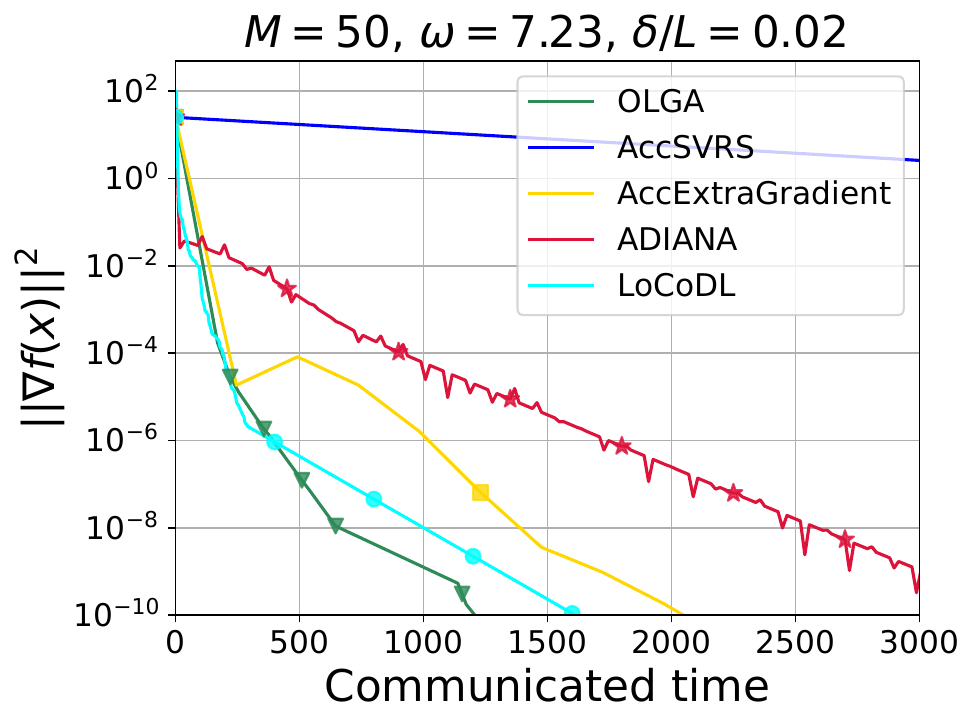}
   \end{subfigure}
   \begin{subfigure}{0.155\textwidth}
       \includegraphics[width=\linewidth]{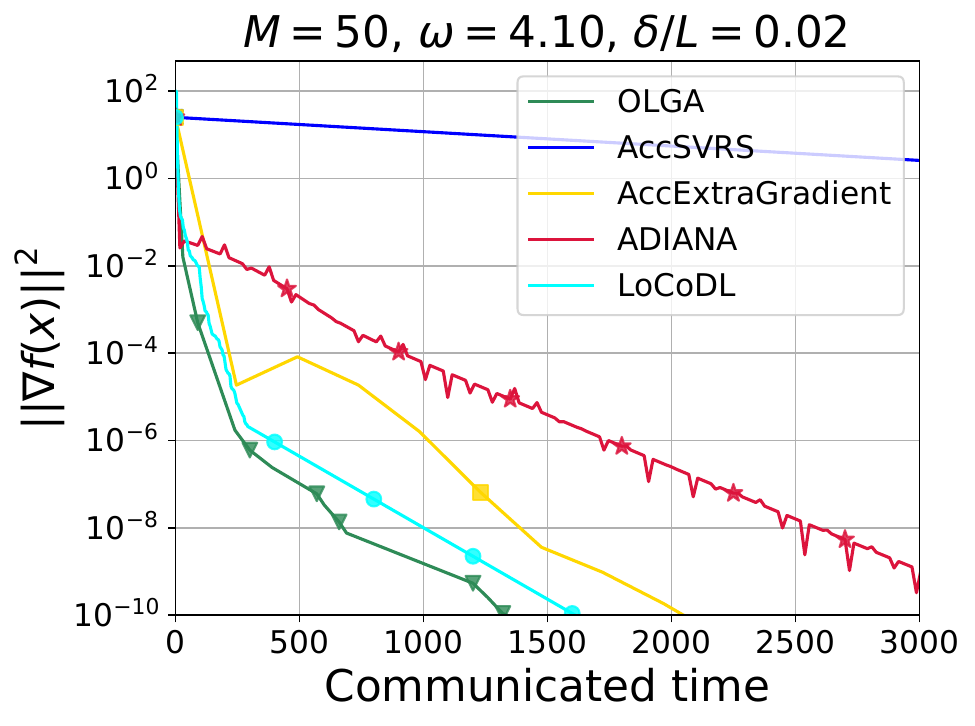}
   \end{subfigure}
   \begin{subfigure}{0.155\textwidth}
       \centering
       (a) $\omega = 123.00$
   \end{subfigure}
   \begin{subfigure}{0.155\textwidth}
       \centering
       (b) $\omega = 7.23$
   \end{subfigure}
   \begin{subfigure}{0.155\textwidth}
       \centering
       (c) $\omega = 4.10$
   \end{subfigure}
   \caption{Comparison of state-of-the-art distributed methods. The comparison is made on \eqref{eq:logloss} with $M=50$ and \texttt{a9a} dataset. The criterion is the communication time (\textit{CC-3}). For methods with compression we vary the power of compression $\omega$.}
\end{figure}
\begin{figure}[h!] 
   \begin{subfigure}{0.155\textwidth}
       \includegraphics[width=\linewidth]{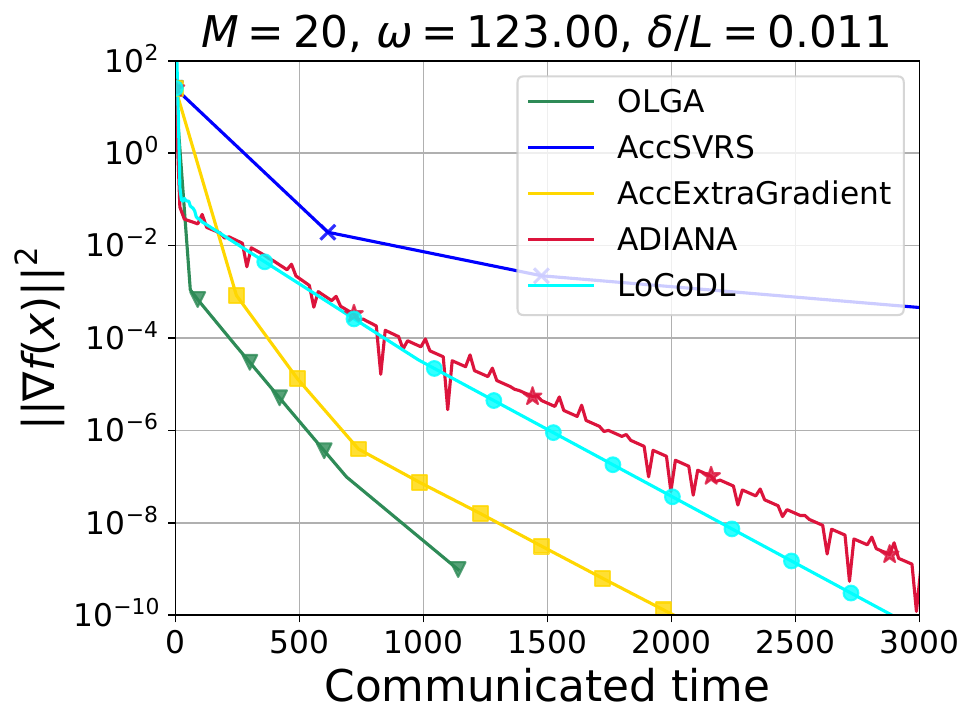}
   \end{subfigure}
   \begin{subfigure}{0.155\textwidth}
       \includegraphics[width=\linewidth]{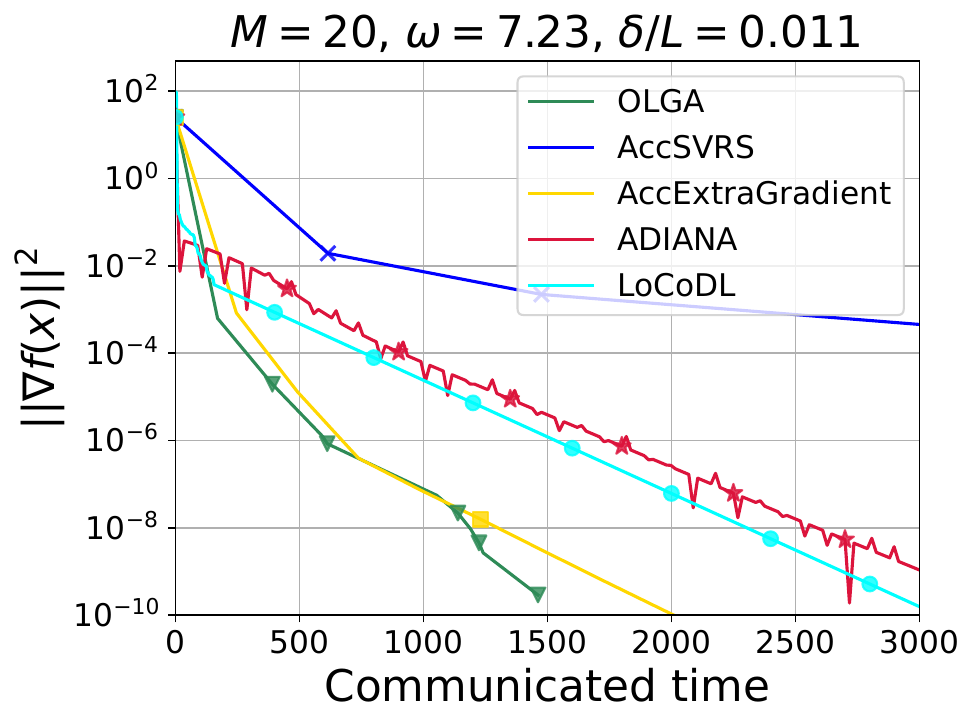}
   \end{subfigure}
   \begin{subfigure}{0.155\textwidth}
       \includegraphics[width=\linewidth]{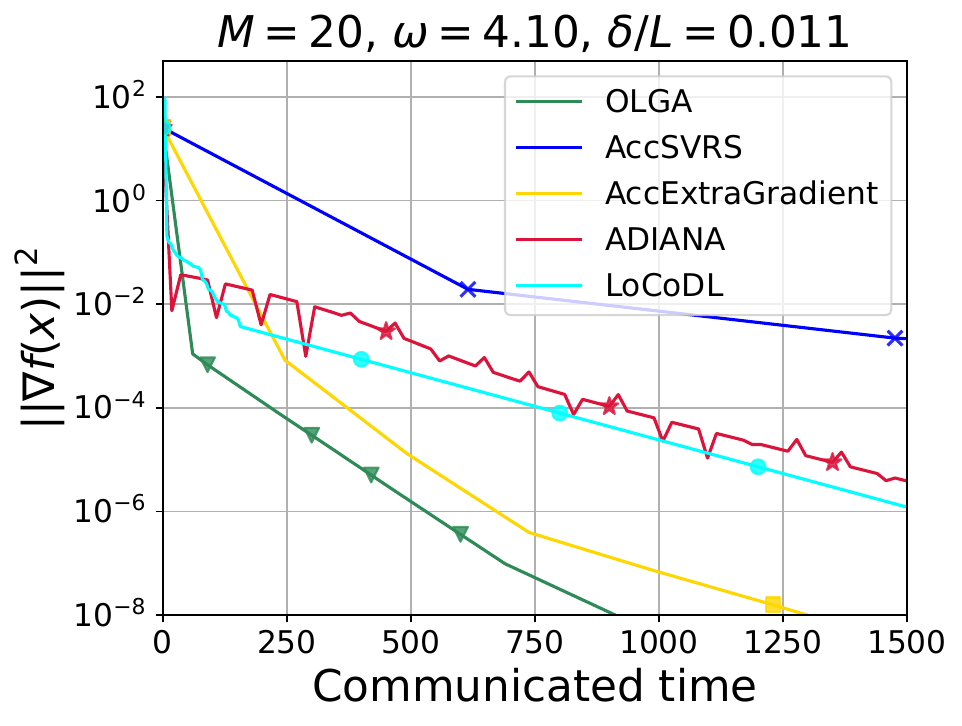}
   \end{subfigure}
   \begin{subfigure}{0.155\textwidth}
       \centering
       (a) $\omega = 123.00$
   \end{subfigure}
   \begin{subfigure}{0.155\textwidth}
       \centering
       (b) $\omega = 7.23$
   \end{subfigure}
   \begin{subfigure}{0.155\textwidth}
       \centering
       (c) $\omega = 4.10$
   \end{subfigure}
   \caption{Comparison of state-of-the-art distributed methods. The comparison is made on \eqref{eq:logloss} with $M=20$ and \texttt{a9a} dataset. The criterion is the communication time (\textit{CC-3}). For methods with compression we vary the power of compression $\omega$.}
   \label{fig:6}
\end{figure}

\section{Conclusion}
In this paper, we pioneered two algorithms, \texttt{OLGA} and \texttt{EF-OLGA}, that bridge the gap between acceleration, similarity and compression. For unbiased compressors, our theory guarantees an advantage over SOTAs in terms of communication time. 
Numerical experiments support our theoretical findings. Nevertheless, a number of questions remain open. For biased compressors, a lower bound on the communication complexity under the similarity condition is unknown: can we design an accelerated algorithm that enjoys a better complexity? A part of the discussion of our results that does not fit into the main text, can be found in Appendix \ref{Future Extensions}.

\section*{Acknowledgments}

The work was done in the Laboratory of Federated Learning Problems of the ISP RAS (Supported by Grant App. No. 2 to Agreement No. 075-03-2024-214).


\onecolumn
\appendixpage
\appendix
\section{Auxillary Definitions and Lemmas}

\begin{definition}\label{def_bregman}
    Consider a continiously differentiable $h\colon\R\rightarrow\mathbb{R}$ and such $D_h(\cdot,\cdot)$ that
    \begin{align*}
        D_h(x,y)=h(x)-h(y)-\langle\nabla h(y),x-y\rangle.
    \end{align*}
    $D_h(\cdot,\cdot)$ is called the Bregman divergence associated with $h(\cdot)$.
\end{definition}

\begin{lemma}\citep{katyushaX}\label{lem:prob}
    Given sequence $D_0,D_1,...,D_N\in\R,$ where $N\in \text{Geom}(p)$. Then
    \begin{equation}\label{lem:geom}
        \E_N[D_{N-1}]=pD_0 + (1-p)\E_N[D_N].
    \end{equation}
\end{lemma}

\begin{lemma}\citep{katyushaX}\label{lem:subtask}
    If $g(\cdot)$ is proper $\sigma$-strongly convex and $z_{k+1}=\arg\min_{z\in\R}\left[ \frac{1}{2\alpha}\|z-z_k\|^2 + \langle G_{k+1},z\rangle +\frac{\mu}{2}\|z-y_{k+1}\|^2  \right]$, then for every $x\in\R$ we have
    \begin{equation}\label{lem:3}
        \langle G_{k+1},z_k-x\rangle + g(z_{k+1})-g(x) \leq \frac{\alpha}{2}\|G_{k+1}\|^2 + \frac{\|z_k-x\|^2}{2\alpha} - \frac{(1+\sigma\alpha)\|z_{k+1}-x\|^2}{2\alpha}.
    \end{equation}
\end{lemma}

\begin{proposition}\citep{nesterov2013introductory}\label{nest_smooth}
    Let $f:\R\to\mathbb{R}$ be a convex function with a $L$-Lipschitz gradient. The following inequality holds for every $x,y\in\R$:
    \begin{align*}
        \|\nabla f(x)-\nabla f(y)\|^2\leq 2L\left( f(y)-f(x)-\langle 
\nabla f(x),y-x \rangle \right).
    \end{align*}
\end{proposition}

\begin{proposition}
    \textbf{(Three-point equality)} \citep{svrs}. Given a differentiable function $h\colon\R\rightarrow\mathbb{R}$. We have
    \begin{align}\label{three-point}
        \langle x-y,\nabla h(y)-\nabla h(z)\rangle = D_h(x, z) - D_h(x, y) - D_h(y, z).
    \end{align}
\end{proposition}

\section{Proof of Lemma 1}
\begin{lemma}\textbf{(Lemma \ref{lem:4}).}
    Consider an epoch of Algorithm \ref{alg:unbiased_compr_epoch}. Let $h(x)=f_1(x)-f(x)+\frac{1}{2\theta}\|x\|^2$, where $\theta\leq\min\left\{\frac{\sqrt{p}\sqrt{M}}{8\delta\sqrt{\omega}}, \frac{1}{2\delta}\right\}$. Then the following inequality holds for every $x\in\R$:
    \begin{align*}
        \begin{split}\E \left[f(x_N)-f(x)\right] \leq& \E \left[\langle x-x_0, \nabla h(x_N)-\nabla h(x_0) \rangle- \frac{p}{2}D_{h}(x_0,x_N)-\frac{\mu}{2}\|x_N-x\|^2\right].\end{split}
    \end{align*}
\end{lemma}
\begin{proof}\label{proof_4}
    Our goal at this point is to get an evaluation within the single epoch. Let us start with the definition of strong convexity:
    \begin{align}\label{strong_convexity:lemma_1}
        \E[f(x_{k+1})-f(x)]&\leq\E\left[ \langle x-x_{k+1},-\nabla f(x_{k+1})\rangle -\frac{\mu}{2}\|x_{k+1}-x\|^2 \right].
    \end{align}
    Writing the optimality condition of subproblem in Line \ref{unbiased_epoch:task} of Algorithm \ref{alg:unbiased_compr_epoch}, we obtain
    \begin{align*}
        g_k-\nabla f_1(x_0)+\nabla f(x_0) + \frac{x_{k+1}-x_k}{\theta}+\nabla f_1(x_{k+1})=0.
    \end{align*}    
    Let us rewrite it in the following form:
    \begin{align}\label{proof:unbiased_prox_rewrite}
        \nabla f_1(x_{k+1}) = \frac{x_k-x_{k+1}}{\theta}+\nabla(f_1-f)(x_0)-g_k.
    \end{align}
    It is easy to note that 
    \begin{align*}
        \nabla f(x_{k+1}) = \nabla f_1(x_{k+1}) + \nabla f(x_{k+1}) - \nabla f_1(x_{k+1}).
    \end{align*}
    By substituting \eqref{proof:unbiased_prox_rewrite} to this expression, we obtain
    \begin{align*}
        \nabla f(x_{k+1}) = \frac{x_k-x_{k+1}}{\theta} + \nabla(f_1-f)(x_0)-\nabla(f_1-f)(x_{k+1}) - g_k.
    \end{align*}
    Thus, we can rewrite the scalar product of \eqref{strong_convexity:lemma_1} in the following way:
    \begin{align*}
        \langle x-x_{k+1}, -\nabla f(x_{k+1})\rangle =& \frac{1}{\theta}\langle x-x_{k+1}, x_{k+1}-x_k \rangle + \langle x-x_{k+1}, g_k \rangle 
        \\
        &+ \langle x-x_{k+1}, \nabla(f_1-f)(x_{k+1}) -\nabla(f_1-f)(x_0)\rangle
        \\
        =& \frac{1}{\theta}\langle x-x_{k+1}, x_{k+1}-x_k \rangle 
        \\
        &+ \langle x-x_{k+1}, g_k - \nabla(f_1-f)(x_0) + \nabla(f_1-f)(x_k) \rangle 
        \\
        &+ \langle x-x_{k+1}, \nabla(f_1-f)(x_{k+1}) -\nabla(f_1-f)(x_k)\rangle 
        \\
        =& \left\langle x-x_{k+1}, \frac{x_{k+1}-x_k}{\theta} + \nabla(f_1-f)(x_{k+1}) - \nabla(f_1-f)(x_k) \right\rangle
        \\
        &+ \langle x-x_{k+1}, g_k-(\nabla(f_1-f)(x_0)-\nabla(f_1-f)(x_k))\rangle.
    \end{align*}
    With the notation of $g_k$, representing $g_k-(\nabla(f_1-f)(x_0)-\nabla(f_1-f)(x_k))$ as a sum
    \begin{align*}
        g_k-\left(\nabla(f_1-f)(x_0)-\nabla(f_1-f)(x_k)\right) = \frac{1}{M}\sum_{m=1}^M\left(g_k^m - \left(\nabla(f_1-f_m)(x_0)-\nabla(f_1-f_m)(x_k)\right)\right),
    \end{align*}
    and denoting $h(x)=f_1(x)-f(x)+\frac{1}{2\theta}\|x\|^2$ (as in the statement of Lemma \ref{lem:4}), we obtain
    \begin{align}\label{unbiased_epoch:almost_ready}
        \begin{split}
        \langle x-x_{k+1}, -\nabla f(x_{k+1})\rangle =& \langle x-x_{k+1},\nabla h(x_{k+1})-\nabla h(x_k) \rangle
        \\
        &+ \left\langle x-x_{k+1}, \frac{1}{M}\sum_{m=1}^Mg_k^m - \{\nabla(f_1-f_m)(x_0)-\nabla(f_1-f_m)(x_k)\} \right\rangle.
        \end{split}
    \end{align}
    Applying  \eqref{three-point} to the first summand of \eqref{unbiased_epoch:almost_ready}, we get
    \begin{align*}
    \begin{split}
        \langle x-x_{k+1}, -\nabla f(x_{k+1})\rangle =& D_h(x,x_k)-D_h(x,x_{k+1})-D_h(x_{k+1},x_k)
        \\
        &+ \left\langle x-x_{k+1}, \frac{1}{M}\sum_{m=1}^Mg_k^m - \{\nabla(f_1-f_m)(x_0)-\nabla(f_1-f_m)(x_k)\} \right\rangle.
    \end{split}
    \end{align*}
    Now we are ready to calculate the expectation from both sides of the inequality. Note that $\E_Q\left[\frac{1}{M}\sum_{m=1}^Mg_k^m - \{\nabla(f_1-f_m)(x_0)-\nabla(f_1-f_m)(x_k)\}\right]=0$ due to Definition \ref{unbiased_def}. Let us apply tower property to the second summand of \eqref{unbiased_epoch:almost_ready} and replace $x$ with $x_k$ (here we introduce $t_k=\frac{1}{M}\sum_{m=1}^Mg_k^m - \{\nabla(f_1-f_m)(x_0)-\nabla(f_1-f_m)(x_k)\}$):
    \begin{align*}
        \E\left[ \langle x-x_{k+1}, t_k \rangle \right] =& \E\E_Q\left[ \langle x-x_{k+1}, t_k \rangle \right] = \E\E_Q\left[ \langle x-x_{k}, t_k \rangle \right] + \E\E_Q\left[ \langle x_k-x_{k+1}, t_k \rangle \right] = \E\left[ \langle x_k-x_{k+1}, t_k \rangle \right].
    \end{align*}
    In the last step, we used the independence of $x_k$, $x$ and output of $Q$ on the $k$-th iteration. Thus, we have
    \begin{align}\label{unbiased_epoch:almost_ready_2}
    \begin{split}
        \E\left[\langle x-x_{k+1}, -\nabla f(x_{k+1})\rangle\right] =& \E\left[D_h(x,x_k)-D_h(x,x_{k+1})-D_h(x_{k+1},x_k)\right]
        \\
        &+ \E\left[\left\langle x_k-x_{k+1}, \frac{1}{M}\sum_{m=1}^M[g_k^m - \{\nabla(f_1-f_m)(x_0)-\nabla(f_1-f_m)(x_k)\}] \right\rangle\right].
    \end{split}
    \end{align}
    Applying the Cauchy-Schwartz inequality to $\frac{1}{M}\sum_{m=1}^M\left\langle x_k-x_{k+1}, g_k - \{\nabla(f_1-f_m)(x_0)-\nabla(f_1-f_m)(x_k)\} \right\rangle$, we get
    \begin{align*}
        &\E\left[\left\langle x_k-x_{k+1}, \frac{1}{M}\sum_{m=1}^Mg_k^m - \{\nabla(f_1-f_m)(x_0)-\nabla(f_1-f_m)(x_k)\} \right\rangle\right]
        \\
        &\leq \frac{1-\theta\delta}{4\theta}\E\left[\|x_{k+1}-x_k\|^2\right] + \frac{\theta}{1-\theta\delta}\E\left[\left\|\frac{1}{M}\sum_{m=1}^Mg_k^m - \{\nabla(f_1-f_m)(x_0)-\nabla(f_1-f_m)(x_k)\}\right\|^2\right].
    \end{align*}
    One can estimate the second term using Definition \ref{unbiased_def} and independence of the operators $Q$ on different devices. Writing out the expectation of the compressor action, we obtain
    \begin{align*}
        \begin{split}
        \E_Q\Bigg[&\left\|\frac{1}{M}\sum_{m=1}^M[g_k^m - \{\nabla(f_1-f_m)(x_0)-\nabla(f_1-f_m)(x_k)\}]\right\|^2\Bigg] 
        \\
        \leq& \E_Q\left[\frac{2}{M^2}\sum_{m_i<m_j=1}^M\left\langle g_k^{m_i} -\nabla(f_1-f_{m_i})(x_0)+\nabla(f_1-f_{m_i})(x_k), g_k^{m_j} - \nabla(f_1-f_{m_j})(x_0)+\nabla(f_1-f_{m_j})(x_k)\right\rangle\right]
        \\
        &+ \E_Q\left[\frac{1}{M^2}\sum_{m=1}^M\E_Q \| g_k^m - \{\nabla(f_1-f_m)(x_0)-\nabla(f_1-f_m)(x_k)\} \|^2\right]
        \\
        =&\frac{1}{M^2}\sum_{m=1}^M\E_Q \| g_k^m - \{\nabla(f_1-f_m)(x_0)-\nabla(f_1-f_m)(x_k)\} \|^2.
        \end{split}
    \end{align*}
    Here the summand with scalar products is equal to zero due to independence of the operator $Q$ actions. Definition \ref{unbiased_def} implies $\E_Q \| g_k^m - \{\nabla(f_1-f_m)(x_0)-\nabla(f_1-f_m)(x_k)\} \|^2\leq\omega\|\nabla(f_1-f_m)(x_0)-\nabla(f_1-f_m)(x_k)\|^2$. Taking the full expectation, we obtain
    \begin{align}\label{unbiased:compr_unrolled}
        \E\Bigg[&\left\|\frac{1}{M}\sum_{m=1}^M[g_k^m - \{\nabla(f_1-f_m)(x_0)-\nabla(f_1-f_m)(x_k)\}]\right\|^2\Bigg]\leq \frac{\omega}{M^2}\sum_{m=1}^M\E\left[\|\nabla(f_1-f_m)(x_0)-\nabla(f_1-f_m)(x_k)\|^2\right].
    \end{align}
    The next step is to apply similarity definition to \eqref{unbiased:compr_unrolled}: under the norm in the right-hand side we use “smart zeros” $\pm \nabla f(x_0)$ and $\pm \nabla f(x_k)$. We divide the obtained expression into two summands using the properties of the norm. Then we apply similarity (Definition \ref{sim_cor}) to each one.
    \begin{align}\label{abc}
    \begin{split}
        \frac{\omega}{M^2}\sum_{m=1}^M\E\left[\|\nabla(f_1-f_m)(x_0)-\nabla(f_1-f_m)(x_k)\|^2\right] \leq& \frac{2\omega}{M^2}\sum_{m=1}^M\E\left[\|\nabla(f_1-f)(x_0)-\nabla(f_1-f)(x_k)\|^2\right] \\&+ \frac{2\omega}{M^2}\sum_{m=1}^M\E\left[\|\nabla(f-f_m)(x_0)-\nabla(f-f_m)(x_k)\|^2\right]
        \\
        \leq& \frac{4\omega\delta^2}{M}\E\left[ \|x_k-x_0\|^2 \right].
    \end{split}
    \end{align}
    We substitute \eqref{abc} into \eqref{unbiased_epoch:almost_ready_2}:
    \begin{align}\label{unbiased_epoch:almost_ready_3}
    \begin{split}
        \E\left[\langle x-x_{k+1}, -\nabla f(x_{k+1})\rangle\right] =& \E\left[D_h(x,x_k)-D_h(x,x_{k+1})-D_h(x_{k+1},x_k)+\frac{1-\theta\delta}{4\theta}\|x_{k+1}-x_k\|^2\right]\\&+\E\left[\frac{4\theta\omega\delta^2}{(1-\theta\delta)M}\|x_k-x_0\|^2\right].
    \end{split}
    \end{align}
    Since $f_m-f$ is $\delta$-smooth (see \eqref{sim_cor}), one can note that $h(x)$ is $\left(\frac{1}{\theta}-\delta\right)$-strongly convex for $\theta\leq\nicefrac{1}{\delta}$. Our choice $\theta\leq\nicefrac{1}{2\delta}$ is appropriate. Moreover, $h(x)$ is $\left(\delta+\frac{1}{\theta}\right)$-smooth. \eqref{strong_convexity} gives $D_h(x,y)\geq\frac{1-\theta\delta}{2\theta}\|x-y\|^2$. Proposition \ref{nest_smooth} gives $D_{h}(x,y)\leq\frac{1+\theta\delta}{2\theta}\|x-y\|^2$. Let us write it down in a more convenient form:
    \begin{align}\label{lem:unbiased_div_estimation}
        0\leq\frac{1-\theta\delta}{2\theta}\|x-y\|^2\leq D_{h}(x,y)\leq\frac{1+\theta\delta}{2\theta}\|x-y\|^2.
    \end{align}
    Next, we use \eqref{lem:unbiased_div_estimation} and $\theta\leq\frac{1}{2\delta}$ to obtain:
    \begin{align}\label{values_1}
        \frac{1-\theta\delta}{4\theta}\|x_{k+1}-x_k\|^2\leq\frac{2\theta}{1-\theta\delta}\frac{1-\theta\delta}{4\theta}D_h(x_{k+1},x_k) = \frac{1}{2}D_h(x_{k+1},x_k),
    \end{align}
    and
    \begin{align}\label{values_2}
        \frac{4\theta\omega\delta^2}{(1-\theta\delta)M}\|x_k-x_0\|^2\leq\frac{2\theta}{1-\theta\delta}\frac{4\theta\omega\delta^2}{(1-\theta\delta)M}D_h(x_0,x_k)\leq\frac{32\theta^2\omega\delta^2}{M}D_h(x_0,x_k).
    \end{align}
    Let $k=N-1$. We can substitute \eqref{values_1}, \eqref{values_2} into \eqref{unbiased_epoch:almost_ready_3}, and then into \eqref{strong_convexity:lemma_1}:
    \begin{align*}
        \E[f(x_{N})-f(x)]\leq&\E\left[D_h(x,x_{N-1})-D_h(x,x_{N})-\frac{\mu}{2}\|x_{N}-x\|^2+\frac{32\theta^2\omega\delta^2}{M}D_h(x_0,x_{N-1})\right].
    \end{align*}
    Since $D_h(x_0,x_0)=0$, Lemma \ref{lem:prob} implies $\E\left[D_h(x_0,x_{N-1})\right]\leq\E\left[D_h(x_0,x_{N})\right]$ and $\E\left[D_h(x,x_{N-1}-D_h(x,x_N))\right]=p\E\left[D_h(x,x_{0})-D_h(x,x_N)\right]$. Thus, we obtain the following:
    \begin{align*}
        \E[f(x_{N})-f(x)]\leq&\E\left[pD_h(x,x_{0})-pD_h(x,x_{N})-\frac{\mu}{2}\|x_{N}-x\|^2+\frac{32\theta^2\omega\delta^2}{M}D_h(x_0,x_{N})\right].
    \end{align*}
    Let us apply Proposition \ref{three-point} in the form $D_h(x,x_0)+D_h(x_0,x_N)-D_h(x,x_N)=\langle x-x_0,\nabla h(x_0)-\nabla h(x_N) \rangle$:
    \begin{align*}
        \E[f(x_{N})-f(x)]\leq&\E\left[p\langle x-x_0,\nabla h(x_0)-\nabla h(x_N) \rangle-pD_h(x_0,x_N)-\frac{\mu}{2}\|x_{N}-x\|^2+\frac{32\theta^2\omega\delta^2}{M}D_h(x_0,x_{N})\right].
    \end{align*}
    Since $\theta\leq\min\left\{ \frac{\sqrt{p}\sqrt{M}}{8\delta\sqrt{\omega}}, \frac{1}{2\delta} \right\}$, we have $p-\frac{32\theta^2\omega\delta^2}{M}\leq-\frac{p}{2}$ and thus
    \begin{align*}
        \begin{split}\E \left[f(x_N)-f(x)\right] \leq& \E \left[\langle x-x_0, \nabla h(x_N)-\nabla h(x_0) \rangle- \frac{p}{2}D_{h}(x_0,x_N)-\frac{\mu}{2}\|x_N-x\|^2\right].\end{split}
    \end{align*}
    This completes the proof of the lemma.
\end{proof}

\section{Proof of Theorem 1}
\begin{theorem}\textbf{(Theorem 1)}
    Let the problem \eqref{prob_form} be solved by Algorithm \ref{alg:unbiased_compr} with $\theta\leq\min\left\{\frac{\sqrt{p}\sqrt{M}}{8\delta\sqrt{\omega}}, \frac{1}{2\delta}\right\}$ and such tuning of parameters that $4\alpha p\tau \leq \theta$. Then the following inequality holds:
    \begin{align*}
        \E\left[ Y_{k+1} + Z_{k+1} \right] \leq& \E\left[ (1-\tau)Y_k + (1+\mu\alpha)^{-1}Z_k  \right].
    \end{align*}
\end{theorem}
\begin{proof}\label{proof_th_2}
    Let us move from one epoch to the method as a whole. Since the output point of Algorithm \ref{alg:unbiased_compr_epoch} was used to calculate $y_{k+1}$, let us re-designate $x_0\to x_{k+1}$ and $x_N\to y_{k+1}$ in Lemma \ref{lem:4}: 
    \begin{align*}
        \begin{split}\E \left[f(y_{k+1})-f(x)\right] \leq& \E \left[\langle x-x_{k+1}, \nabla h(y_{k+1})-\nabla h(x_{k+1}) \rangle- \frac{p}{2}D_{h}(x_{k+1},y_{k+1})-\frac{\mu}{2}\|y_{k+1}-x\|^2\right].\end{split}
    \end{align*}
    One can see that $G_{k+1}$ from Line \ref{unbiased_G} is almost the same as $\nabla h(y_{k+1})-\nabla h(x_{k+1})$ in the expression above. Namely, $\nabla h(y_{k+1})-\nabla h(x_{k+1})=-G_{k+1}$. 
    Hence, using (\ref{lem:unbiased_estimation}) for the iteration of Algorithm \ref{alg:unbiased_compr}, we get
    \begin{align*}
        \E[f(y_{k+1})-f(x)] \leq& \E\left[ \langle x-x_{k+1},-G_{k+1} \rangle -\frac{p}{2}D_h(x_{k+1},y_{k+1})-\frac{\mu}{2}\|y_{k+1}-x\|^2 \right] 
        \\
        =& \E\left[ \langle z_k-x_{k+1},-G_{k+1} \rangle +  \langle x-z_k,-G_{k+1} \rangle -\frac{p}{2}D_h(x_{k+1},y_{k+1})\right]
        \\
        &-\E\left[\frac{\mu}{2}\|y_{k+1}-x\|^2 \right]
        \\
        =& \E\left[ \langle z_k-x_{k+1},-G_{k+1} \rangle +  \langle z_k-x,G_{k+1} \rangle -\frac{p}{2}D_h(x_{k+1},y_{k+1})\right]
        \\
        &-\E\left[\frac{\mu}{2}\|y_{k+1}-x\|^2 \right].
    \end{align*}
    Let us apply Line \ref{unbiased:3} of Algorithm \ref{alg:unbiased_compr} to the first term $\left(z_k-x_{k+1}=\frac{1-\tau}{\tau}(x_{k+1}-y_k)\right)$:
    \begin{align}\label{th_almost_1}
    \begin{split}
        \E[f(y_{k+1})-f(x)] \leq& \E\left[ 
\frac{1-\tau}{\tau}\langle x_{k+1}-y_k, -G_{k+1} \rangle +\langle z_k-x,G_{k+1} \rangle -\frac{p}{2}D_h(x_{k+1},y_{k+1})\right]
        \\
        &-\E\left[\frac{\mu}{2}\|y_{k+1}-x\|^2 \right].
    \end{split}
    \end{align}
    Using Lemma \ref{lem:4} (see Lemma \eqref{lem:unbiased_estimation}) with $x=y_k$:
    \begin{align*}
        \begin{split}\E \left[f(y_{k+1})-f(y_k)\right] \leq& \E \left[\langle y_k-x_{k+1}, -G_{k+1} \rangle- \frac{p}{2}D_{h}(x_{k+1},y_{k+1})-\frac{\mu}{2}\|y_{k+1}-y_k\|^2\right].\end{split}
    \end{align*}
    We rewrite this inequality in the following form:
    \begin{align*}
        \begin{split}\E \left[ \langle x_{k+1}-y_k, -G_{k+1} \rangle \right] \leq& \E \left[\left(f(y_{k})-f(y_{k+1})\right) -\frac{p}{2}D_{h}(x_{k+1},y_{k+1})-\frac{\mu}{2}\|y_{k+1}-y_k\|^2\right].\end{split}
    \end{align*}
    and substitute it into \eqref{th_almost_1}:
    \begin{align*}
    \begin{split}
        \E\left[\frac{1}{\tau}\left(f(y_{k+1})-f(x)\right)\right] \leq& \E\left[ 
\frac{1-\tau}{\tau}\left(f(y_k)-f(x)\right)+\langle z_k-x,G_{k+1} \rangle -\frac{p}{2\tau}D_h(x_{k+1},y_{k+1})\right]
        \\
        &-\E\left[\frac{\mu}{2}\|y_{k+1}-x\|^2 \right].
    \end{split}
    \end{align*}
    Applying Lemma \ref{lem:subtask} to the $\langle z_k-x,G_{k+1} \rangle - \frac{\mu}{2}\|y_{k+1}-x\|^2$ taking into account the fact that $\|y_{k+1}-z_{k+1}\|^2\geq0$, we obtain:
    \begin{align}\label{cbd}
    \begin{split}
        \E\left[\frac{1}{\tau}\left(f(y_{k+1})-f(x)\right)\right] \leq& \E\left[ \frac{1-\tau}{\tau}[f(y_k)-f(x)] - \frac{p}{2\tau}D_h(x_{k+1},y_{k+1}) + \frac{\alpha}{2}\|G_{k+1}\|^2  \right]
        \\
        &+ \E\left[ \frac{1}{2\alpha}\| z_k-x \|^2 - \frac{1+\mu\alpha}{2\alpha}\|z_{k+1}-x\|^2 \right].
    \end{split}
    \end{align}
    The last step left to do to complete the proof is to estimate $\|G_{k+1}\|^2$. Let us use \eqref{lem:unbiased_div_estimation} and Proposition \ref{nest_smooth} and the fact that $\theta\leq\nicefrac{1}{2\delta}$ here:
    \begin{align*}
        \frac{\alpha}{2}\|G_{k+1}\|^2 =& \frac{\alpha p^2}{2}\|\nabla h(y_{k+1})-\nabla h(x_{k+1})\|^2 \leq \frac{\alpha p^2}{2}\frac{2(1+\theta\delta)}{\theta}D_h(x_{k+1},y_{k+1})
        \\
        \leq& \frac{3\alpha p^2}{2\theta}D_h(x_{k+1},y_{k+1}) \leq \frac{2\alpha p^2}{\theta}D_h(x_{k+1},y_{k+1}).
    \end{align*}
    Thus, we can rewrite \eqref{cbd} as follows:
    \begin{align*}
        \E\left[\frac{1}{\tau}(f(y_{k+1})-f(x))\right] \leq& \E\left[ \frac{1-\tau}{\tau}[f(y_k)-f(x)] + p\left(\frac{2\alpha p}{\theta}-\frac{1}{2\tau}\right)D_h(x_{k+1},y_{k+1})  \right]
        \\
        &+ \E\left[ \frac{1}{2\alpha}\| z_k-x \|^2 - \frac{1+\mu\alpha}{2\alpha}\|z_{k+1}-x\|^2 \right].
    \end{align*}
    Since $4\alpha p\tau\leq\theta$, we have
    \begin{align*}
        \E\left[\frac{1}{\tau}(f(y_{k+1})-f(x))+\frac{1+\mu\alpha}{2\alpha}\|z_{k+1}-x\|^2\right] \leq& \E\left[ \frac{1-\tau}{\tau}[f(y_k)-f(x)] + \frac{1}{2\alpha}\| z_k-x \|^2 \right].
    \end{align*}
\end{proof}

\section{Proof of Corollary 1}
\begin{corollary}\textbf{(Corollary 1).}
    Let the problem (\ref{prob_form}) be solved by Algorithm \ref{alg:unbiased_compr}. Choose $$p=\frac{1}{\gamma_{\omega}},\quad\theta\leq\min\left\{\frac{\sqrt{p}\sqrt{M}}{8\delta\sqrt{\omega}}, \frac{1}{2\delta}\right\},$$ $$\tau=\min\left\{ \frac{\sqrt{\mu}\theta^{\nicefrac{1}{2}}p^{-\nicefrac{1}{2}}}{4},\frac{1}{4} \right\},\quad\alpha=\frac{\theta p^{-1}}{8\tau},$$ then Algorithm \ref{alg:unbiased_compr} has
    $$\tilde{\mathcal{O}}\left( \gamma_{\omega}+\sqrt{\frac{\delta}{\mu}}\left[ \gamma_{\omega}M^{-\nicefrac{1}{4}} + \gamma_{\omega}^{\nicefrac{1}{2}} \right] \right) \textit{ CC-1},\quad\tilde{\mathcal{O}}\left( M\gamma_{\omega}+\sqrt{\frac{\delta}{\mu}}\left[ \gamma_{\omega}M^{\nicefrac{3}{4}} + M\gamma_{\omega}^{\nicefrac{1}{2}} \right] \right) \textit{ CC-2},$$
    and
    $$\tilde{\mathcal{O}}\left( 1+\sqrt{\frac{\delta}{\mu}}\left[ M^{-\nicefrac{1}{4}} + \gamma_{\omega}^{-\nicefrac{1}{2}} \right] \right) \textit{ CC-3}.$$
\end{corollary}
\begin{proof}
    Let us enumerate the cases:
    \begin{enumerate}
        \item Let $\tau=\min\left\{ \frac{1}{4}, \frac{\sqrt{\mu}\theta^{\nicefrac{1}{2}}p^{-\nicefrac{1}{2}}}{4} \right\} = \frac{1}{4}$. In this case, we have
        \begin{align*}
            (1-\tau)(1+\mu\alpha)=& (1-\tau)\left(1+\frac{\mu\theta p^{-1}}{8\tau}\right) \geq (1-\tau)\left(1+\frac{1}{8\tau}\right) = \frac{3}{4}\cdot\frac{3}{2}\geq 1.
        \end{align*}
         This implies $\E[Z_{k+1}+Y_{k+1}]\leq(1-
    \frac{1}{4})\E[Z_{k}+Y_{k}]$. Thus, \textit{CC-3} of Algorithm \ref{alg:unbiased_compr_epoch} is $\tilde{\mathcal{O}}\left( 1 \right)$. 
        \item Let  $\tau = \min\left\{ \frac{1}{4}, \frac{\sqrt{\mu}\theta^{\nicefrac{1}{2}}p^{-\nicefrac{1}{2}}}{4} \right\} = \frac{\sqrt{\mu}\theta^{\nicefrac{1}{2}}p^{-\nicefrac{1}{2}}}{4}$. In this case, we have
        \begin{align*}
            \mu\alpha=\frac{\mu\theta p^{-1}}{8\tau}=\frac{\sqrt{\mu}\theta^{\nicefrac{1}{2}}p^{-\nicefrac{1}{2}}}{2}\leq\frac{1}{2}<1.
        \end{align*}
        Thus, $(1+\mu\alpha)^{-1}\leq1-\frac{\mu\alpha}{2}$. This implies $\E[Z_{k+1}+Y_{k+1}]\leq(1-\tau)\E[\Phi_k]$. Or, in \textit{CC-3} terms: \begin{align*}
            \tilde{\mathcal{O}}\left( 1+\frac{1}{\tau} \right)=&\tilde{\mathcal{O}}\left( 1+\sqrt{\frac{\delta}{\mu}}\left[M^{-\nicefrac{1}{4}}\omega^{\nicefrac{1}{4}}p^{\nicefrac{1}{4}}+p^{\nicefrac{1}{2}}\right] \right).
        \end{align*}
    Substitute $p=\nicefrac{1}{\gamma_{\omega}}$:
    \begin{align*}
        \tilde{\mathcal{O}}\left( 1+\frac{1}{\tau} \right)=&\tilde{\mathcal{O}}\left( 1+\sqrt{\frac{\delta}{\mu}}\left[M^{-\nicefrac{1}{4}}\omega^{\nicefrac{1}{4}}\gamma_{\omega}^{-\nicefrac{1}{4}}+\gamma_{\omega}^{-\nicefrac{1}{2}}\right] \right).
    \end{align*}
    \end{enumerate}
    Summing both cases and using $\omega\leq\gamma_{\omega}$, we obtain $\tilde{\mathcal{O}}\left( 1+\sqrt{\frac{\delta}{\mu}}\left[M^{-\nicefrac{1}{4}}+\gamma_{\omega}^{-\nicefrac{1}{2}}\right] \right)$ of \textit{CC-3}.
\end{proof}

\section{Proof of Lemma 5}
Before proceeding to the proof of the Theorem \ref{th:biased}, we introduce an auxiliary lemma.
\begin{lemma}\label{lem:2}
    Consider an epoch of Algorithm \ref{alg:biased_compr_epoch}. Consider $h(x)=f_1(x)-f(x)+\frac{1}{2\theta}\|x\|^2$, where $\theta\leq\frac{p^{\nicefrac{3}{2}}}{24\delta}$. The following inequality holds:
    \begin{equation*}
        \E[f(x_{N})-f(x)] \leq \E\left[ p\left\langle x-x_0,\nabla(f_1-f)(x_N)-\nabla(f_1-f)(x_0)+\frac{\tilde{x}_N-x_0}{\theta}\right\rangle -\frac{p}{7}D_h(x_0,x_N)-\frac{\mu}{2}\|x_{N}-x\|^2 \right].
    \end{equation*}
\end{lemma}
\begin{proof}\label{proof_lem_2}
    Same as in the unbiased case, our goal at the begining is to get some sort of evaluation within a single epoch. Let us start with strong convexity definition:
    \begin{align}\label{proof:convexity}
        \E[f(x_{k+1})-f(x)]&\leq\E\left[ \langle x-x_{k+1},-\nabla f(x_{k+1})\rangle -\frac{\mu}{2}\|x_{k+1}-x\|^2 \right].
    \end{align}
    Denote $e_k=\frac{1}{M}\sum_{m=1}^Me_k^m$ and introduce virtual sequence $\tilde{x}_k=x_k-e_k$. We write the optimality condition for Line \ref{epoch:task} of Algorithm \ref{alg:biased_compr_epoch}:
    \begin{align*}
        \frac{1}{\theta}g_k-\nabla f_1(x_0) + \nabla f(x_0)+\frac{x_{k+1}-x_k}{\theta}=0.
    \end{align*}
    Next we are going to use $x_{k+1}$, so let us express it:
    \begin{align}\label{proof:prox_rewrite}
        x_{k+1}=x_k - g_k - \theta[\nabla f_1(x_{k+1})+\nabla f(x_0) - \nabla f_1(x_0)].
    \end{align}
    For virtual sequence $\tilde{x}_k$ we have $\tilde{x}_{k+1}=x_{k+1}-e_{k+1}$. Let us obtain an expression for it using (\ref{proof:prox_rewrite}) and Line \ref{epoch:error} of Algorithm \ref{alg:biased_compr_epoch}.
    \begin{align*}
        \tilde{x}_{k+1}=&x_{k+1}-e_{k+1} 
        \\
        =& x_k - \theta[\nabla f_1(x_{k+1})+\nabla f(x_0) - \nabla f_1(x_0)] - g_k - e_k + g_k
        \\
        &+\theta[\nabla f(x_0) - \nabla f_1(x_0) + \nabla f_1(x_k)-\nabla f(x_k)] 
        \\
        =& \tilde{x}_k - \theta[\nabla f(x_k) - \nabla f_1(x_k) + \nabla f_1(x_{k+1})].
    \end{align*}
    After re-arranging terms, we write
    \begin{align}\label{biased_serv_grad}
        \nabla f_1(x_{k+1})=\frac{\tilde{x}_k-\tilde{x}_{k+1}}{\theta} + \nabla f_1(x_k) - \nabla f(x_k).
    \end{align}
    It is quite easy to note that
    \begin{align*}
        \nabla f(x_{k+1}) = \nabla f_1(x_{k+1}) + \nabla f(x_{k+1}) - \nabla f_1(x_{k+1}).
    \end{align*}
    Substitute \eqref{biased_serv_grad} to this expression and obtain
    \begin{align}\label{proof:grad}
        \nabla f(x_{k+1}) = \frac{\tilde{x}_k-\tilde{x}_{k+1}}{\theta} + \nabla(f_1-f)(x_k)-\nabla(f_1-f)(x_{k+1}).
    \end{align}
    Now we are ready to rewrite scalar product of \eqref{proof:convexity} using \eqref{proof:grad} in the following way:
    \begin{align*}
        \langle x-x_{k+1},-\nabla f(x_{k+1})\rangle =& \frac{1}{\theta}\langle 
        x-x_{k+1}\pm \tilde{x}_{k+1}, \tilde{x}_{k+1}-\tilde{x}_k \rangle 
        \\
        &+ \langle x-x_{k+1}, \nabla(f_1-f)(x_{k+1})-\nabla(f_1-f)(x_k)\rangle
        \\
        =& \langle x-x_{k+1}, \nabla(f_1-f)(x_{k+1})-\nabla(f_1-f)(x_k)\rangle 
        \\&+\frac{1}{\theta}\langle 
        x-\tilde{x}_{k+1}, \tilde{x}_{k+1}-\tilde{x}_k \rangle +\frac{1}{\theta}\langle 
        \tilde{x}_{k+1}-x_{k+1}, \tilde{x}_{k+1}-\tilde{x}_k \rangle 
    \end{align*}
    Next apply Proposition \ref{three-point} to the first summand, square of the norm formula to the second summand and Cauchy-Schwartz inequality to the third one.
    \begin{align*}
        \langle x-x_{k+1},-\nabla f(x_{k+1})\rangle =& D_{f_1-f}(x,x_k)-D_{f_1-f}(x,x_{k+1})-D_{f_1-f}(x_{k+1},x_k) 
        \\
        &+ \frac{1}{2\theta}\| \tilde{x}_k-x \|^2
        - \frac{1}{2\theta}\|\tilde{x}_{k+1}-x\|^2 - \frac{1}{2\theta}\|\tilde{x}_{k+1}-\tilde{x}_k\|^2 + \frac{1}{\theta}\|e_{k+1}\|^2 
        \\
        &+ \frac{1}{4\theta}\|\tilde{x}_{k+1}-\tilde{x}_k\|^2.
    \end{align*}
    Note that $-\|a-b\|^2\geq-\|(a-c)+c-(b-d)-d\|^2\geq-3\|(a-c)-(b-d)\|^2-3\|c\|^2-3\|e\|^2$ and therefore $$-\|\tilde{x}_{k+1}-\tilde{x}_k\|^2\leq -\frac{1}{3}\|x_{k+1}-x_k\|^2 + \|e_{k+1}\|^2 + \|e_k\|^2.$$ Thus, we can rewrite (\ref{proof:convexity}) in the following form:
    \begin{align*}
    \begin{split}
        \E[f(x_{k+1})-f(x)] \leq& \E\left[ D_{f_1-f}(x,x_k) - D_{f_1-f}(x,x_{k+1}) - D_{f_1-f}(x_{k+1},x_k) + \frac{1}{2\theta}\|\tilde{x}_k-x\|^2 \right]
        \\
        &+\E\left[- \frac{1}{2\theta}\|\tilde{x}_{k+1}-x\|^2 - \frac{1}{12\theta}\|x_{k+1}-x_k\|^2 + \frac{5}{4\theta}\|e_{k+1}\|^2 + \frac{1}{4\theta}\|e_k\|^2\right] 
        \\
        &- \E\left[\frac{\mu}{2}\|x_{k+1}-x\|^2 \right]. 
    \end{split}
    \end{align*}    
    Let $k=N-1$. We can change previous inequality:
    \begin{align*}
    \begin{split}
        \E[f(x_{N})-f(x)] \leq& \E\left[ D_{f_1-f}(x,x_{N-1}) - D_{f_1-f}(x,x_{N}) - D_{f_1-f}(x_{N},x_{N-1}) + \frac{1}{2\theta}\|\tilde{x}_{N-1}-x\|^2 \right]
        \\
        &+\E\left[- \frac{1}{2\theta}\|\tilde{x}_{N}-x\|^2 - \frac{1}{12\theta}\|x_{N}-x_{N-1}\|^2 + \frac{5}{4\theta}\|e_{N}\|^2 + \frac{1}{4\theta}\|e_{N-1}\|^2\right] 
        \\
        &- \E\left[\frac{\mu}{2}\|x_{k+1}-x\|^2 \right]. 
    \end{split}
    \end{align*}
    Since $D_h(x_0,x_0)=0$, Lemma \ref{lem:prob} implies $\E\left[D_h(x,x_{N-1})-D_h(x,x_N)\right]=p\E\left[D_h(x,x_{0})-D_h(x,x_N)\right]$, $\frac{1}{2\theta}\|\tilde{x}_{N-1}-x\|^2-\frac{1}{2\theta}\|\tilde{x}_{N}-x\|^2=\frac{p}{2\theta}\|\tilde{x}_{0}-x\|^2-\frac{p}{2\theta}\|\tilde{x}_{N-1}-x\|^2$ and $\E_N\left[\|e_{N-1}\|^2\right]=(1-p)\E_N\left[\|e_{N}\|^2\right]\leq\E_N\left[\|e_{N}\|^2\right]$. Take into account that $e_0=0$, thus $\tilde{x}_0=x_0$, and obtain the following:
    \begin{align}\label{biased_almost_ready_1}
    \begin{split}
        \E[f(x_{N})-f(x)] \leq& \E\left[ pD_{f_1-f}(x,x_{0}) - pD_{f_1-f}(x,x_{N}) + \frac{p}{2\theta}\|x_{0}-x\|^2 - \frac{p}{2\theta}\|\tilde{x}_{N}-x\|^2\right]
        \\
        &+\E\left[- D_{f_1-f}(x_{N},x_{N-1}) - \frac{1}{12\theta}\|x_{N}-x_{N-1}\|^2 + \frac{3}{2\theta}\|e_{N}\|^2- \frac{\mu}{2}\|x_{N}-x\|^2 \right]. 
    \end{split}
    \end{align}
    It is not as easy to work in this setting as with the unbiased compressor, since it is not possible to generate the Bregman divergence by a strongly convex function. For this purpose, we will use a workaround. We use Proposition \ref{three-point} twice: for the divergence generated by $f_1-f$ and $\frac{1}{2\theta}\|x\|^2$. Let us write down the equations of interest:
    \begin{align*}
        pD_{f_1-f}(x,x_0)+pD_{f_1-f}(x_0,x_N)-pD_{f_1-f}(x,x_N)=p\langle x-x_0,\nabla(f_1-f)(x_N)-\nabla(f_1-f)(x_0)\rangle,
    \end{align*}
    \begin{align*}
        \frac{p}{2\theta}\|x_0-x\|^2+\frac{p}{2\theta}\|x_0-\tilde{x}_N\|^2-\frac{p}{2\theta}\|\tilde{x}_N-x\|^2 = \frac{p}{\theta}\langle x-x_0,\tilde{x}_N-x_0 \rangle
    \end{align*}
    Transform \eqref{biased_almost_ready_1} according to the written out expressions and get
    \begin{align}\label{biased_almost_ready_3}
    \begin{split}
        \E[f(x_{N})-f(x)] \leq& \E\left[ p\left\langle x-x_0,\nabla(f_1-f)(x_N)-\nabla(f_1-f)(x_0)+\frac{\tilde{x}_N-x_0}{\theta}\right\rangle -pD_{f_1-f}(x_0,x_N)-\frac{p}{2\theta}\|x_0-\tilde{x}_N\|^2\right]
        \\
        &+\E\left[- D_{f_1-f}(x_{N},x_{N-1}) - \frac{1}{12\theta}\|x_{N}-x_{N-1}\|^2 + \frac{3}{2\theta}\|e_{N}\|^2- \frac{\mu}{2}\|x_{N}-x\|^2 \right]. 
    \end{split}
    \end{align}   
    It is known that the function $f_1-f$ $\delta$-smooth. This means that we have an upper bound $D_{f_1-f}(x,y)\leq\frac{\delta}{2}\|x-y\|^2$ for every $x,y\in\R$. Thus, $\theta\leq\frac{1}{6\delta}$ is sufficient to fulfill the inequality $- D_{f_1-f}(x_{N},x_{N-1}) - \frac{1}{12\theta}\|x_{N}-x_{N-1}\|^2\leq0$. For $-pD_{f_1-f}(x_0,x_N)-\frac{p}{2\theta}\|x_0-\tilde{x}_N\|^2$ we perform a more careful analysis.
    \begin{align*}
        -pD_{f_1-f}(x_0,x_N)-\frac{p}{2\theta}\|x_0-\tilde{x}_N\|^2 \leq& -pD_{f_1-f}(x_0,x_N)-\frac{p}{4\theta}\|x_0-x_N\|^2 + \frac{1}{2\theta}\|e_N\|^2 \\\leq& \frac{p\delta}{2}\|x_0-x_N\|^2-\frac{p}{4\theta}\|x_0-x_N\|^2 + \frac{1}{2\theta}\|e_N\|^2.
    \end{align*}
    $\theta\leq\frac{1}{6\delta}$. It follows that we can estimate $\delta\leq\frac{1}{6\theta}$. Thus,
    \begin{align*}
        -pD_{f_1-f}(x_0,x_N)-\frac{p}{2\theta}\|x_0-\tilde{x}_N\|^2 \leq&  \frac{1}{2\theta}\|e_N\|^2-\frac{p}{6\theta}\|x_0-x_N\|^2.
    \end{align*}
    Using this facts, rewrite \eqref{biased_almost_ready_3}:
    \begin{align}\label{biased_almost_ready_4}
    \begin{split}
        \E[f(x_{N})-f(x)] \leq& \E\left[ p\left\langle x-x_0,\nabla(f_1-f)(x_N)-\nabla(f_1-f)(x_0)+\frac{\tilde{x}_N-x_0}{\theta}\right\rangle -\frac{p}{6\theta}\|x_0-x_N\|^2\right]
        \\
        &+\E\left[\frac{2}{\theta}\|e_{N}\|^2- \frac{\mu}{2}\|x_{N}-x\|^2 \right]. 
    \end{split}
    \end{align}
    Now let us deal with the "error' term $\|e_{k+1}\|^2$. Firstly, we use Definition \ref{biased_def} and write
        \begin{align*}
        \E[\|e_{k+1}\|^2] \leq& \E\left[\frac{1}{M}\sum_{m=1}^M\|e_{k+1}^m\|^2\right] \leq \frac{1}{M}\sum_{m=1}^M \left( 1-\frac{1}{\beta} \right)\E\left[\|e_k^m + \theta[\nabla f_m(x_k)-\nabla f_1(x_k) - \nabla f_m(x_0) + \nabla f_1(x_0)] \|^2\right].
        \end{align*}
        Next, we ommit $(1-\frac{1}{\beta})$-factor and use the Cauchy-Schwarz inequality:
        \begin{align*}
        \E[\|e_{k+1}\|^2] \leq& (1+c)\E\left[\frac{1}{M}\sum_{m=1}^M\|e_k^m\|^2\right] 
        \\
        &+ \left(1+\frac{1}{c}\right)\theta^2\frac{1}{M}\sum_{m=1}^M\E\left[\|\nabla(f_1-f_m)(x_0)-\nabla(f_1-f_m)(x_{k})\|^2\right].
        \end{align*}
    Next, we enroll the recursion and take into account the fact that $e_0^m=0$. Thus, we obtain
    \begin{align*}
        \E[\|e_{k+1}\|^2] \leq& 4\theta^2\left(1+\frac{1}{c}\right)\sum_{j=1}^{k}(1+c)^{k-j}\E\left[\|\nabla(f_1-f_m)(x_0)-\nabla(f_1-f_m)(x_{j})\|^2\right].
    \end{align*}
    $\delta$-smoothness (see \eqref{sim_cor}) of $f_m-f$ gives
    \begin{align*}
        \E[\|e_{k+1}\|^2] \leq& 4\theta^2\delta^2\left(1+\frac{1}{c}\right)\sum_{j=1}^{k}(1+c)^{k-j}\E\left[\|x_j-x_0\|^2\right].
    \end{align*}
    The number of iterations of the Algorithm \ref{alg:biased_compr_epoch} $N$ is a random variable. Let us calculate the expectation on the action of the compressive operator and on the random variable $N$ at once. We are specifically interested in the “error” at the last iteration of the Algorithm \ref{alg:biased_compr_epoch}. Since $N\in\text{Geom}(p)$, we obtain 
    \begin{align*}
        \E_{C,N}[\|e_N\|^2] =& \E_{C}\left[\sum_{k\geq0}p(1-p)^k\|e_k\|^2\right]
        \\
        \leq& 4\theta^2\delta^2\left(1+\frac{1}{c}\right)\sum_{k\geq0}p(1-p)^k\sum_{j=1}^{k-1}(1+c)^{k-j}\|x_j-x_0\|^2.
    \end{align*}
    Choose $c=\frac{p}{2}$:
    \begin{align*}
        \sum_{k\geq0}p(1-p)^k\sum_{j=1}^{k-1}(1+c)^{k-j}\|x_j-x_0\|^2 =& p[\{(1-p)^2(1+c) + (1-p)^3(1+c)^2+...\}\|x_1-x_0\|^2
        \\
        &+ \{(1-p)^3(1+c) + (1-p)^4(1+c)^2+...\}\|x_2-x_0\|^2
        \\
        &+...]
        \\
        \leq& p\left[\frac{2}{p}(1-p)\|x_1-x_0\|^2
        + \frac{2}{p}(1-p)^2\|x_2-x_0\|^2+...\right]
        \\
        =& \frac{2}{p}\E_N[\|x_N-x_0\|^2].
    \end{align*}
    Note that $\left(1+\frac{1}{c}\right)\leq\frac{3}{p}$. Thus, we obtain
    \begin{align}\label{lem:error_estimation}
        \E[\|e_{N}\|^2] \leq \frac{24\theta^2\delta^2}{p^2}\E_N[\|x_N-x_0\|^2].
    \end{align}
    Substitute into \eqref{biased_almost_ready_4}:
    \begin{align}\label{biased_almost_ready_5}
    \begin{split}
        \E[f(x_{N})-f(x)] \leq& \E\left[ p\left\langle x-x_0,\nabla(f_1-f)(x_N)-\nabla(f_1-f)(x_0)+\frac{\tilde{x}_N-x_0}{\theta}\right\rangle -\frac{p}{6\theta}\|x_0-x_N\|^2\right]
        \\
        &+\E\left[\frac{48\theta\delta^2}{p^2}\|x_N-x_0\|^2 - \frac{\mu}{2}\|x_{N}-x\|^2 \right]. 
    \end{split}
    \end{align}
    Since $f_m-f$ is $\delta$-smooth (see \eqref{sim_cor}), one can note that $h(x)$ is $\left(\frac{1}{\theta}-\delta\right)$-strongly convex for $\theta\leq\nicefrac{1}{\delta}$. Our choice $\theta\leq\nicefrac{1}{6\delta}$ is appropriate. Moreover, $h(x)$ is $\left(\delta+\frac{1}{\theta}\right)$-smooth. \eqref{strong_convexity} gives $D_h(x,y)\geq\frac{1-\theta\delta}{2\theta}\|x-y\|^2$. Proposition \ref{nest_smooth} gives $D_{h}(x,y)\leq\frac{1+\theta\delta}{2\theta}\|x-y\|^2$. Let us write it down in a more convenient form:
    \begin{align}\label{lem:new_div_estimation}
        0\leq\frac{1-\theta\delta}{2\theta}\|x-y\|^2\leq D_{h}(x,y)\leq\frac{1+\theta\delta}{2\theta}\|x-y\|^2.
    \end{align}
    Choose $\theta\leq\frac{p^{\nicefrac{3}{2}}}{24\delta}$. With such a choice we have
    \begin{align*}
        \left(\frac{48\theta\delta^2}{p^2}-\frac{p}{6\theta}\right)\|x_N-x_0\|^2\leq&-\frac{p}{12}\|x_N-x_0\|^2\leq-\frac{2\theta}{1+\theta\delta}\frac{p}{12}D_h(x_0,x_N)\leq-\frac{p}{7}D_h(x_0,x_N).
    \end{align*}
    Finally, substituting it into \eqref{biased_almost_ready_5} we obtain
    \begin{align*}
    \begin{split}
        \E[f(x_{N})-f(x)] \leq& \E\left[ p\left\langle x-x_0,\nabla(f_1-f)(x_N)-\nabla(f_1-f)(x_0)+\frac{\tilde{x}_N-x_0}{\theta}\right\rangle -\frac{p}{7}D_h(x_0,x_N)-\frac{\mu}{2}\|x_{N}-x\|^2 \right]. 
    \end{split}
    \end{align*}
\end{proof}

\section{Proof of Theorem 2}
\begin{theorem}\textbf{(Theorem \ref{th:biased})}
    Let the problem (\ref{prob_form}) be solved by Algorithm \ref{alg:biased_compr} with $\theta\leq\frac{p^{\nicefrac{3}{2}}}{24\delta}\leq\frac{1}{6\delta}$ and such tuning of parameters that $28\alpha p \tau \leq\theta$. Then the following inequality holds:
    \begin{align*}
        \E\left[ Y_{k+1} + Z_{k+1} \right] \leq& \E\left[ (1-\tau)Y_k + (1+\mu\alpha)^{-1}Z_k  \right].
    \end{align*}
\end{theorem}
\begin{proof}\label{proof_th_biased}
    Let us move from one epoch to the method as a whole. Since the output point of Algorithm \ref{alg:biased_compr_epoch} was used to calculate $y_{k+1}$, let us re-designate $x_0=x_{k+1}$ and $x_N=y_{k+1}$ in Lemma \ref{lem:2}:
    \begin{align*}
    \begin{split}
        \E[f(y_{k+1})-f(x)] \leq& \E\left[ p\left\langle x-x_{k+1},\nabla(f_1-f)(y_{k+1})-\nabla(f_1-f)(x_{k+1})+\frac{\tilde{y}_{k+1}-x_{k+1}}{\theta}\right\rangle -\frac{p}{7}D_h(x_{k+1},y_{k+1})\right]\\&-\E\left[\frac{\mu}{2}\|y_{k+1}-x\|^2 \right]. 
    \end{split}
    \end{align*}
    See that $G_{k+1}$ from Line \ref{biased:10} is almost the same as $\nabla(f_1-f)(y_{k+1})-\nabla(f_1-f)(x_{k+1})+\frac{\tilde{y}_{k+1}-x_{k+1}}{\theta}$ in the expression above. Namely, $\nabla(f_1-f)(y_{k+1})-\nabla(f_1-f)(x_{k+1})+\frac{\tilde{y}_{k+1}-x_{k+1}}{\theta}=-G_{k+1}$.  
    \begin{align}\label{qwert}
        \E[f(y_{k+1})-f(x)] \leq& \E\left[ \langle x-x_{k+1},-G_{k+1} \rangle -\frac{p}{7}D_h(x_{k+1},y_{k+1})-\frac{\mu}{2}\|y_{k+1}-x\|^2 \right]\notag 
        \\
        =& \E\left[ \langle z_k-x_{k+1},-G_{k+1} \rangle +  \langle x-z_k,-G_{k+1} \rangle -\frac{p}{7}D_h(x_{k+1},y_{k+1})\right]\notag
        \\
        &-\E\left[\frac{\mu}{2}\|y_{k+1}-x\|^2 \right]\notag
        \\
        =& \E\left[ \langle z_k-x_{k+1},-G_{k+1} \rangle +  \langle z_k-x,G_{k+1} \rangle -\frac{p}{7}D_h(x_{k+1},y_{k+1})\right]\notag
        \\
        &-\E\left[\frac{\mu}{2}\|y_{k+1}-x\|^2 \right].
    \end{align}
    Let us rewrite Line \ref{biased:3} of Algorithm \ref{alg:biased_compr}
    \begin{align*}
        (1-\tau)x_{k+1}=\tau(z_k-x_{k+1})+(1-\tau)y_k\Rightarrow z_k-x_{k+1}=\frac{1-\tau}{\tau}[x_{k+1}-y_k].
    \end{align*}
    and substitute it into (\ref{qwert}):
    \begin{align*}
    \begin{split}
        \E[f(y_{k+1})-f(x)] \leq& \E\left[ \frac{1-\tau}{\tau}\langle x_{k+1}-y_k,-G_{k+1} \rangle +  \langle z_k-x,G_{k+1} \rangle -\frac{p}{2}D_h(x_{k+1},y_{k+1})\right]\notag
        \\
        &-\E\left[\frac{\mu}{2}\|y_{k+1}-x\|^2 \right].
    \end{split}
    \end{align*}
    Next, apply Lemma \ref{lem:subtask}:
    \begin{align*}
        \langle z_k-x,G_{k+1}\rangle-\frac{\mu}{2}\|y_{k+1}-x\|^2\leq&\langle z_k-x,G_{k+1}\rangle-\frac{\mu}{2}\|y_{k+1}-x\|^2+\frac{\mu}{2}\|y_{k+1}-z_{k+1}\|^2\\\leq&\frac{\alpha}{2}\|G_{k+1}\|^2+\frac{1}{2\alpha}\| z_k-x \|^2 - \frac{1+\mu\alpha}{2\alpha}\|z_{k+1}-x\|^2.
    \end{align*}
    Combining the obtained results, we have
    \begin{align}\label{biased_almost_ready_6}
    \begin{split}
        \E[f(y_{k+1})-f(x)] \leq& \E\left[ \frac{1-\tau}{\tau}\langle x_{k+1}-y_k,-G_{k+1} \rangle -\frac{p}{2}D_h(x_{k+1},y_{k+1})\right]
        \\
        &+\E\left[-\frac{\mu}{2}\|y_{k+1}-x\|^2 +\frac{\alpha}{2}\|\tilde{G}_{k+1}\|^2\right]\\
        &+ \E\left[ \frac{1}{2\alpha}\| z_k-x \|^2 - \frac{1+\mu\alpha}{2\alpha}\|z_{k+1}-x\|^2 \right].
    \end{split}
    \end{align}
    Let us write Lemma \ref{lem:2} with $x=y_k$:
    \begin{align*}
    \begin{split}
        \E[f(y_{k+1})-f(y_k)] \leq& \E\left[ p\left\langle y_k-x_{k+1},-G_{k+1}\right\rangle -\frac{p}{7}D_h(x_{k+1},y_{k+1})\right]\\&-\E\left[\frac{\mu}{2}\|y_{k+1}-y_k\|^2 \right]. 
    \end{split}
    \end{align*}
    Rewrite it in the following form
    \begin{align*}
    \begin{split}
        \E[p\left\langle x_{k+1}-y_k,-G_{k+1}\right\rangle] \leq& \E\left[ (f(y_{k+1})-f(y_k)) -\frac{p}{7}D_h(x_{k+1},y_{k+1})\right]\\&-\E\left[\frac{\mu}{2}\|y_{k+1}-y_k\|^2 \right]. 
    \end{split}
    \end{align*}
    and substitute into \eqref{biased_almost_ready_6}. We get
    \begin{align*}
        \E\left[\frac{1}{\tau}(f(y_{k+1})-f(x))\right] \leq& \E\left[ \frac{1-\tau}{\tau}(f(y_k)-f(x)) - \frac{p}{7\tau}D_h(x_{k+1},y_{k+1}) + \frac{\alpha}{2}\|G_{k+1}\|^2  \right]
        \\
        &+ \E\left[ \frac{1}{2\alpha}\| z_k-x \|^2 - \frac{1+\mu\alpha}{2\alpha}\|z_{k+1}-x\|^2 \right].
    \end{align*}
    The final hurdle to proof is the need to evaluate $\|G_{k+1}\|^2$. The key idea is to put the "error" term out of $\|G_{k+1}\|^2$ and work with two terms.
    \begin{align*}
        \frac{\alpha}{2}\|G_{k+1}\|^2 \leq& \frac{\alpha p^2}{\theta^2}\|e_N\|^2 + \alpha p^2\|\nabla h(x_{k+1})-\nabla h(y_{k+1})\|^2.
    \end{align*}
    Let us estimate it by parts. 
    \begin{align*}
        \frac{\alpha p^2}{\theta^2}\|e_N\|^2 \leq& \frac{\alpha p^2}{\theta^2}\frac{24\theta^2\delta^2}{p^2}\frac{2\theta}{1-\theta\delta}D_h(x_{k+1},y_{k+1}) \leq \frac{288\alpha\theta\delta^2}{5}D_h(x_{k+1},y_{k+1}).
    \end{align*}
    Since $\theta\leq\frac{p^{\nicefrac{3}{2}}}{24\delta}$, we can estimate $\delta\leq\frac{p^{\nicefrac{3}{2}}}{24\theta}$ and obtain:
    \begin{align*}
        \frac{\alpha p^2}{\theta^2}\|e_N\|^2\leq\frac{\alpha p^3}{10\theta}D_h(x_{k+1},y_{k+1}).
    \end{align*}
    Let us move on to the second term:
    \begin{align*}
         \alpha p^2\|\nabla h(x_{k+1})-\nabla h(y_{k+1})\|^2 \leq \alpha p^2\frac{2(1+\theta\delta)}{\theta}D_h(x_{k+1},y_{k+1})\leq\frac{3\alpha p^2}{\theta}D_h(x_{k+1},y_{k+1}).
    \end{align*}
    Here we used (\ref{lem:error_estimation}), (\ref{lem:new_div_estimation}) and smoothness of $h(x)$. Put two terms together and get
    \begin{align*}
        \frac{\alpha}{2}\|G_{k+1}\|^2\leq\frac{4\alpha p^2}{\theta}D_h(x_{k+1},y_{k+1}).
    \end{align*}
    \begin{align*}
        \E\left[\frac{1}{\tau}(f(y_{k+1})-f(x))\right] \leq& \E\left[ \frac{1-\tau}{\tau}(f(y_k)-f(x)) + p\left(\frac{4\alpha p}{\theta} - \frac{1}{7\tau} \right)D_h(x_{k+1},y_{k+1})  \right]
        \\
        &+ \E\left[\frac{1}{2\alpha}\| z_k-x \|^2- \frac{1+\mu\alpha}{2\alpha}\|z_{k+1}-x\|^2 \right].
    \end{align*}
    Note that $28\alpha p\tau\leq\theta$ because of our choice of parameters. We obtain
    \begin{align*}
        \E\left[\frac{1}{\tau}(f(y_{k+1})-f(x))+\frac{1+\mu\alpha}{2\alpha}\|z_{k+1}-x\|^2\right] \leq& \E\left[ \frac{1-\tau}{\tau}(f(y_k)-f(x) + \frac{1}{2\alpha}\| z_k-x \|^2\right].
    \end{align*}
\end{proof}

\section{Additional Experiments} \label{app:add_exp}
This section presents more runs of \texttt{OLGA}. First, we solve the problems \eqref{eq:quadr} and \eqref{eq:logloss} using the \texttt{mushrooms} dataset \citep{chang2011libsvm} (Figures \ref{fig:begin_mushrooms}-\ref{fig:end_mushrooms}). This allows us to further show the robustness of the method to varying the parameters $\mu$, $L$, $\delta$ over a wider range. Second, we present measurements of the training time of the distributed models (Figures \ref{fig:begin_cluster}-\ref{fig:end_cluster}). Computation time depends on system we work with. Therefore, we provide two runs of experiments: on local cluster with cable connection (fast connect), on remote CPUs with Internet connection (slow connect). Note that compression plays a more essential role in the first case.

\section{Future Extensions}
\label{Future Extensions}
The methods we propose could be widely extended. Next we provide a few observations in this regard. In modern distributed learning, local data is often private and should not be compromised. As a result, one cannot use the full gradient, as data could be recovered \citep{weeraddana2017privacy}. This refers to the very popular federated learning setting. Since compression is a kind of stochasticity, we can add local additive noise to \texttt{OLGA} or \texttt{EF-OLGA} to preserve differential privacy \citep{nozari2016differentially}. We go into more detail below. Let us start with Algorithm \ref{alg:unbiased_compr_epoch}. There are two options: to noise gradient before or after compression. The first means the use of
\begin{align*}
    \nabla f_m(x,\xi^m)=\nabla f_m(x)+\xi^m,~\text{where }\xi^m\sim N(0,\sigma_m^2)
\end{align*}
Then in Line \ref{unbiased:line_7}, we have 
\begin{align*}
    \widehat{g}_k^m=\nabla f_m(x_k)-\nabla f_1(x_k) - \nabla f_m(x_0) + \nabla f_1(x_0)+\xi_k^m-\xi_0^m,
\end{align*}
and Line \ref{unbiased:line_11} is 
\begin{align*}
    t_k=g_k-\nabla f_1(x_0)+\nabla f(x_0)+\frac{1}{M}\sum_{m=1}^M\xi_0^m.
\end{align*}
The second option involves adding noise to compressed gradient. This means using \begin{align*}
    g_k^m=Q\left( \widehat{g}_k^m \right) + \xi_k^m
\end{align*}
in Line \ref{unbiased:line_8}. In Algorithm \ref{alg:unbiased_compr} is an accelerating framework that does not use compression. Therefore, there is only one option of noising available:
\begin{align*}
    \nabla f(y_{k+1})=\frac{1} {M}\sum_{m=1}^M\nabla f_m(y_{k+1})+\frac{1} {M}\sum_{m=1}^M\xi_{k+1}^m
\end{align*}
in Line \ref{unbiased_full:line_6}. Algorithms \ref{alg:biased_compr_epoch} and \ref{alg:biased_compr} differ only in their handling of the compression error, and could be modified in a similar way. Proofs in the case of noisy gradients repeats our ones with addition of well-known techniques \citep{bylinkin2024accelerated}.

We note again that variance reduction helps us to design methods with compression. The same approach is used to implement client sampling. This means that this technique could also be easily exploited in our algorithms. In Algorithm \ref{alg:unbiased_compr_epoch}, Line \ref{unbiased:line_6} should be changed to 
\begin{align*}
    \text{For chosen } m_{i_k}\sim U[1,M].
\end{align*} 
Then in Algorithm \ref{alg:unbiased_compr} the gradient difference is approximated by
\begin{align*}
    t_k=\nabla(f_1-f_{m_{i_k}})(x_{k+1}) - \nabla(f_1-f_{m_{i_k}})(y_{k+1}),~m_{i_k}\sim U[1,M].
\end{align*}
The same can be done in Algorithms \ref{alg:biased_compr_epoch} and \ref{alg:biased_compr}. This approach opens the door to asynchronous setting, where the communication channels between devices and servers cannot be opened simultaneously.

\section{Reproducibility Checklist}

    This paper:

    $\bullet$ Includes a conceptual outline and/or pseudocode description of AI methods introduced. \textbf{Yes}
    
    $\bullet$ Clearly delineates statements that are opinions, hypothesis, and speculation from objective facts and results. \textbf{Yes}

    $\bullet$ Provides well marked pedagogical references for less-familiare readers to gain background necessary to replicate the paper. \textbf{Yes}

    Does this paper make theoretical contributions? \textbf{Yes}

    $\bullet$ All assumptions and restrictions are stated clearly and formally. \textbf{Yes}
    
    $\bullet$ All novel claims are stated formally (e.g., in theorem statements). \textbf{Yes}
    
    $\bullet$ Proofs of all novel claims are included. \textbf{Yes}
    
    $\bullet$ Proof sketches or intuitions are given for complex and/or novel results. \textbf{Partial}
    
    $\bullet$ Appropriate citations to theoretical tools used are given. \textbf{Yes}
    
    $\bullet$ All theoretical claims are demonstrated empirically to hold. \textbf{Yes}
    
    $\bullet$ All experimental code used to eliminate or disprove claims is included. \textbf{No}

    Does this paper rely on one or more datasets? \textbf{Yes}

$\bullet$ A motivation is given for why the experiments are conducted on the selected datasets \textbf{No}, classical datasets for the theory validation.

$\bullet$ All novel datasets introduced in this paper are included in a data appendix. \textbf{NA}, no new data

$\bullet$ All novel datasets introduced in this paper will be made publicly available upon publication of the paper with a license that allows free usage for research purposes. \textbf{NA}, no new data

$\bullet$ All datasets drawn from the existing literature (potentially including authors’ own previously published work) are accompanied by appropriate citations. \textbf{Yes}

$\bullet$ All datasets drawn from the existing literature (potentially including authors’ own previously published work) are publicly available. \textbf{Yes}

$\bullet$ All datasets that are not publicly available are described in detail, with explanation why publicly available alternatives are not scientifically satisficing. \textbf{NA}, no non-public data

    Does this paper include computational experiments? \textbf{Yes}

    $\bullet$ Any code required for pre-processing data is included in the appendix. \textbf{No}, simple experiments without pre-processing.
    
    $\bullet$ All source code required for conducting and analyzing the experiments is included in a code appendix. \textbf{No}, simple experiments - easy to reproduce.
    
    $\bullet$ All source code required for conducting and analyzing the experiments will be made publicly available upon publication of the paper with a license that allows free usage for research purposes. \textbf{No}, simple experiments - easy to reproduce.
    
    $\bullet$ All source code implementing new methods have comments detailing the implementation, with references to the paper where each step comes from. \textbf{No}, simple experiments - easy to reproduce.
    
    $\bullet$ If an algorithm depends on randomness, then the method used for setting seeds is described in a way sufficient to allow replication of results. \textbf{No}, simple experiments - easy to reproduce.
    
    $\bullet$ This paper specifies the computing infrastructure used for running experiments (hardware and software), including GPU/CPU models; amount of memory; operating system; names and versions of relevant software libraries and frameworks. \textbf{No}, simple experiments - easy to reproduce on laptops.
    
    $\bullet$ This paper formally describes evaluation metrics used and explains the motivation for choosing these metrics. \textbf{No}, simple experiments - easy to reproduce.
    
    $\bullet$ This paper states the number of algorithm runs used to compute each reported result. \textbf{No}, simple experiments - easy to reproduce.
    
    $\bullet$ Analysis of experiments goes beyond single-dimensional summaries of performance (e.g., average; median) to include measures of variation, confidence, or other distributional information. \textbf{No}, simple experiments - easy to reproduce.
    
    $\bullet$ The significance of any improvement or decrease in performance is judged using appropriate statistical tests (e.g., Wilcoxon signed-rank). \textbf{No}, simple experiments - easy to reproduce.
    
    $\bullet$ This paper lists all final (hyper-)parameters used for each model/algorithm in the paper’s experiments. \textbf{No}, simple experiments - easy to reproduce.
    
    $\bullet$ This paper states the number and range of values tried per (hyper-) parameter during development of the paper, along with the criterion used for selecting the final parameter setting. \textbf{No}, simple experiments - easy to reproduce.

\begin{figure}[h!] 
   \begin{subfigure}{0.310\textwidth}
       \includegraphics[width=\linewidth]{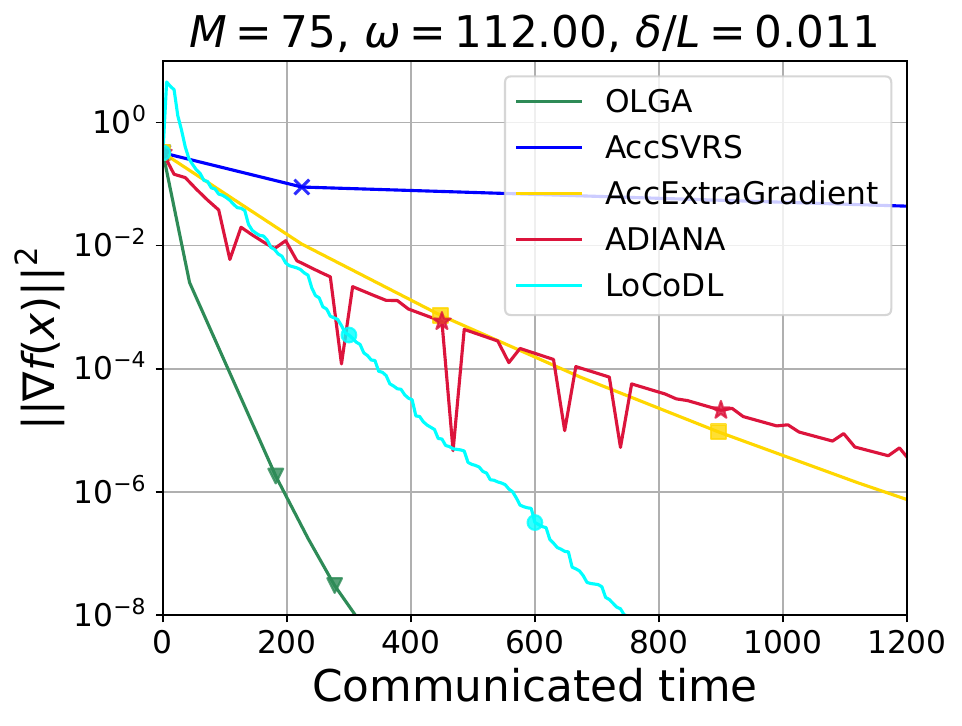}
   \end{subfigure}
   \begin{subfigure}{0.310\textwidth}
       \includegraphics[width=\linewidth]{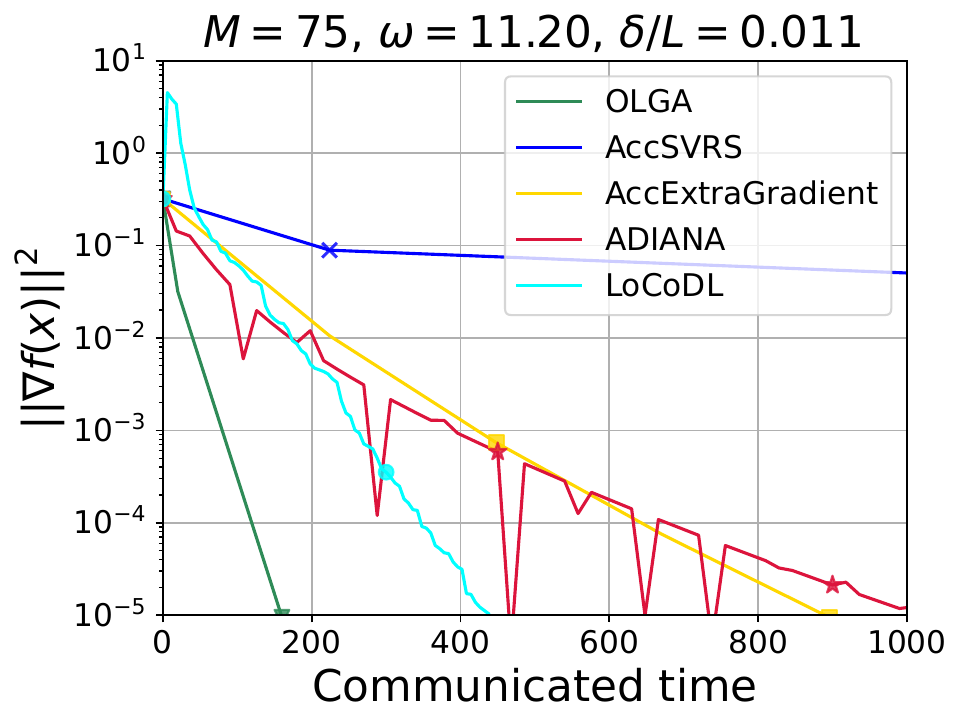}
   \end{subfigure}
   \begin{subfigure}{0.310\textwidth}
       \includegraphics[width=\linewidth]{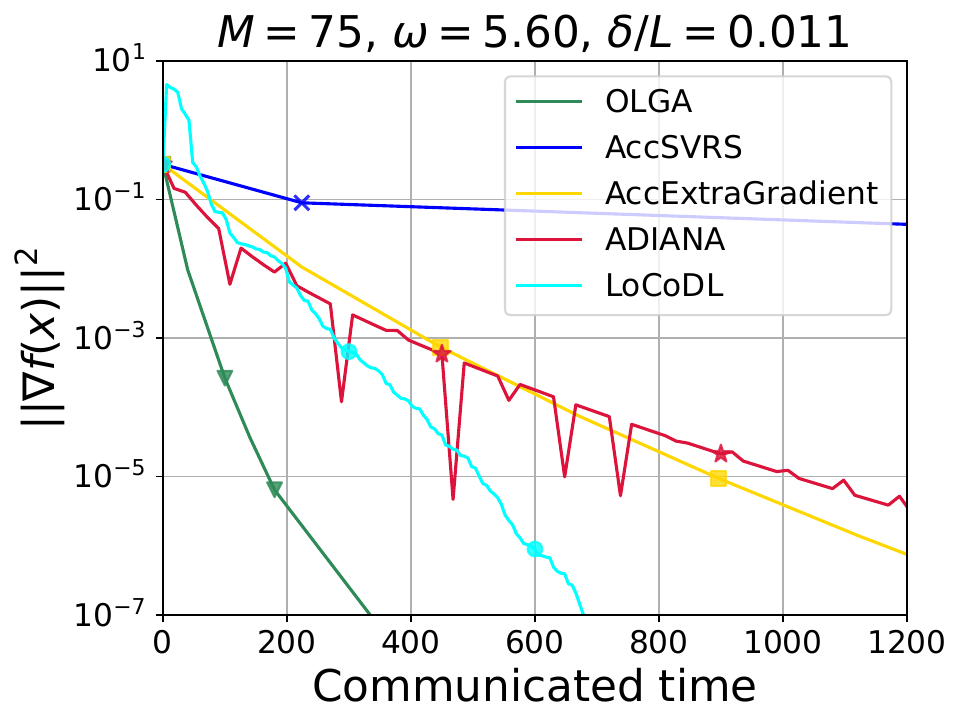}
   \end{subfigure}\\
   \begin{subfigure}{0.310\textwidth}
       \centering
       (a) $\omega = 112.00$
   \end{subfigure}
   \begin{subfigure}{0.310\textwidth}
       \centering
       (b) $\omega = 11.20$
   \end{subfigure}
   \begin{subfigure}{0.310\textwidth}
       \centering
       (c) $\omega = 5.6$
   \end{subfigure}
   \caption{Comparison of state-of-the-art distributed methods. The comparison is made on \eqref{eq:logloss} with $M=75$ and \texttt{mushrooms} dataset. The criterion is the communication time (\textit{CC-3}). For methods with compression we vary the power of compression $\omega$.}
   \label{fig:begin_mushrooms}
\end{figure}

\begin{figure}[h!] 
   \begin{subfigure}{0.310\textwidth}
       \includegraphics[width=\linewidth]{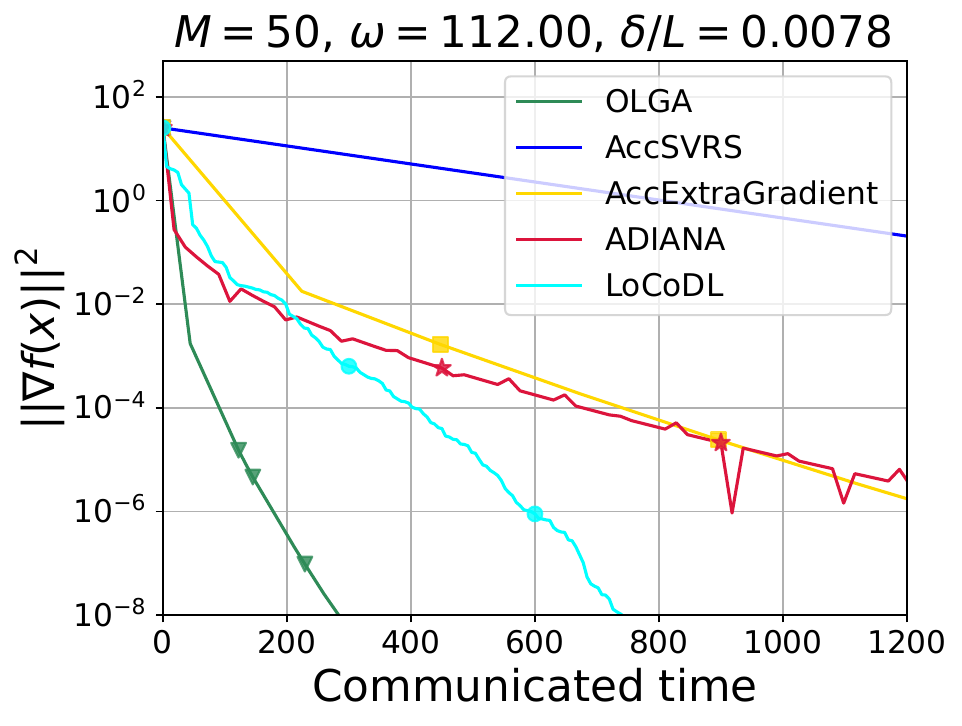}
   \end{subfigure}
   \begin{subfigure}{0.310\textwidth}
       \includegraphics[width=\linewidth]{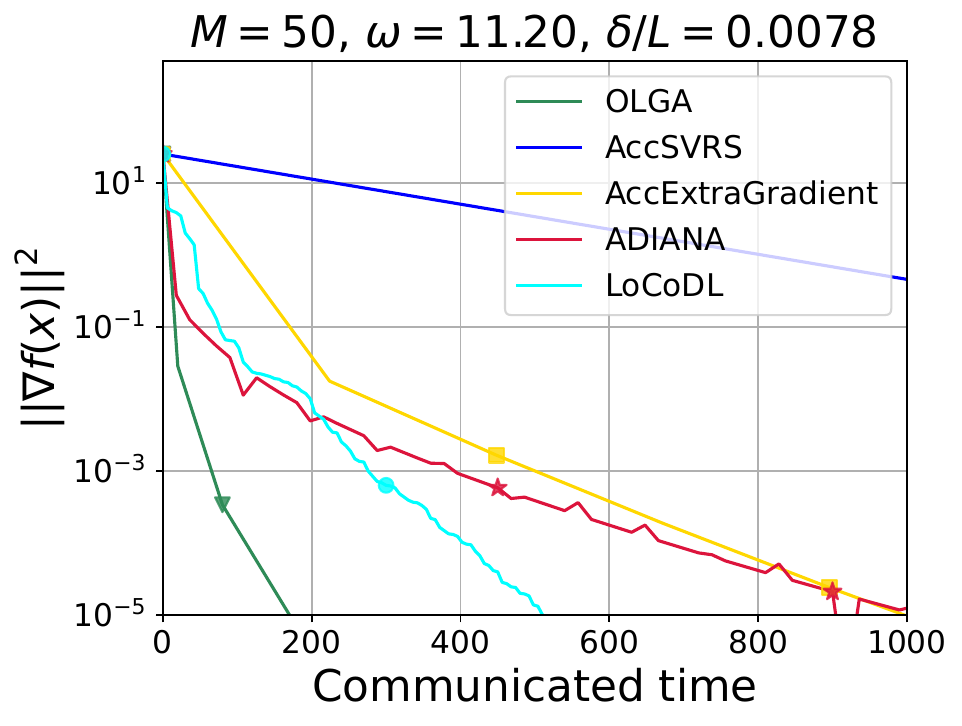}
   \end{subfigure}
   \begin{subfigure}{0.310\textwidth}
       \includegraphics[width=\linewidth]{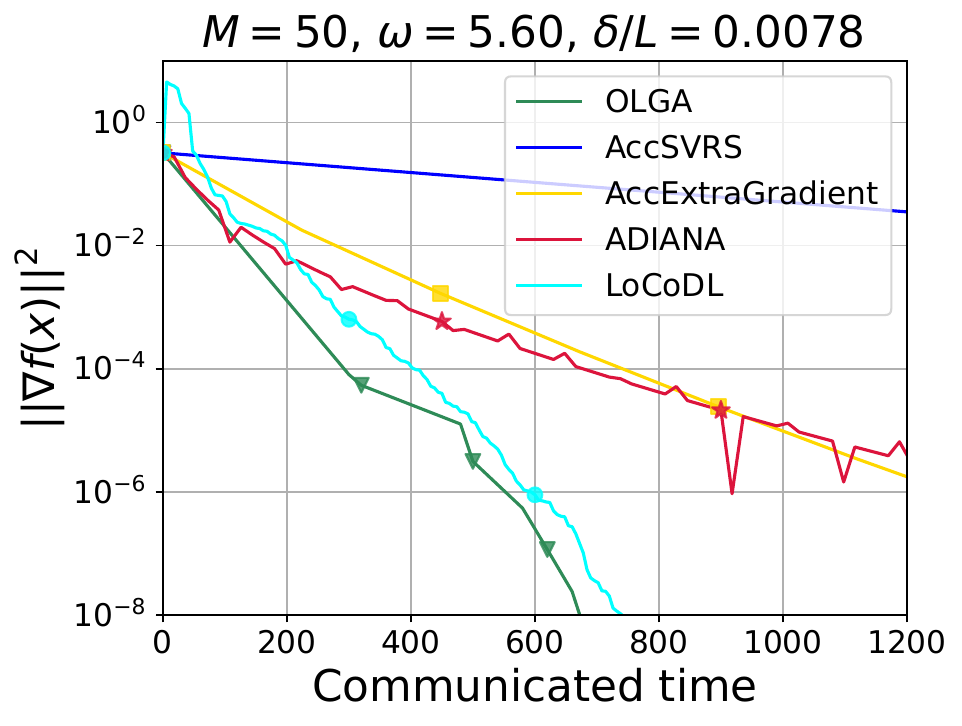}
   \end{subfigure}\\
   \begin{subfigure}{0.310\textwidth}
       \centering
       (a) $\omega = 112.00$
   \end{subfigure}
   \begin{subfigure}{0.310\textwidth}
       \centering
       (b) $\omega = 11.20$
   \end{subfigure}
   \begin{subfigure}{0.310\textwidth}
       \centering
       (c) $\omega = 5.6$
   \end{subfigure}
   \caption{Comparison of state-of-the-art distributed methods. The comparison is made on \eqref{eq:logloss} with $M=50$ and \texttt{mushrooms} dataset. The criterion is the communication time (\textit{CC-3}). For methods with compression we vary the power of compression $\omega$.}
\end{figure}

\begin{figure}[h!] 
   \begin{subfigure}{0.310\textwidth}
       \includegraphics[width=\linewidth]{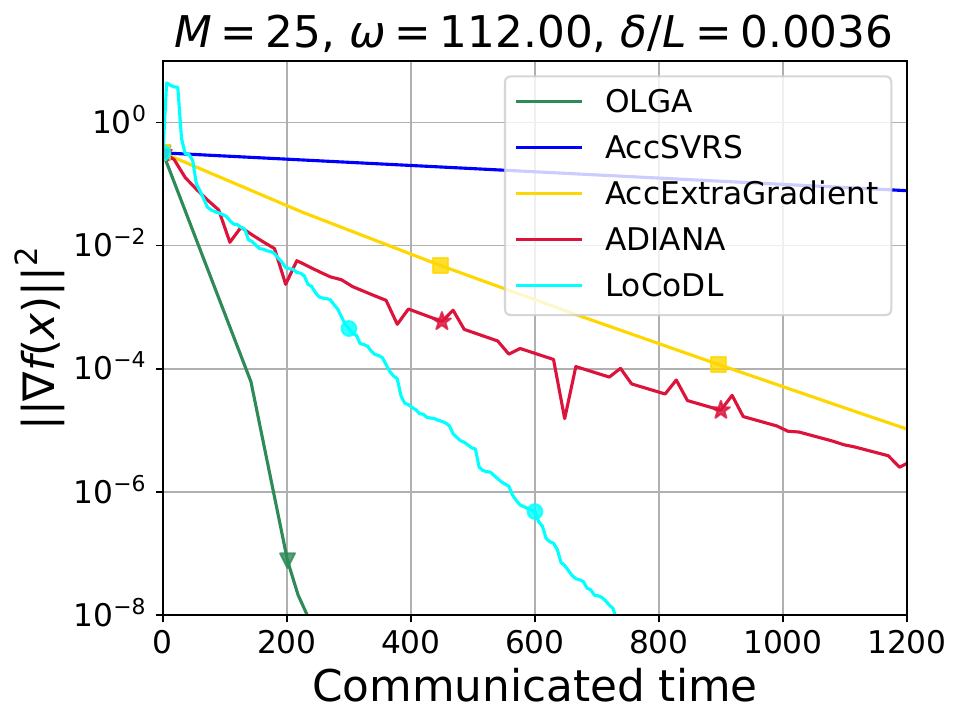}
   \end{subfigure}
   \begin{subfigure}{0.310\textwidth}
       \includegraphics[width=\linewidth]{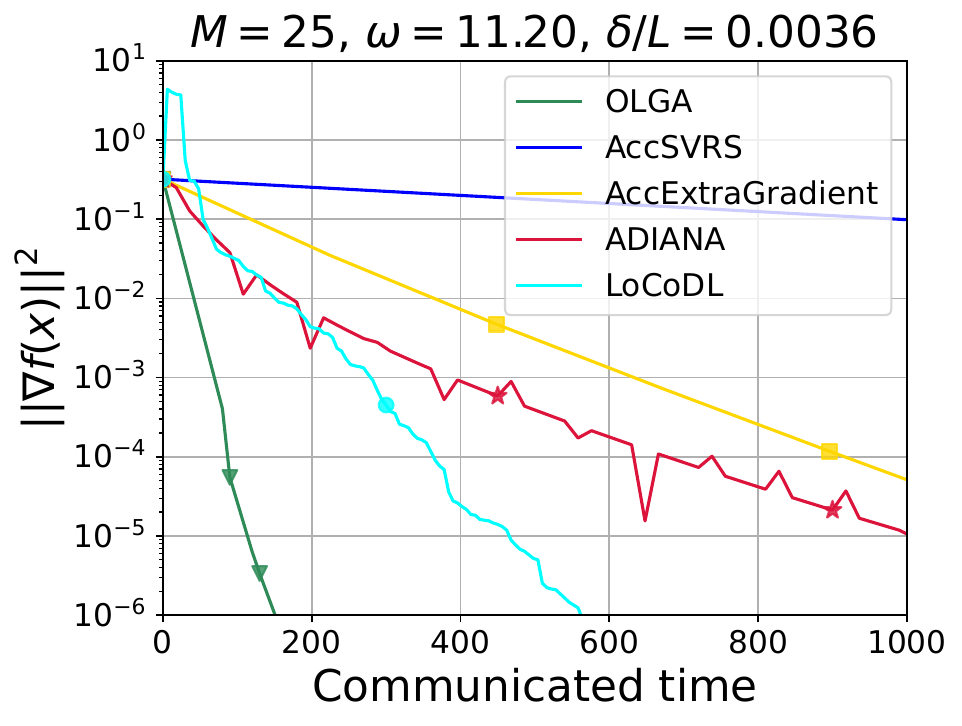}
   \end{subfigure}
   \begin{subfigure}{0.310\textwidth}
       \includegraphics[width=\linewidth]{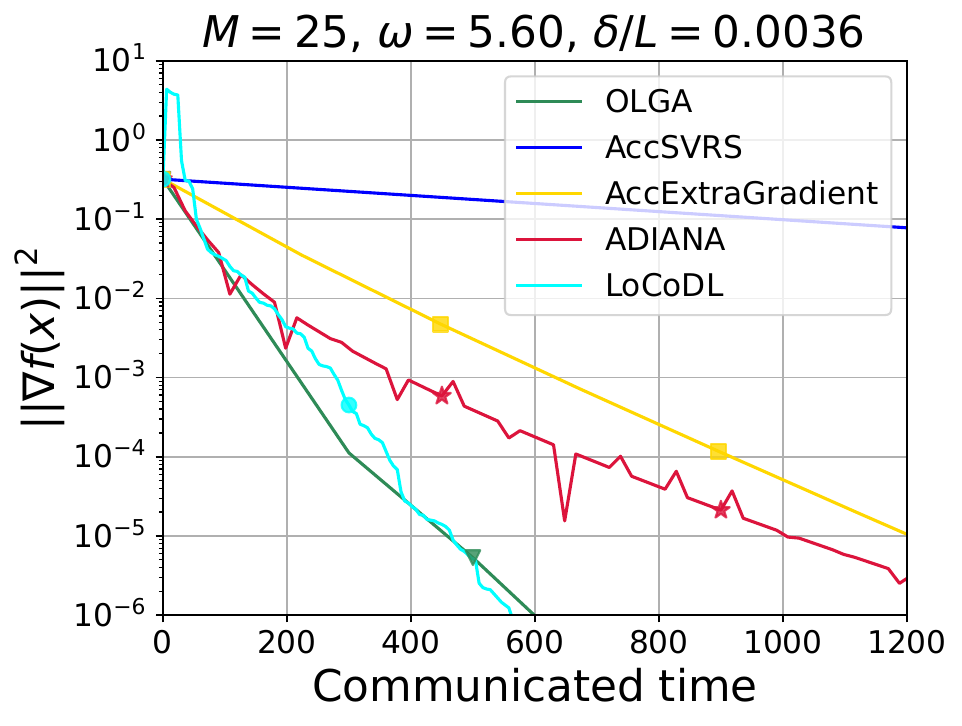}
   \end{subfigure}\\
   \begin{subfigure}{0.310\textwidth}
       \centering
       (a) $\omega = 112.00$
   \end{subfigure}
   \begin{subfigure}{0.310\textwidth}
       \centering
       (b) $\omega = 11.20$
   \end{subfigure}
   \begin{subfigure}{0.310\textwidth}
       \centering
       (c) $\omega = 5.6$
   \end{subfigure}
   \caption{Comparison of state-of-the-art distributed methods. The comparison is made on \eqref{eq:logloss} with $M=25$ and \texttt{mushrooms} dataset. The criterion is the communication time (\textit{CC-3}). For methods with compression we vary the power of compression $\omega$.}
\end{figure}

\begin{figure}[h!] 
   \begin{subfigure}{0.310\textwidth}
       \includegraphics[width=\linewidth]{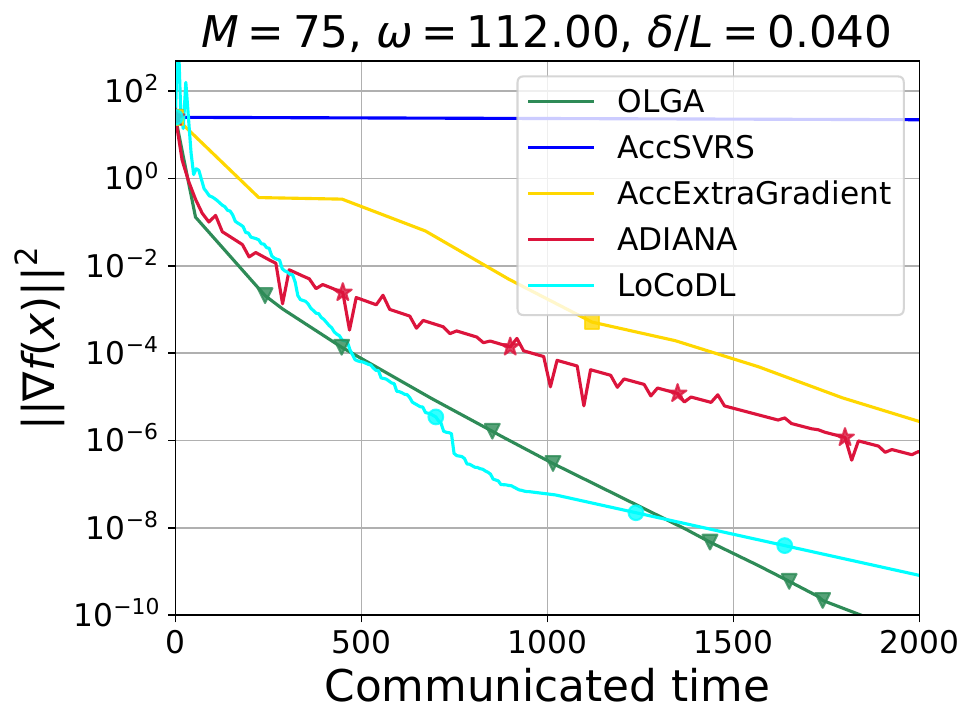}
   \end{subfigure}
   \begin{subfigure}{0.310\textwidth}
       \includegraphics[width=\linewidth]{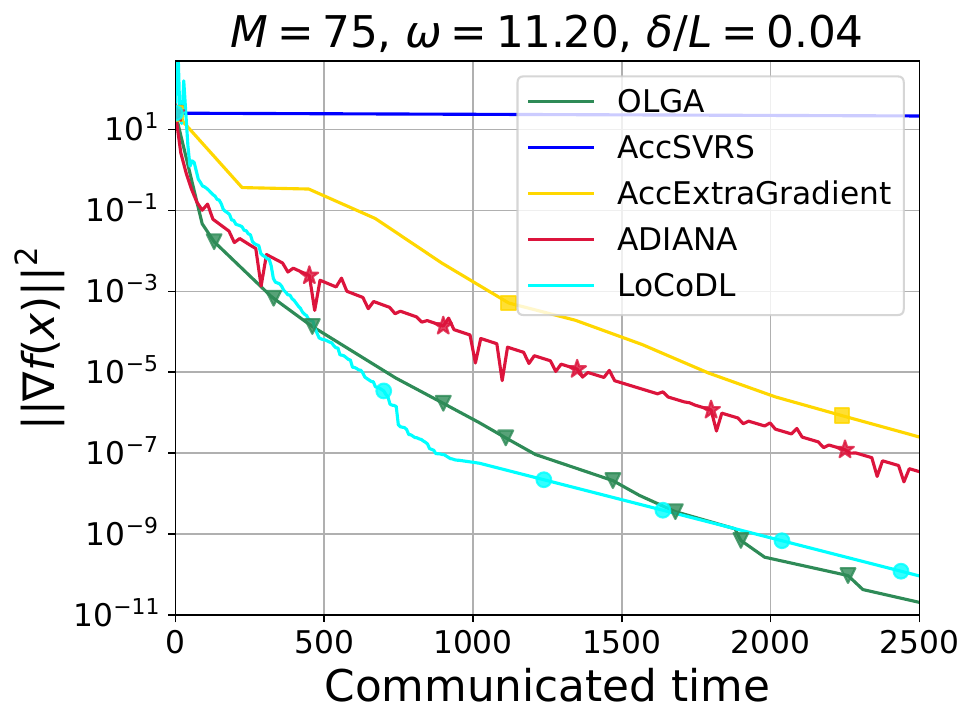}
   \end{subfigure}
   \begin{subfigure}{0.310\textwidth}
       \includegraphics[width=\linewidth]{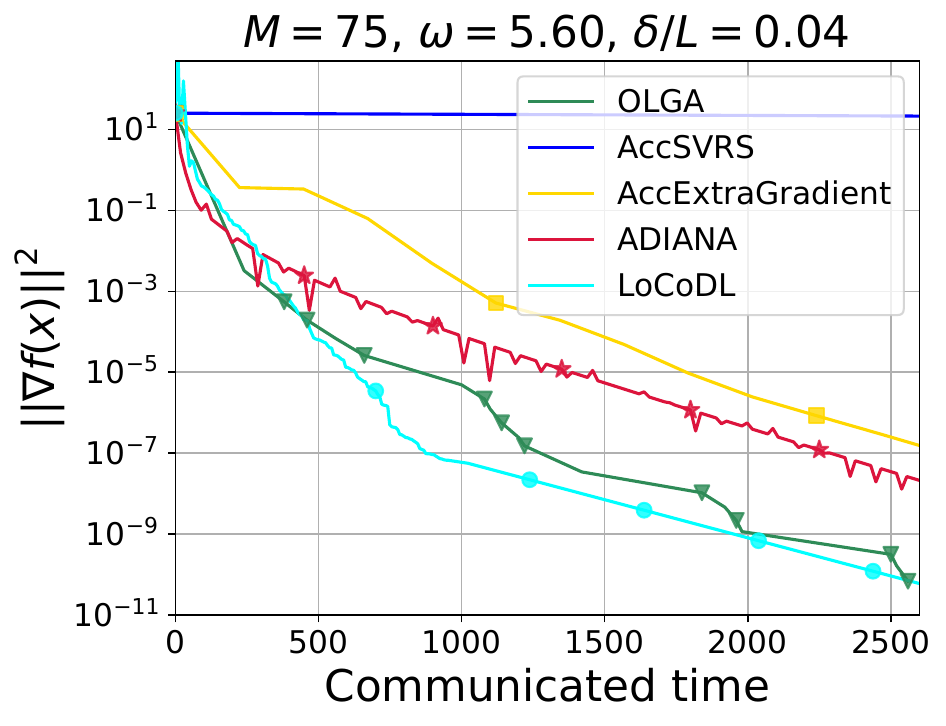}
   \end{subfigure}\\
   \begin{subfigure}{0.310\textwidth}
       \centering
       (a) $\omega = 112.00$
   \end{subfigure}
   \begin{subfigure}{0.310\textwidth}
       \centering
       (b) $\omega = 11.20$
   \end{subfigure}
   \begin{subfigure}{0.310\textwidth}
       \centering
       (c) $\omega = 5.6$
   \end{subfigure}
   \caption{Comparison of state-of-the-art distributed methods. The comparison is made on \eqref{eq:quadr} with $M=75$ and \texttt{mushrooms} dataset. The criterion is the communication time (\textit{CC-3}). For methods with compression we vary the power of compression $\omega$.}
\end{figure}

\begin{figure}[h!] 
   \begin{subfigure}{0.310\textwidth}
       \includegraphics[width=\linewidth]{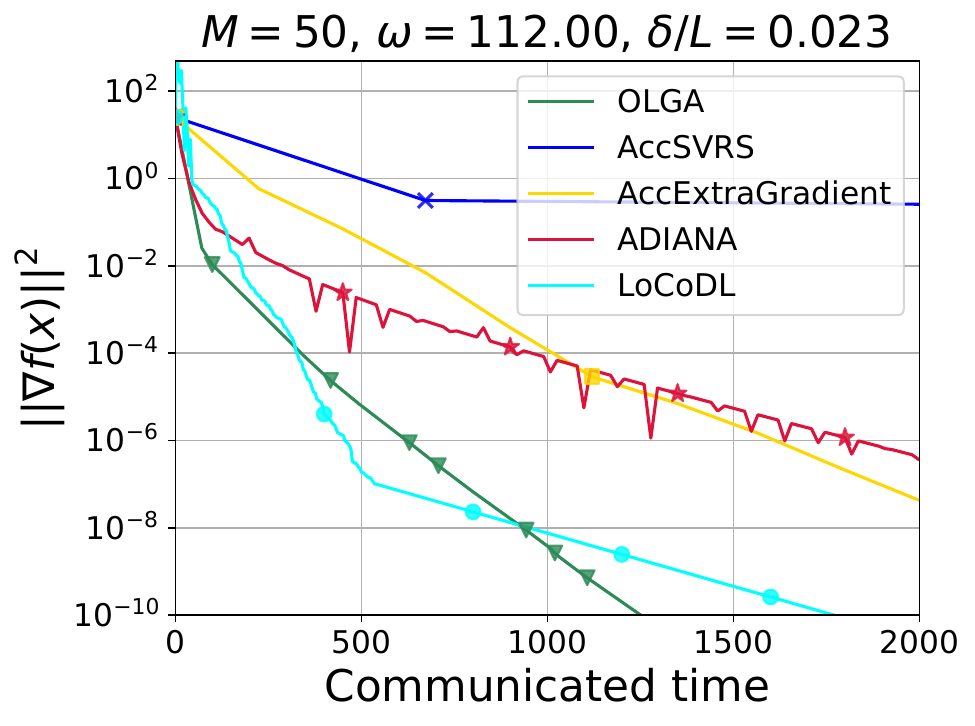}
   \end{subfigure}
   \begin{subfigure}{0.310\textwidth}
       \includegraphics[width=\linewidth]{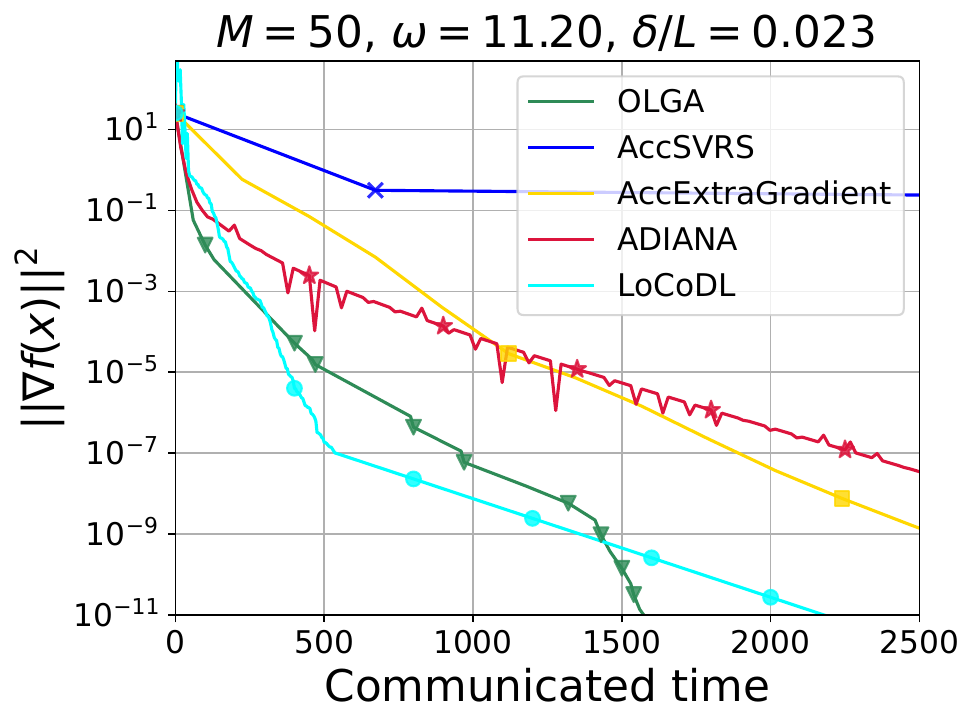}
   \end{subfigure}
   \begin{subfigure}{0.310\textwidth}
       \includegraphics[width=\linewidth]{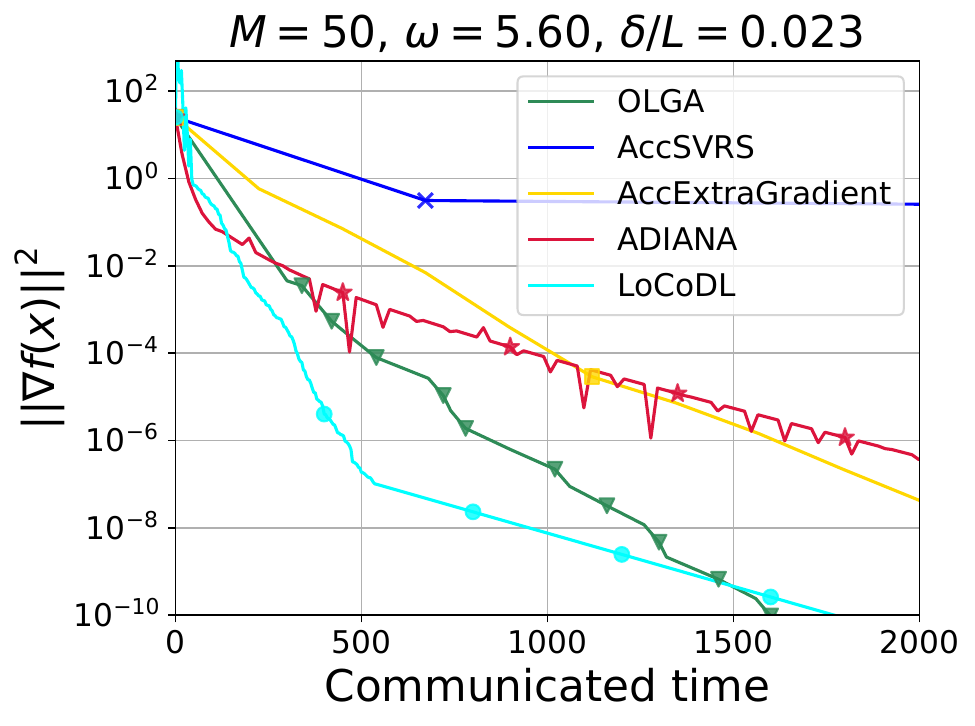}
   \end{subfigure}\\
   \begin{subfigure}{0.310\textwidth}
       \centering
       (a) $\omega = 112.00$
   \end{subfigure}
   \begin{subfigure}{0.310\textwidth}
       \centering
       (b) $\omega = 11.20$
   \end{subfigure}
   \begin{subfigure}{0.310\textwidth}
       \centering
       (c) $\omega = 5.6$
   \end{subfigure}
   \caption{Comparison of state-of-the-art distributed methods. The comparison is made on \eqref{eq:quadr} with $M=50$ and \texttt{mushrooms} dataset. The criterion is the communication time (\textit{CC-3}). For methods with compression we vary the power of compression $\omega$.}
\end{figure}

\begin{figure}[h!] 
   \begin{subfigure}{0.310\textwidth}
       \includegraphics[width=\linewidth]{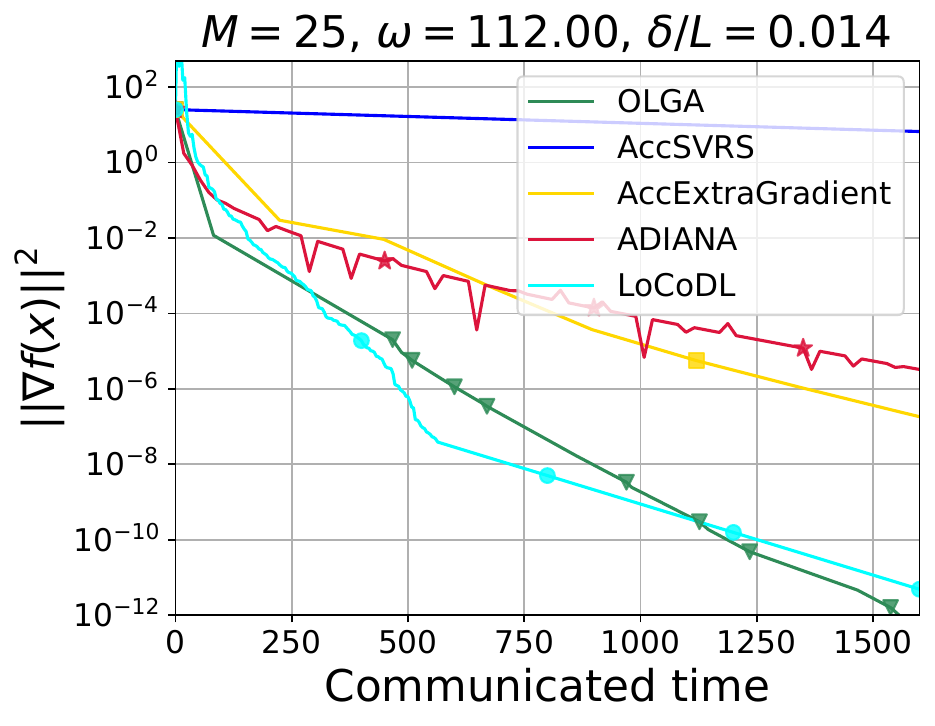}
   \end{subfigure}
   \begin{subfigure}{0.310\textwidth}
       \includegraphics[width=\linewidth]{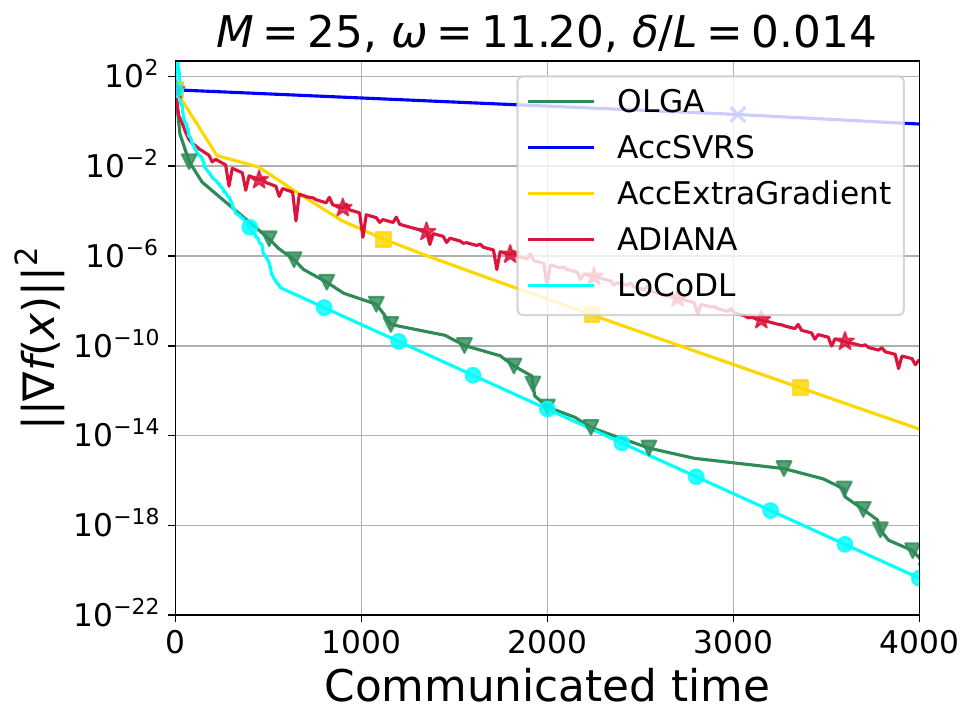}
   \end{subfigure}
   \begin{subfigure}{0.310\textwidth}
       \includegraphics[width=\linewidth]{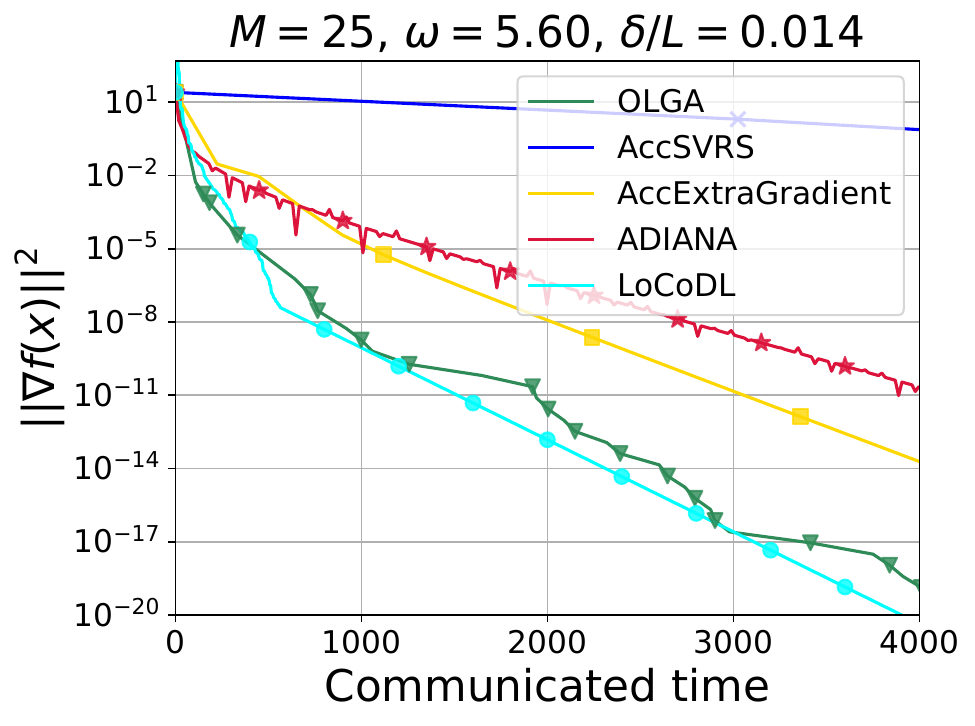}
   \end{subfigure}\\
   \begin{subfigure}{0.310\textwidth}
       \centering
       (a) $\omega = 112.00$
   \end{subfigure}
   \begin{subfigure}{0.310\textwidth}
       \centering
       (b) $\omega = 11.20$
   \end{subfigure}
   \begin{subfigure}{0.310\textwidth}
       \centering
       (c) $\omega = 5.6$
   \end{subfigure}
   \caption{Comparison of state-of-the-art distributed methods. The comparison is made on \eqref{eq:quadr} with $M=25$ and \texttt{mushrooms} dataset. The criterion is the communication time (\textit{CC-3}). For methods with compression we vary the power of compression $\omega$.}
   \label{fig:end_mushrooms}
\end{figure}

\begin{figure}[h!] 
   \begin{subfigure}{0.310\textwidth}
       \includegraphics[width=\linewidth]{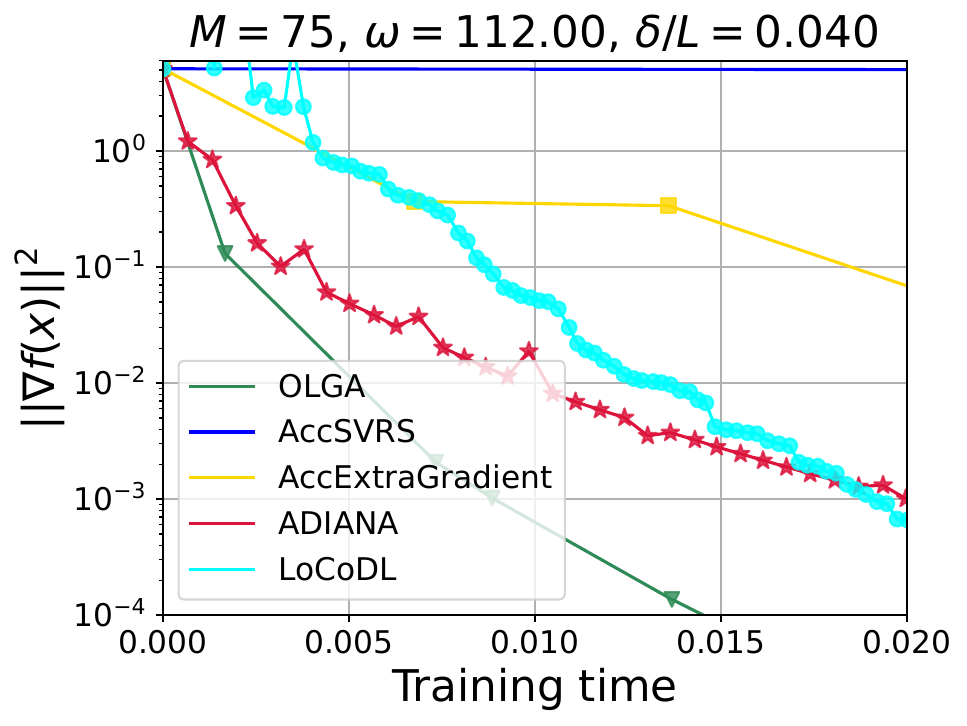}
   \end{subfigure}
   \begin{subfigure}{0.310\textwidth}
       \includegraphics[width=\linewidth]{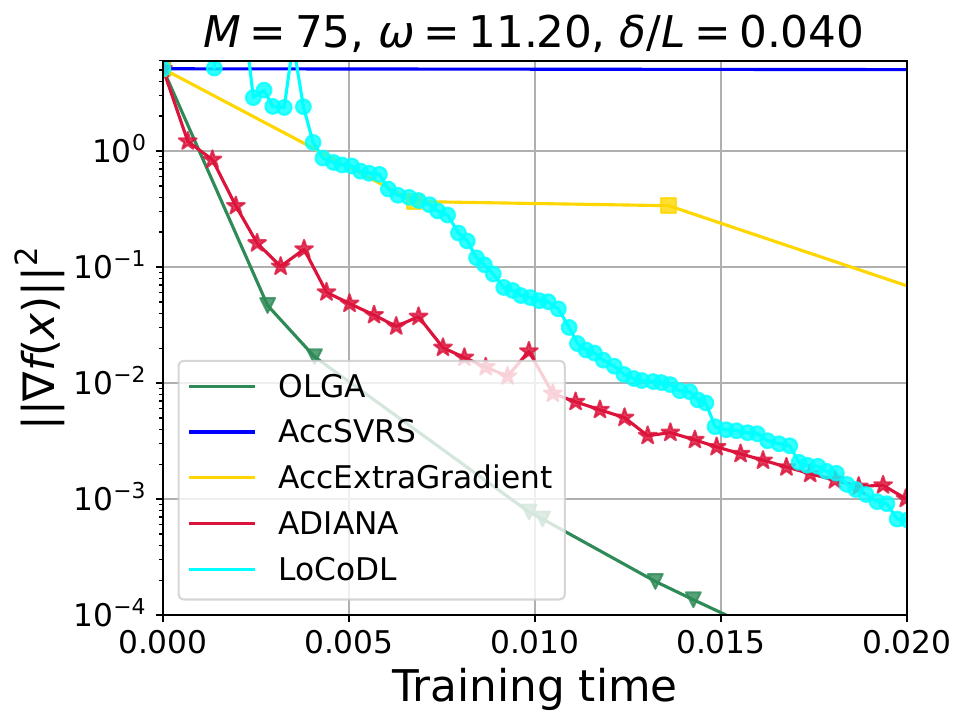}
   \end{subfigure}
   \begin{subfigure}{0.310\textwidth}
       \includegraphics[width=\linewidth]{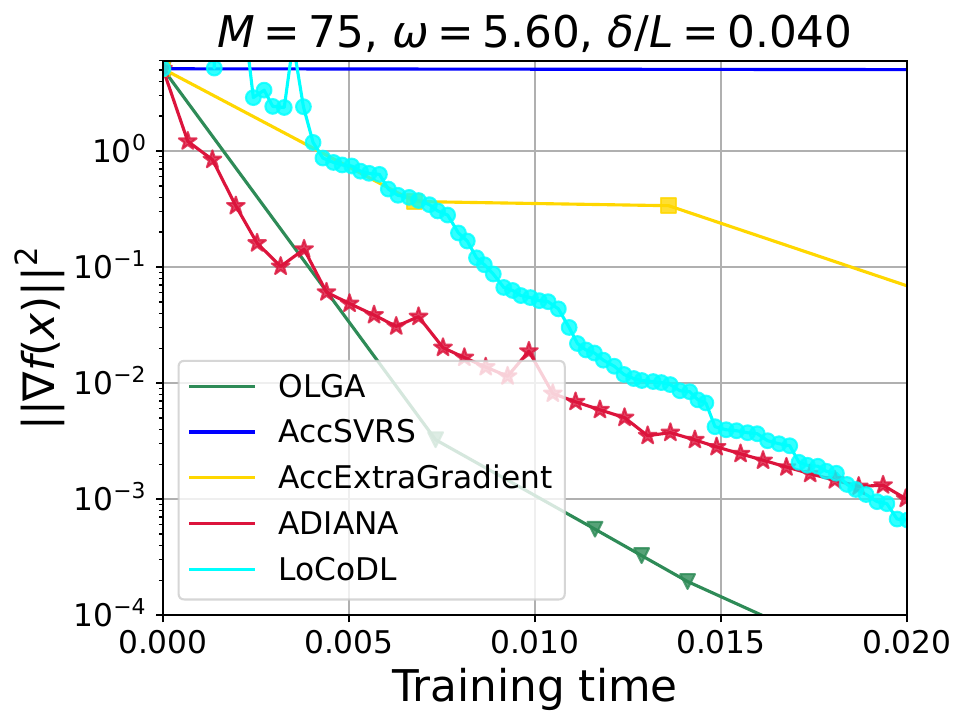}
   \end{subfigure}\\
   \begin{subfigure}{0.310\textwidth}
       \centering
       (a) $\omega = 112.00$
   \end{subfigure}
   \begin{subfigure}{0.310\textwidth}
       \centering
       (b) $\omega = 11.20$
   \end{subfigure}
   \begin{subfigure}{0.310\textwidth}
       \centering
       (c) $\omega = 5.6$
   \end{subfigure}
   \caption{Comparison of state-of-the-art distributed methods. The comparison is made on \eqref{eq:quadr} with $M=75$ and \texttt{mushrooms} dataset. The criterion is the training time on local cluster (fast connection). For methods with compression we vary the power of compression $\omega$.}
   \label{fig:begin_cluster}
\end{figure}

\begin{figure}[h!] 
   \begin{subfigure}{0.310\textwidth}
       \includegraphics[width=\linewidth]{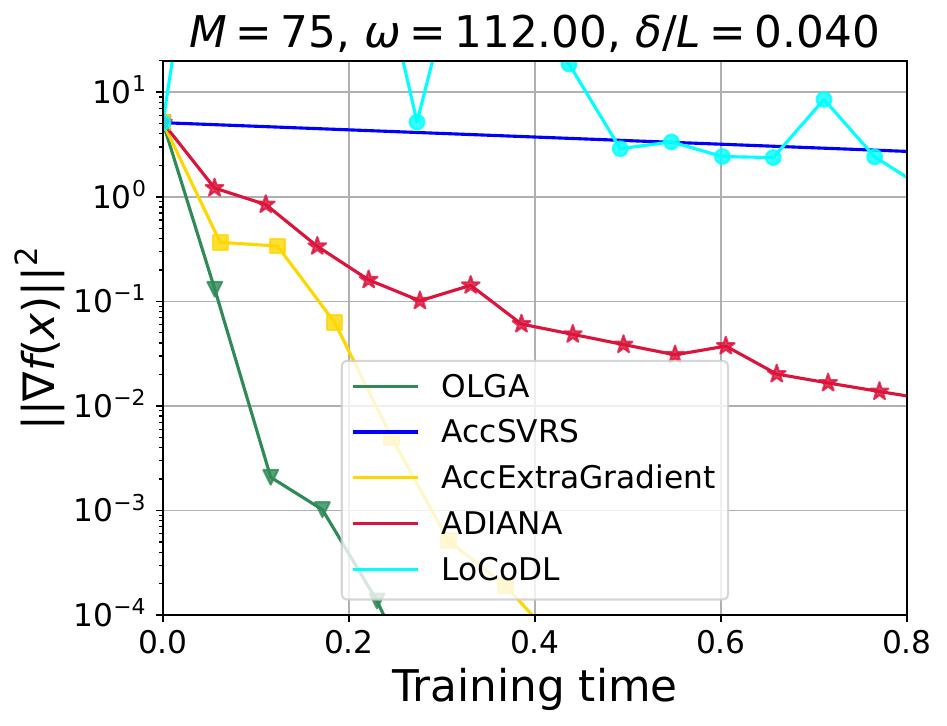}
   \end{subfigure}
   \begin{subfigure}{0.310\textwidth}
       \includegraphics[width=\linewidth]{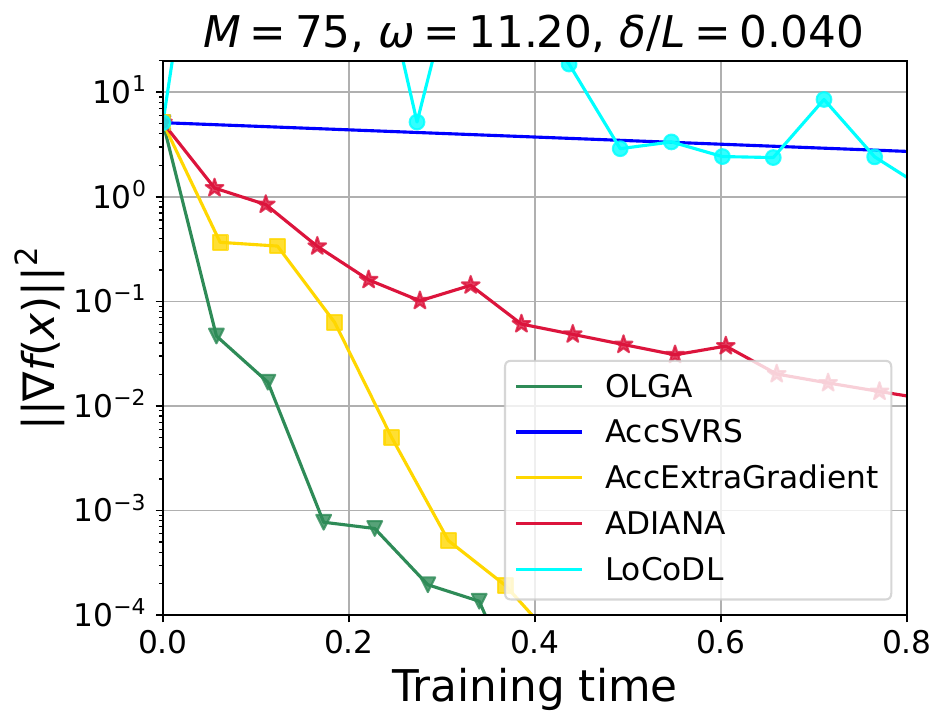}
   \end{subfigure}
   \begin{subfigure}{0.310\textwidth}
       \includegraphics[width=\linewidth]{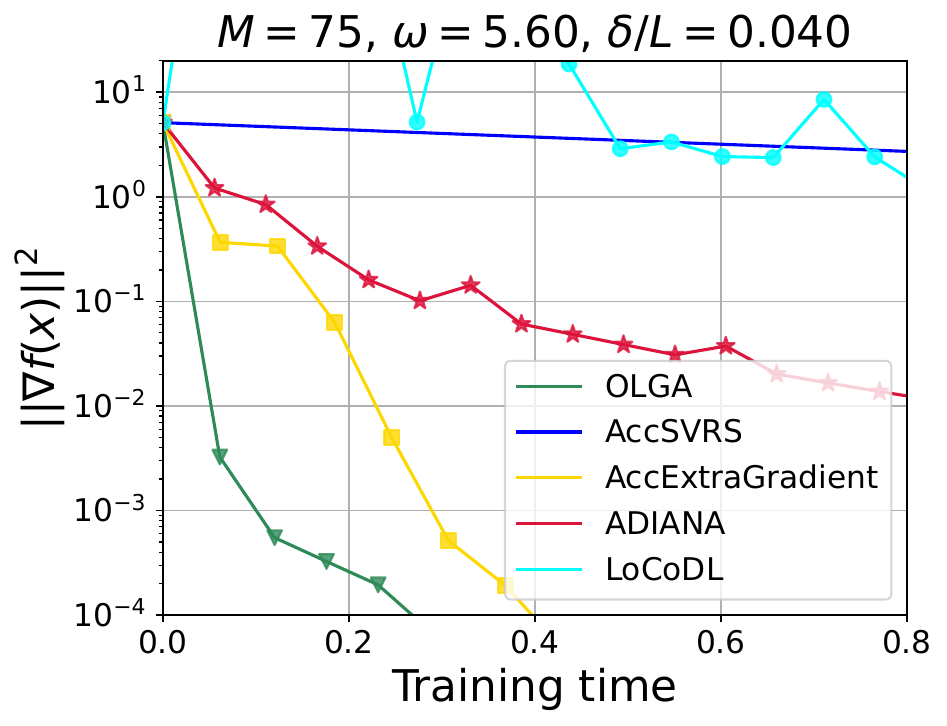}
   \end{subfigure}\\
   \begin{subfigure}{0.310\textwidth}
       \centering
       (a) $\omega = 112.00$
   \end{subfigure}
   \begin{subfigure}{0.310\textwidth}
       \centering
       (b) $\omega = 11.20$
   \end{subfigure}
   \begin{subfigure}{0.310\textwidth}
       \centering
       (c) $\omega = 5.6$
   \end{subfigure}
   \caption{Comparison of state-of-the-art distributed methods. The comparison is made on \eqref{eq:quadr} with $M=75$ and \texttt{mushrooms} dataset. The criterion is the training time on remote CPUs (slow connection). For methods with compression we vary the power of compression $\omega$.}
\end{figure}

\begin{figure}[h!] 
   \begin{subfigure}{0.310\textwidth}
       \includegraphics[width=\linewidth]{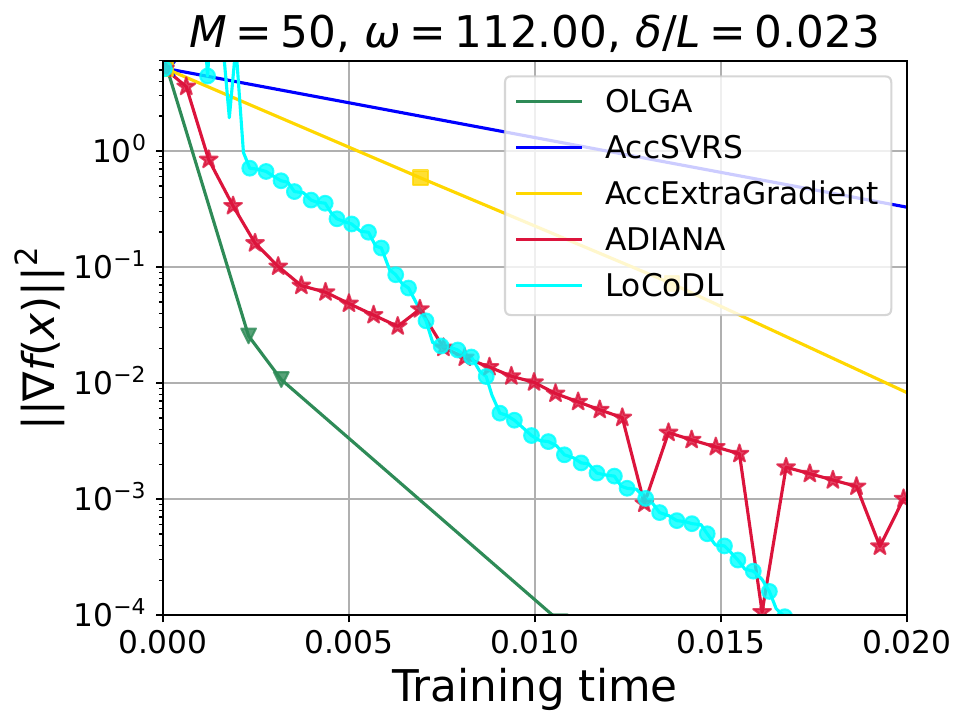}
   \end{subfigure}
   \begin{subfigure}{0.310\textwidth}
       \includegraphics[width=\linewidth]{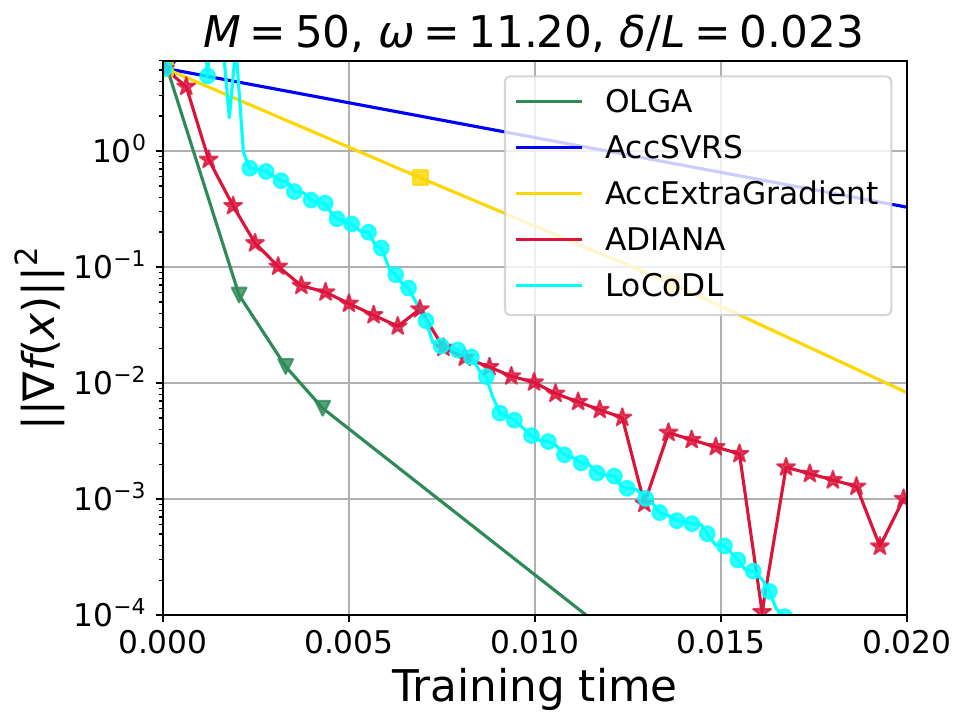}
   \end{subfigure}
   \begin{subfigure}{0.310\textwidth}
       \includegraphics[width=\linewidth]{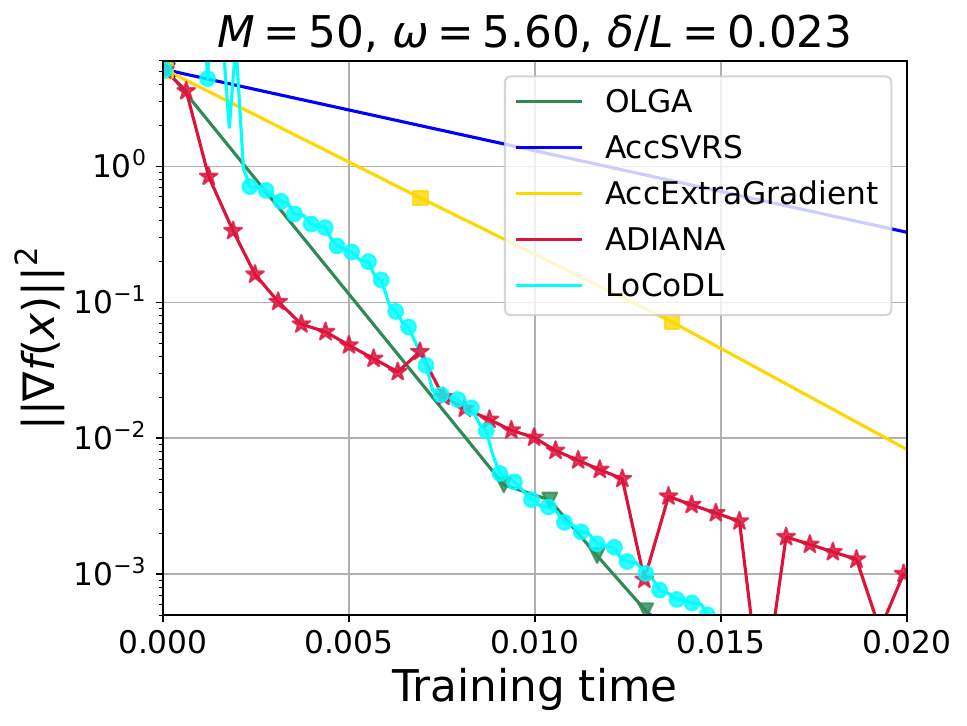}
   \end{subfigure}\\
   \begin{subfigure}{0.310\textwidth}
       \centering
       (a) $\omega = 112.00$
   \end{subfigure}
   \begin{subfigure}{0.310\textwidth}
       \centering
       (b) $\omega = 11.20$
   \end{subfigure}
   \begin{subfigure}{0.310\textwidth}
       \centering
       (c) $\omega = 5.6$
   \end{subfigure}
   \caption{Comparison of state-of-the-art distributed methods. The comparison is made on \eqref{eq:quadr} with $M=50$ and \texttt{mushrooms} dataset. The criterion is the training time on local cluster (fast connection). For methods with compression we vary the power of compression $\omega$.}
\end{figure}

\begin{figure}[h!] 
   \begin{subfigure}{0.310\textwidth}
       \includegraphics[width=\linewidth]{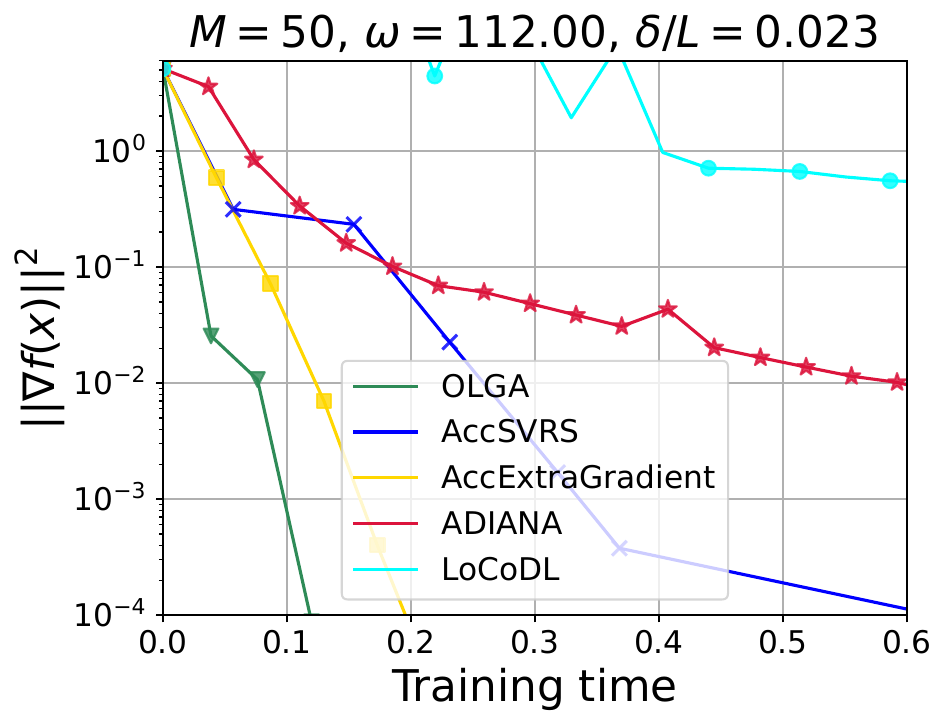}
   \end{subfigure}
   \begin{subfigure}{0.310\textwidth}
       \includegraphics[width=\linewidth]{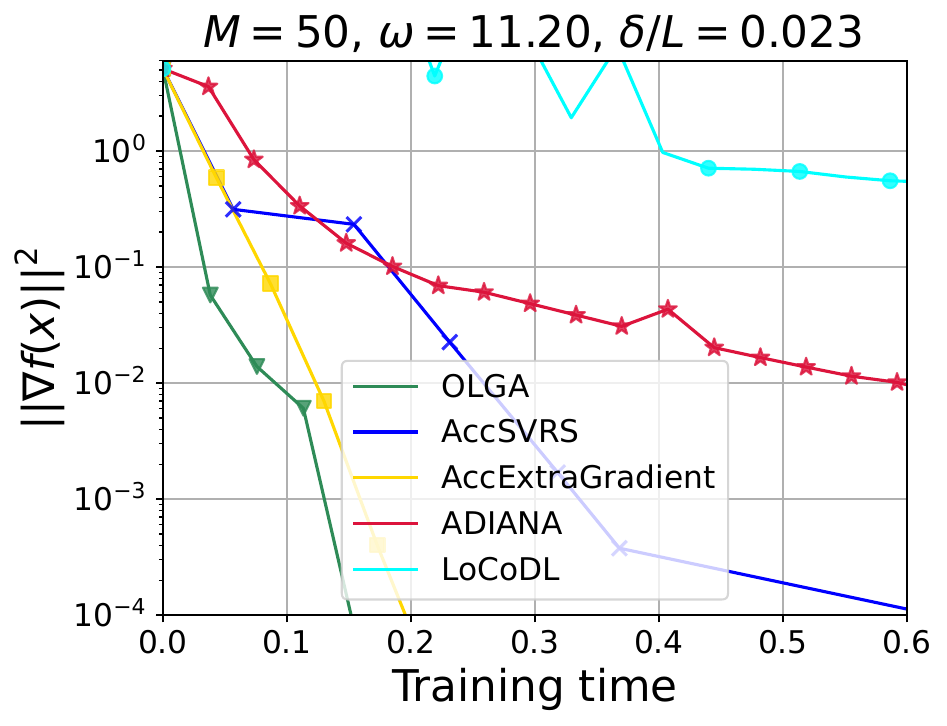}
   \end{subfigure}
   \begin{subfigure}{0.310\textwidth}
       \includegraphics[width=\linewidth]{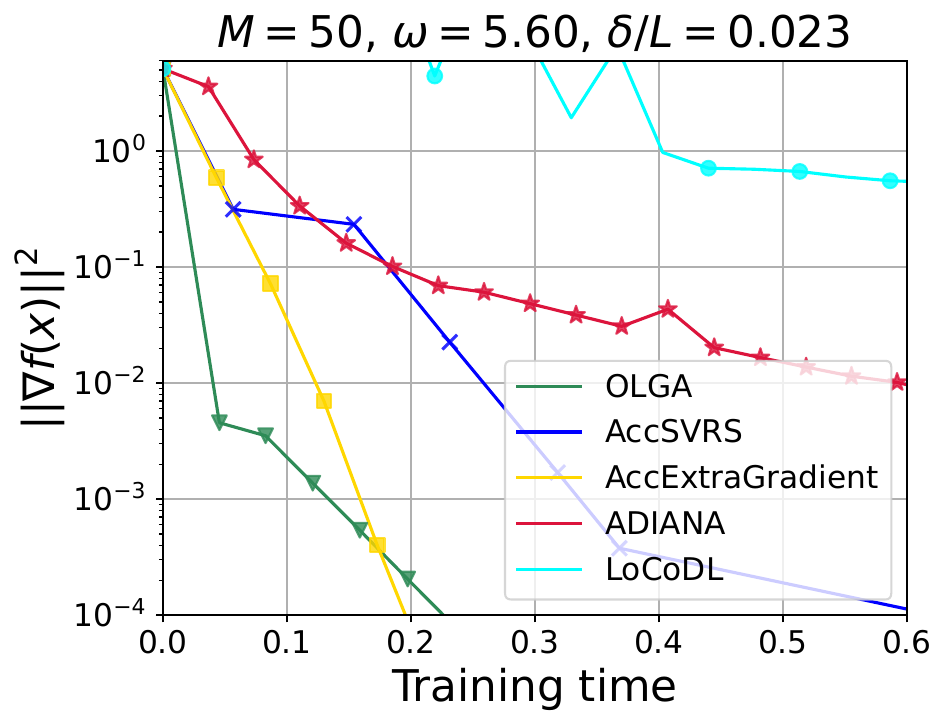}
   \end{subfigure}\\
   \begin{subfigure}{0.310\textwidth}
       \centering
       (a) $\omega = 112.00$
   \end{subfigure}
   \begin{subfigure}{0.310\textwidth}
       \centering
       (b) $\omega = 11.20$
   \end{subfigure}
   \begin{subfigure}{0.310\textwidth}
       \centering
       (c) $\omega = 5.6$
   \end{subfigure}
   \caption{Comparison of state-of-the-art distributed methods. The comparison is made on \eqref{eq:quadr} with $M=50$ and \texttt{mushrooms} dataset. The criterion is the training time on remote CPUs (slow connection). For methods with compression we vary the power of compression $\omega$.}
\end{figure}

\begin{figure}[h!] 
   \begin{subfigure}{0.310\textwidth}
       \includegraphics[width=\linewidth]{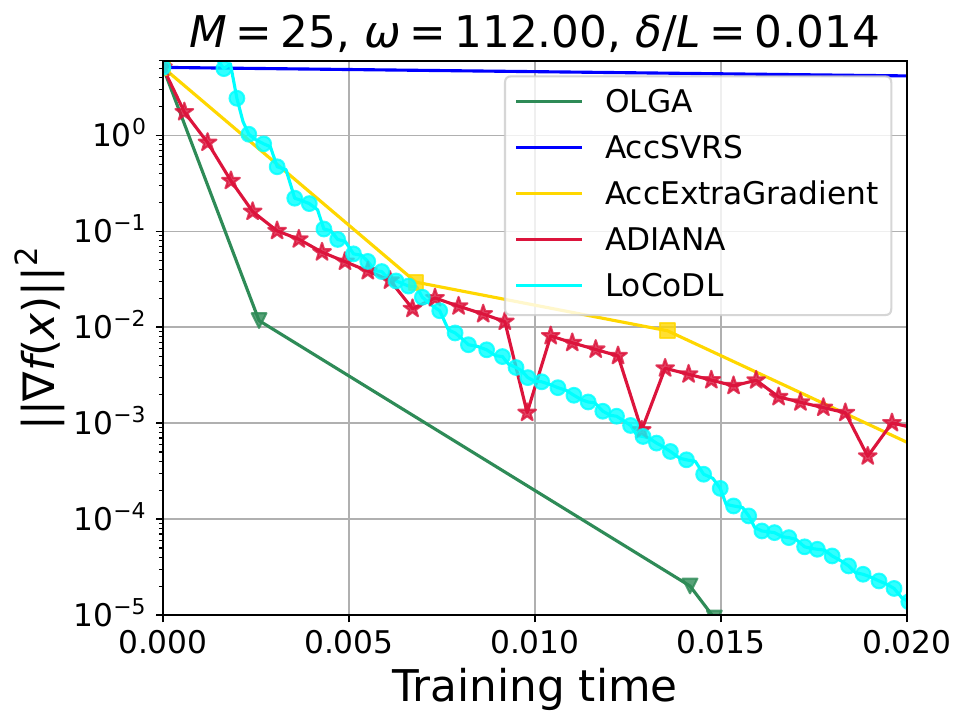}
   \end{subfigure}
   \begin{subfigure}{0.310\textwidth}
       \includegraphics[width=\linewidth]{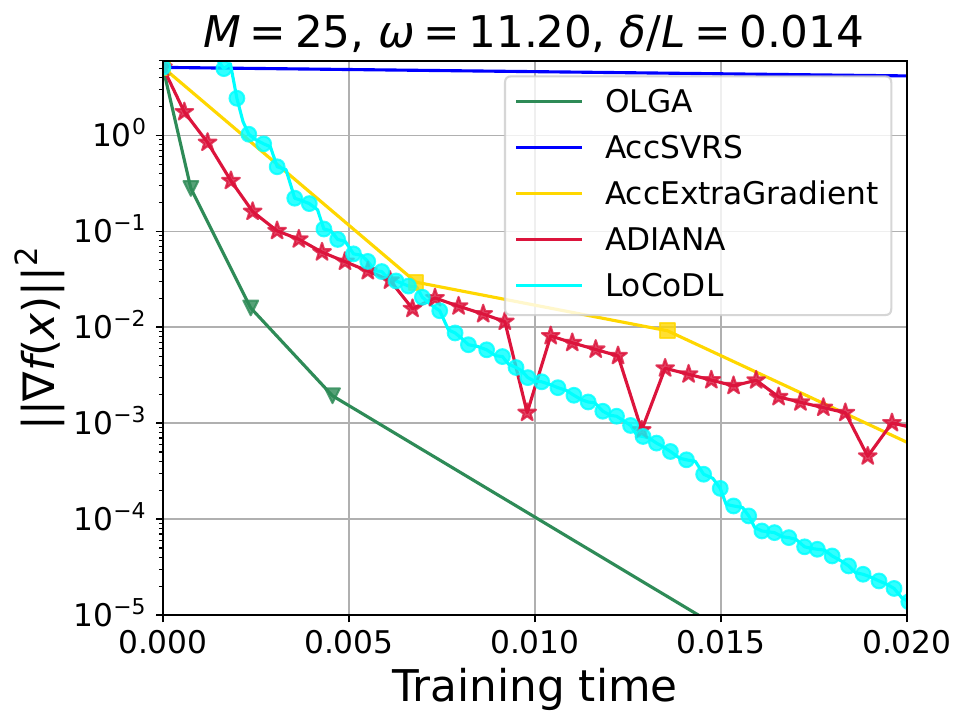}
   \end{subfigure}
   \begin{subfigure}{0.310\textwidth}
       \includegraphics[width=\linewidth]{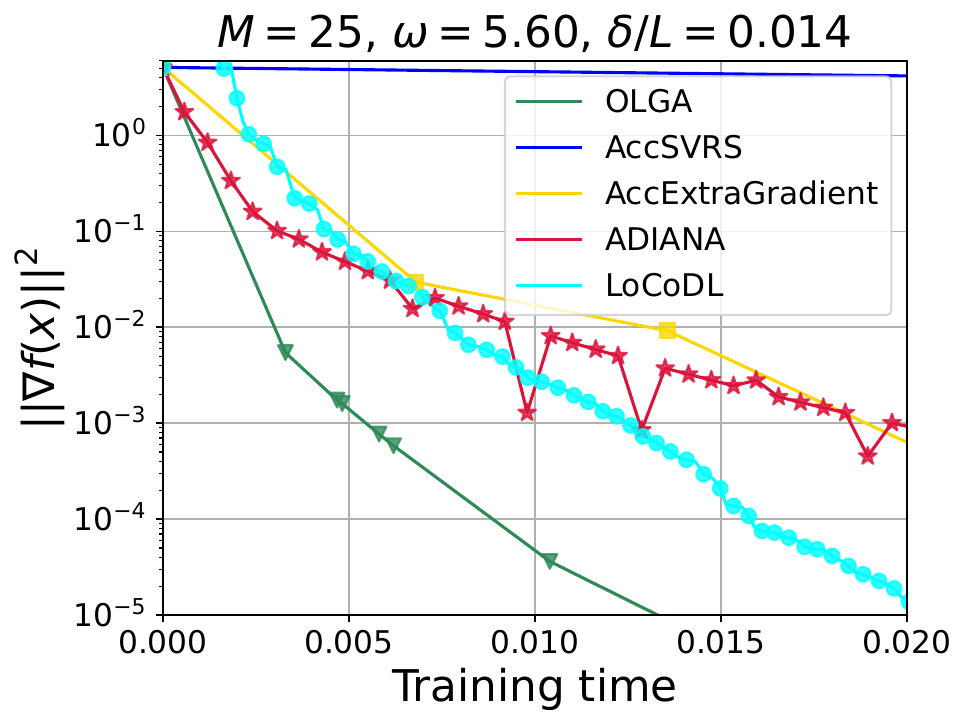}
   \end{subfigure}\\
   \begin{subfigure}{0.310\textwidth}
       \centering
       (a) $\omega = 112.00$
   \end{subfigure}
   \begin{subfigure}{0.310\textwidth}
       \centering
       (b) $\omega = 11.20$
   \end{subfigure}
   \begin{subfigure}{0.310\textwidth}
       \centering
       (c) $\omega = 5.6$
   \end{subfigure}
   \caption{Comparison of state-of-the-art distributed methods. The comparison is made on \eqref{eq:quadr} with $M=25$ and \texttt{mushrooms} dataset. The criterion is the training time on local cluster (fast connection). For methods with compression we vary the power of compression $\omega$.}
\end{figure}

\begin{figure}[h!] 
   \begin{subfigure}{0.310\textwidth}
       \includegraphics[width=\linewidth]{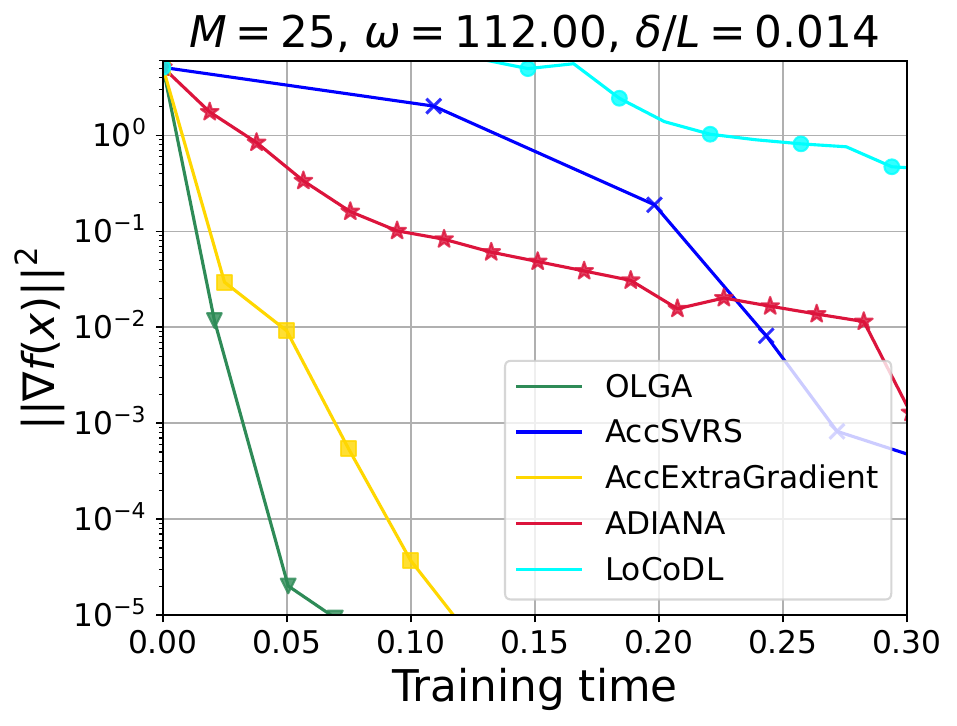}
   \end{subfigure}
   \begin{subfigure}{0.310\textwidth}
       \includegraphics[width=\linewidth]{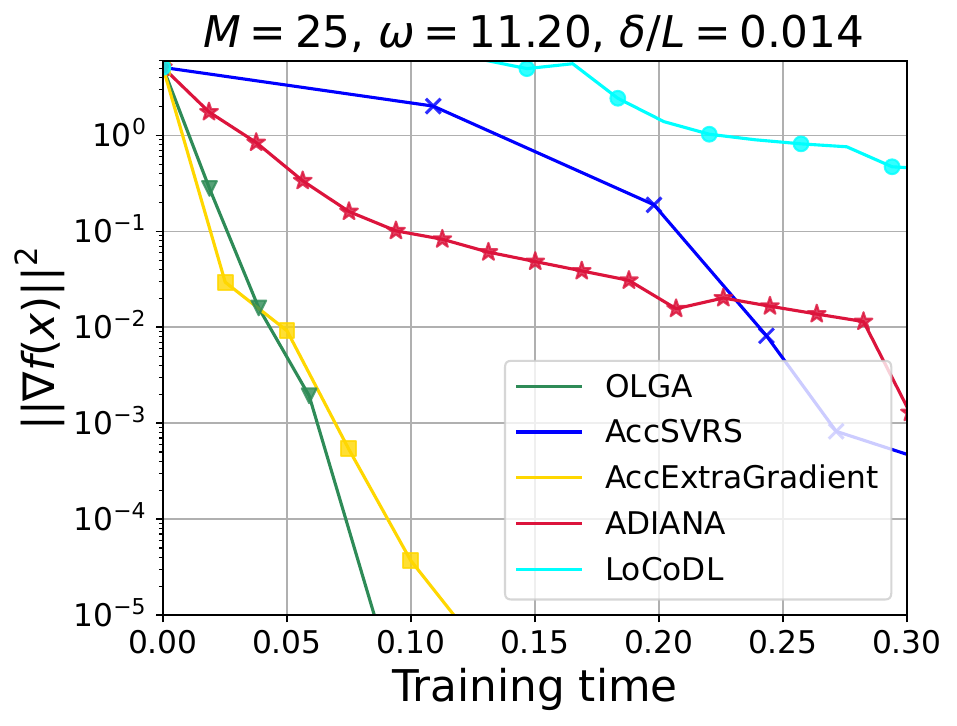}
   \end{subfigure}
   \begin{subfigure}{0.310\textwidth}
       \includegraphics[width=\linewidth]{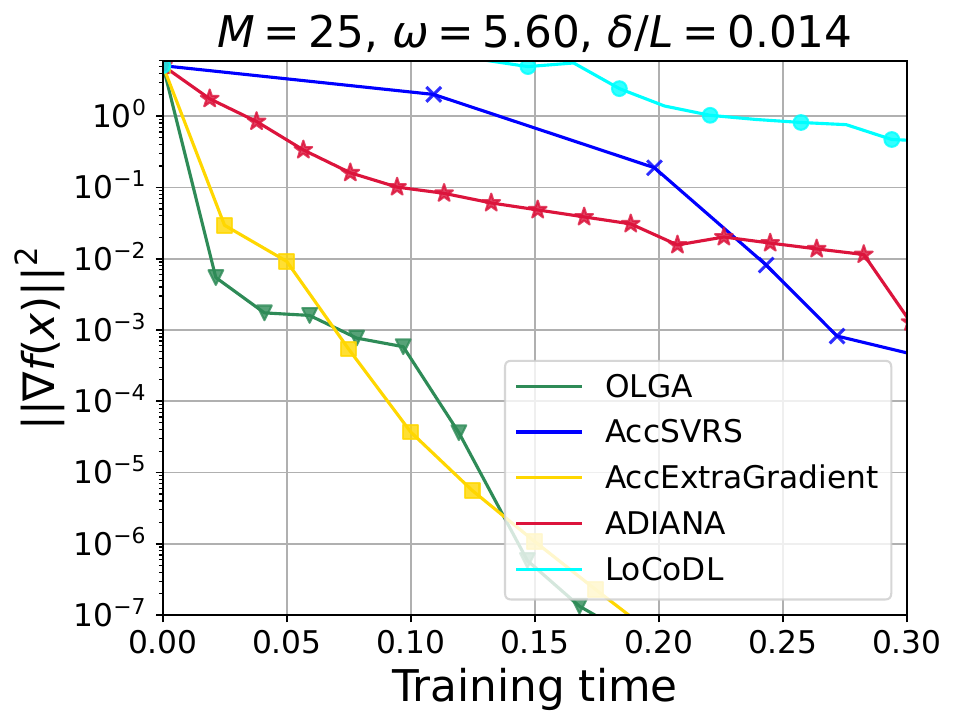}
   \end{subfigure}\\
   \begin{subfigure}{0.310\textwidth}
       \centering
       (a) $\omega = 112.00$
   \end{subfigure}
   \begin{subfigure}{0.310\textwidth}
       \centering
       (b) $\omega = 11.20$
   \end{subfigure}
   \begin{subfigure}{0.310\textwidth}
       \centering
       (c) $\omega = 5.6$
   \end{subfigure}
   \caption{Comparison of state-of-the-art distributed methods. The comparison is made on \eqref{eq:quadr} with $M=25$ and \texttt{mushrooms} dataset. The criterion is the training time on remote CPUs (slow connection). For methods with compression we vary the power of compression $\omega$.}
\end{figure}

\begin{figure}[h!] 
   \begin{subfigure}{0.310\textwidth}
       \includegraphics[width=\linewidth]{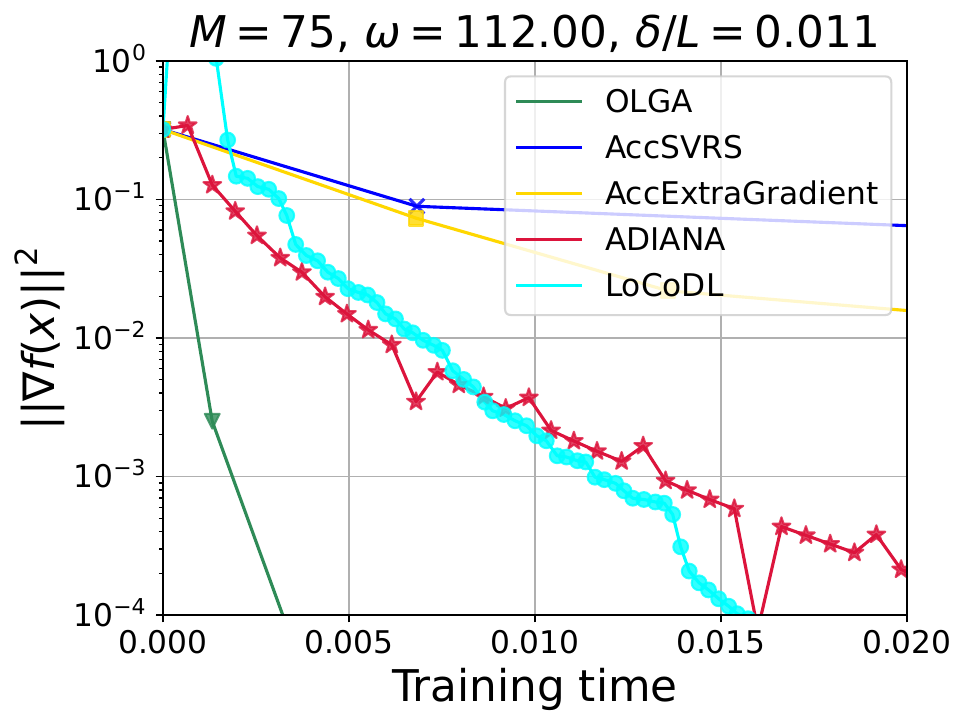}
   \end{subfigure}
   \begin{subfigure}{0.310\textwidth}
       \includegraphics[width=\linewidth]{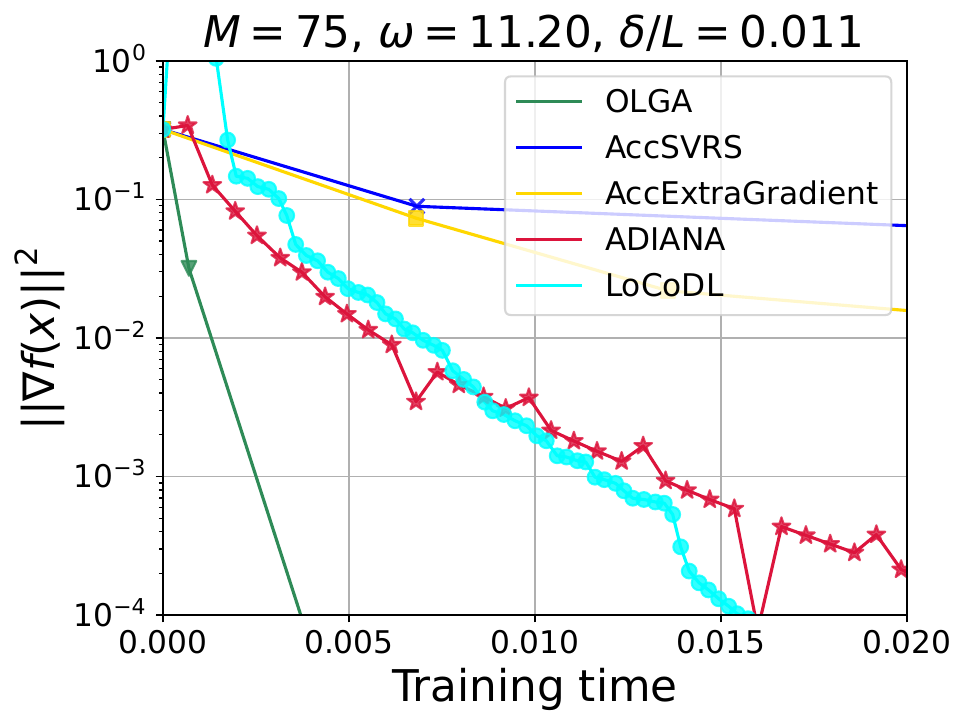}
   \end{subfigure}
   \begin{subfigure}{0.310\textwidth}
       \includegraphics[width=\linewidth]{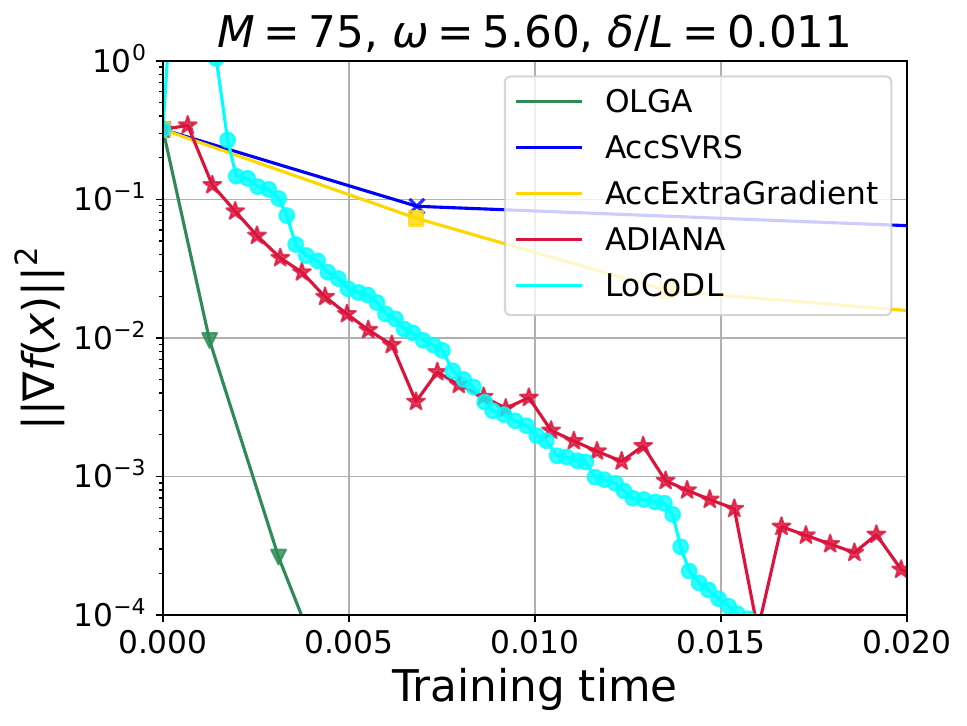}
   \end{subfigure}\\
   \begin{subfigure}{0.310\textwidth}
       \centering
       (a) $\omega = 112.00$
   \end{subfigure}
   \begin{subfigure}{0.310\textwidth}
       \centering
       (b) $\omega = 11.20$
   \end{subfigure}
   \begin{subfigure}{0.310\textwidth}
       \centering
       (c) $\omega = 5.6$
   \end{subfigure}
   \caption{Comparison of state-of-the-art distributed methods. The comparison is made on \eqref{eq:logloss} with $M=75$ and \texttt{mushrooms} dataset. The criterion is the training time on local cluster (fast connection). For methods with compression we vary the power of compression $\omega$.}
\end{figure}

\begin{figure}[h!] 
   \begin{subfigure}{0.310\textwidth}
       \includegraphics[width=\linewidth]{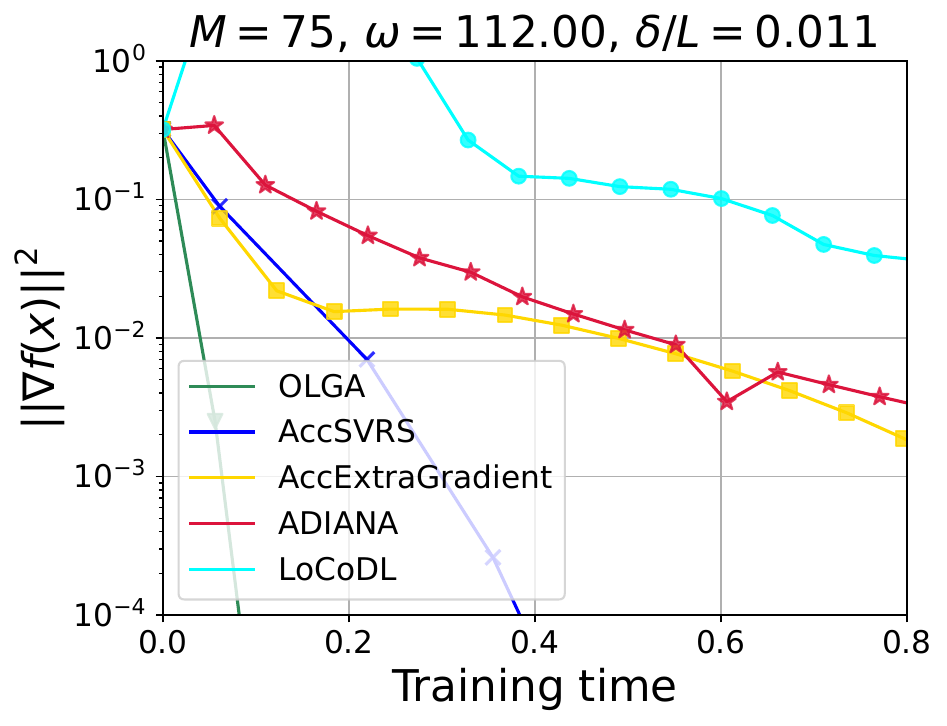}
   \end{subfigure}
   \begin{subfigure}{0.310\textwidth}
       \includegraphics[width=\linewidth]{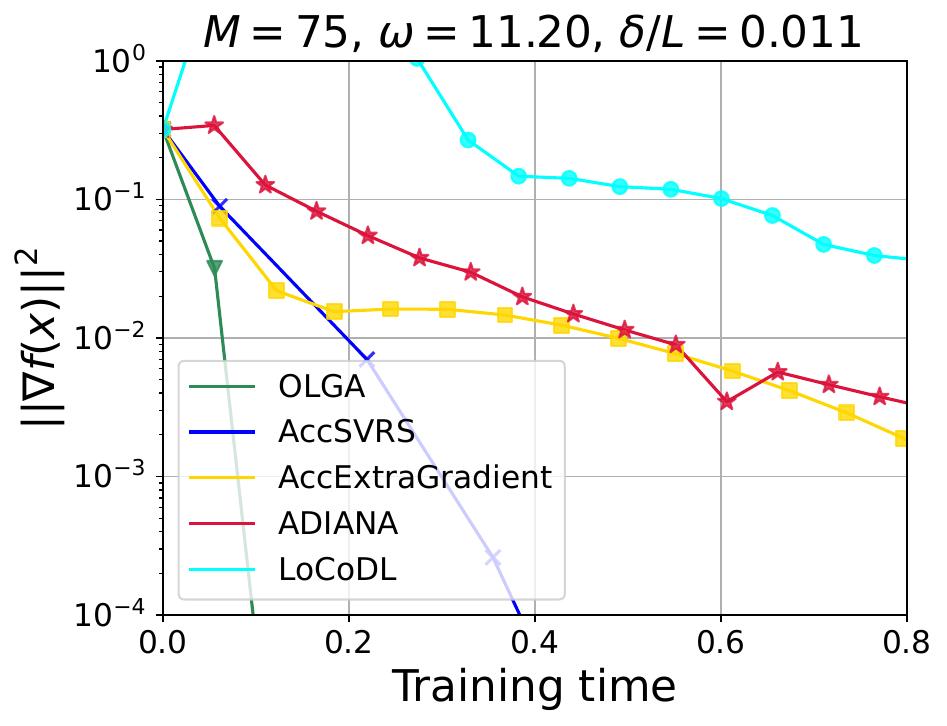}
   \end{subfigure}
   \begin{subfigure}{0.310\textwidth}
       \includegraphics[width=\linewidth]{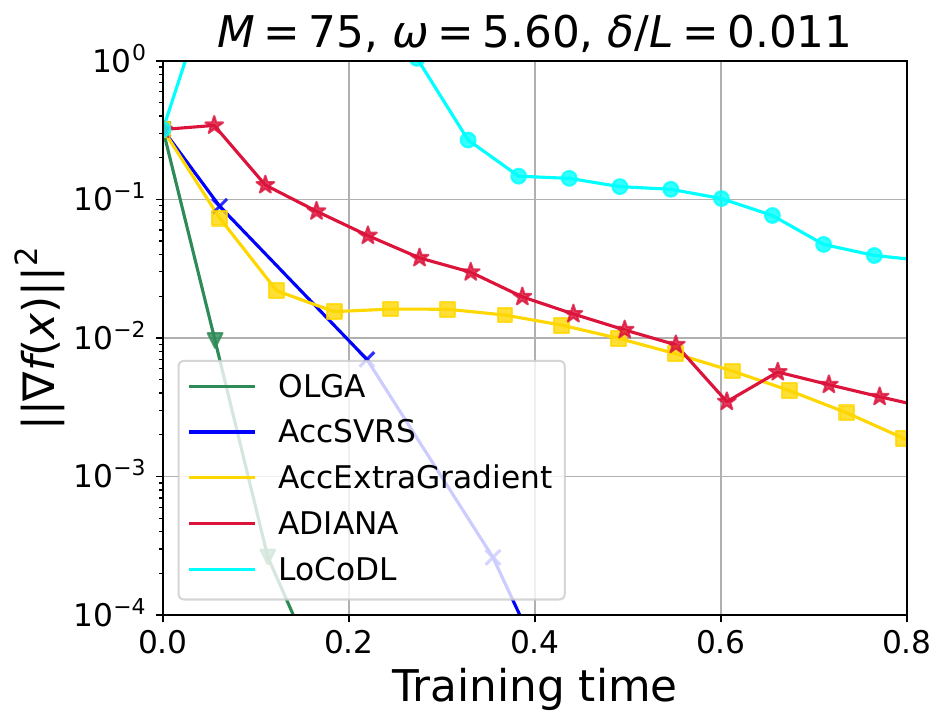}
   \end{subfigure}\\
   \begin{subfigure}{0.310\textwidth}
       \centering
       (a) $\omega = 112.00$
   \end{subfigure}
   \begin{subfigure}{0.310\textwidth}
       \centering
       (b) $\omega = 11.20$
   \end{subfigure}
   \begin{subfigure}{0.310\textwidth}
       \centering
       (c) $\omega = 5.6$
   \end{subfigure}
   \caption{Comparison of state-of-the-art distributed methods. The comparison is made on \eqref{eq:logloss} with $M=75$ and \texttt{mushrooms} dataset. The criterion is the training time on remote CPUs (slow connection). For methods with compression we vary the power of compression $\omega$.}
\end{figure}

\begin{figure}[h!] 
   \begin{subfigure}{0.310\textwidth}
       \includegraphics[width=\linewidth]{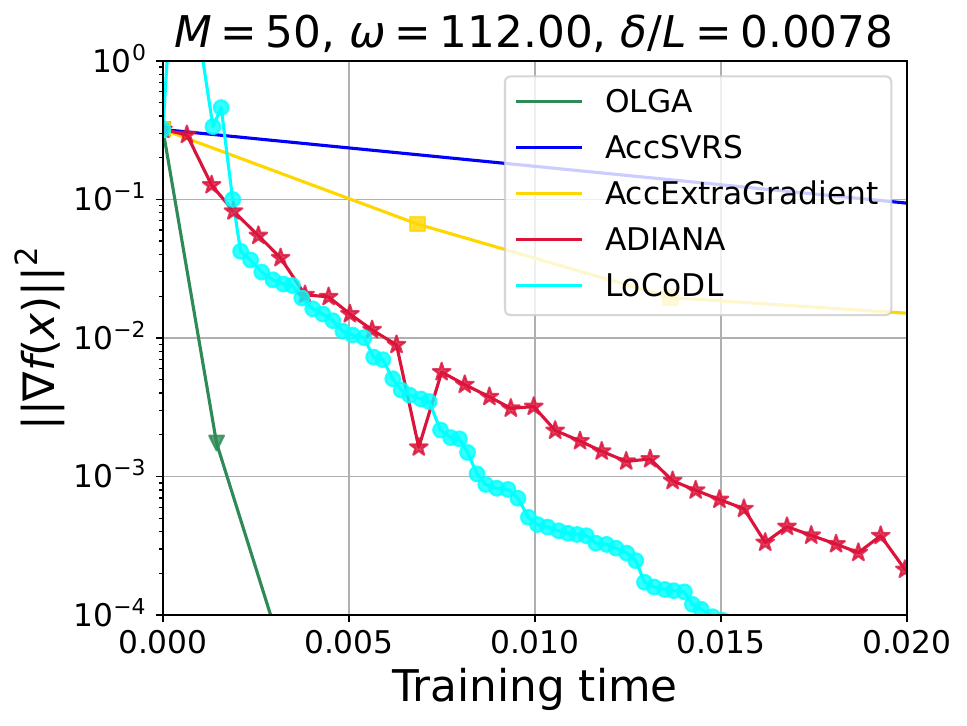}
   \end{subfigure}
   \begin{subfigure}{0.310\textwidth}
       \includegraphics[width=\linewidth]{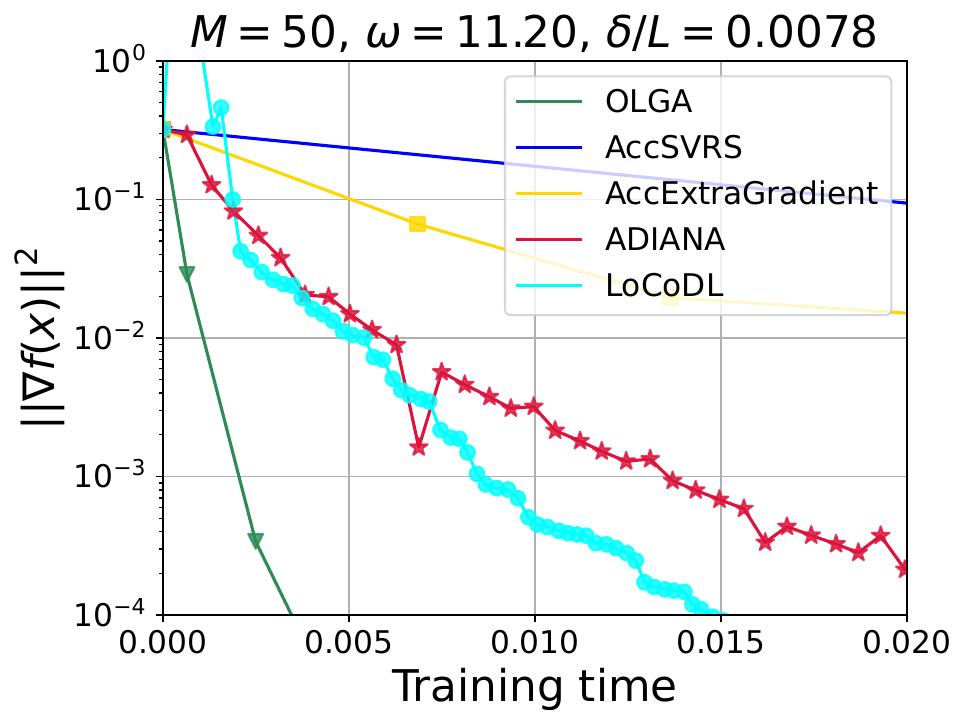}
   \end{subfigure}
   \begin{subfigure}{0.310\textwidth}
       \includegraphics[width=\linewidth]{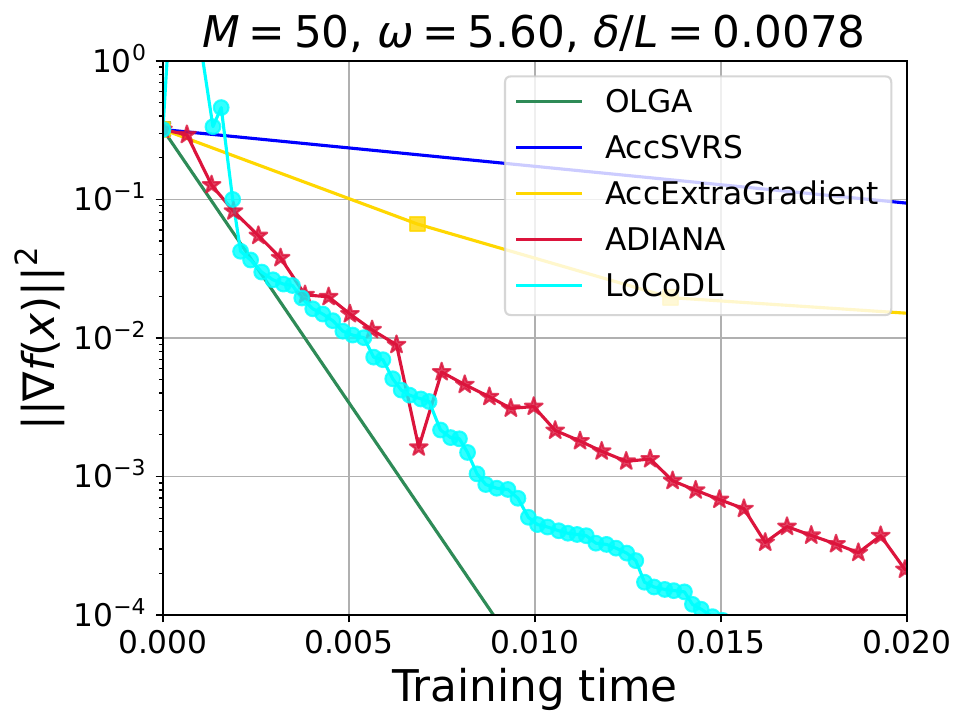}
   \end{subfigure}\\
   \begin{subfigure}{0.310\textwidth}
       \centering
       (a) $\omega = 112.00$
   \end{subfigure}
   \begin{subfigure}{0.310\textwidth}
       \centering
       (b) $\omega = 11.20$
   \end{subfigure}
   \begin{subfigure}{0.310\textwidth}
       \centering
       (c) $\omega = 5.6$
   \end{subfigure}
   \caption{Comparison of state-of-the-art distributed methods. The comparison is made on \eqref{eq:logloss} with $M=50$ and \texttt{mushrooms} dataset. The criterion is the training time on local cluster (fast connection). For methods with compression we vary the power of compression $\omega$.}
\end{figure}

\begin{figure}[h!] 
   \begin{subfigure}{0.310\textwidth}
       \includegraphics[width=\linewidth]{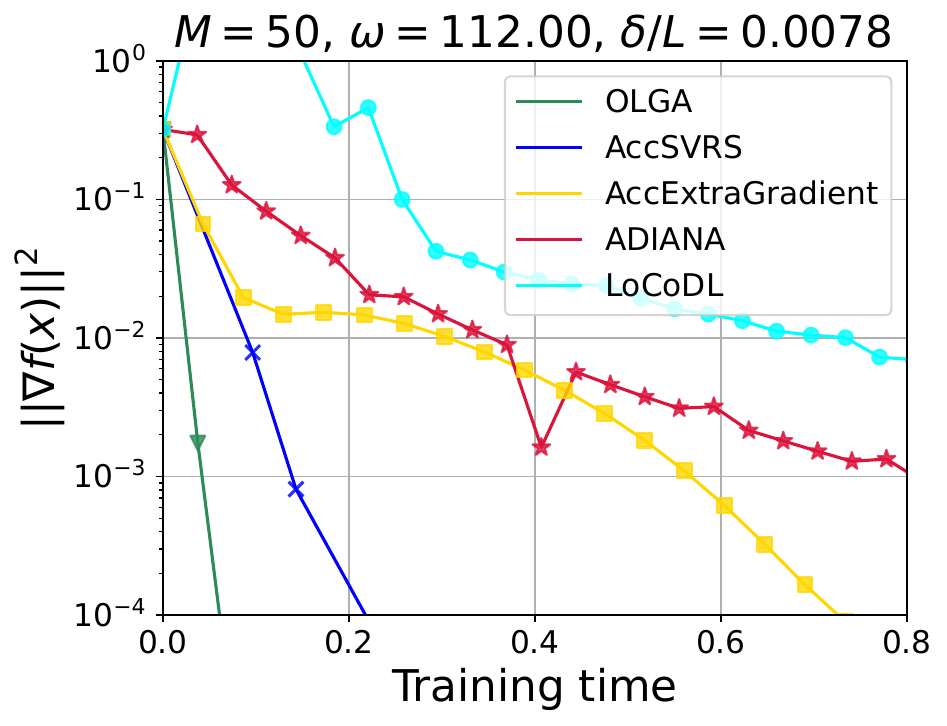}
   \end{subfigure}
   \begin{subfigure}{0.310\textwidth}
       \includegraphics[width=\linewidth]{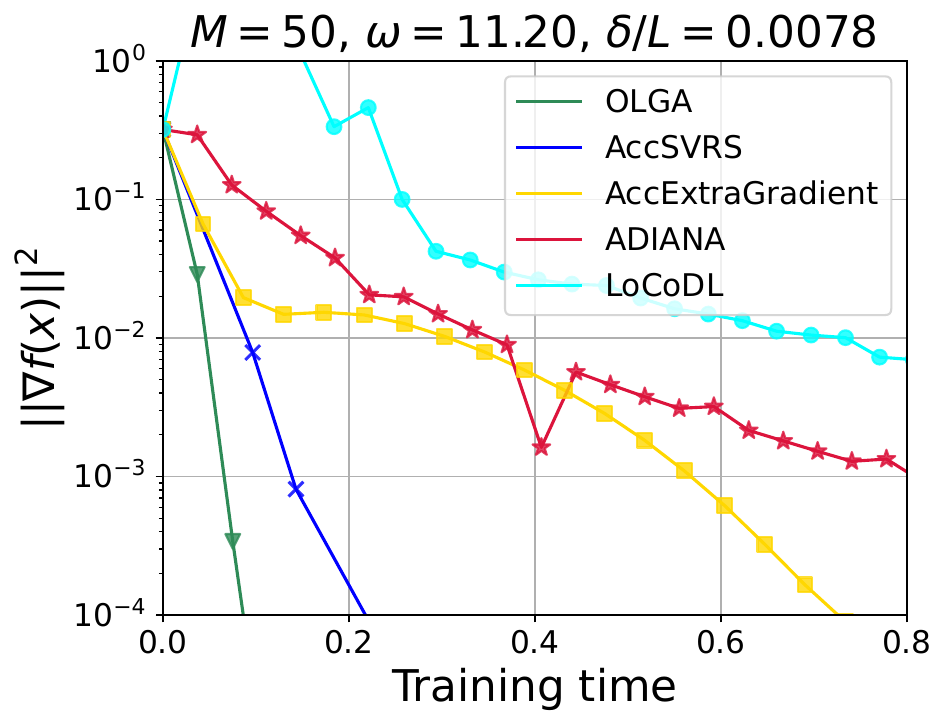}
   \end{subfigure}
   \begin{subfigure}{0.310\textwidth}
       \includegraphics[width=\linewidth]{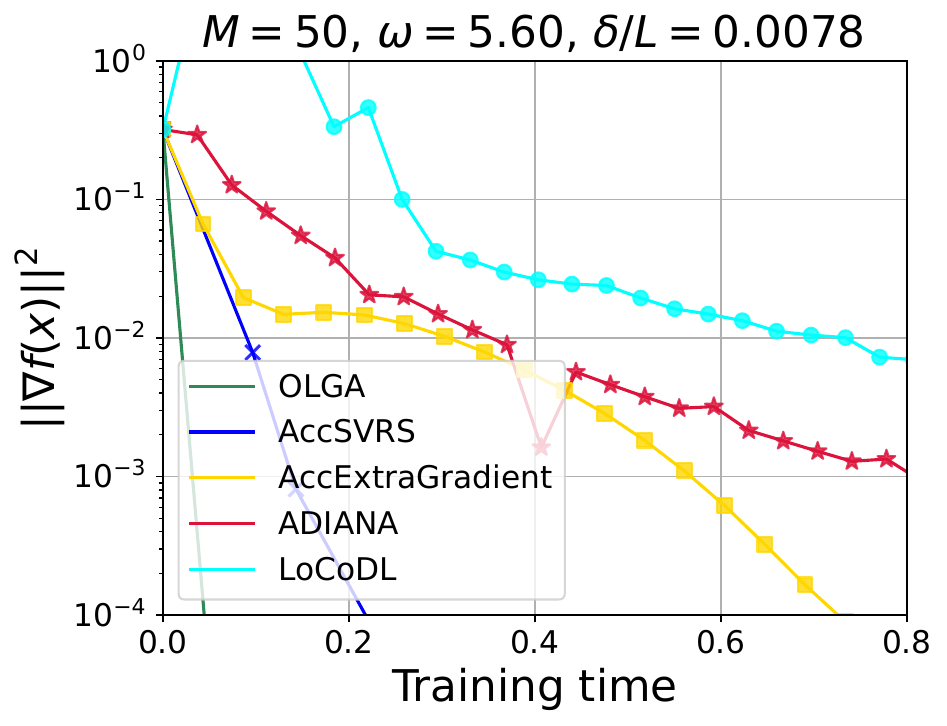}
   \end{subfigure}\\
   \begin{subfigure}{0.310\textwidth}
       \centering
       (a) $\omega = 112.00$
   \end{subfigure}
   \begin{subfigure}{0.310\textwidth}
       \centering
       (b) $\omega = 11.20$
   \end{subfigure}
   \begin{subfigure}{0.310\textwidth}
       \centering
       (c) $\omega = 5.6$
   \end{subfigure}
   \caption{Comparison of state-of-the-art distributed methods. The comparison is made on \eqref{eq:logloss} with $M=50$ and \texttt{mushrooms} dataset. The criterion is the training time on remote CPUs (slow connection). For methods with compression we vary the power of compression $\omega$.}
\end{figure}

\begin{figure}[h!] 
   \begin{subfigure}{0.310\textwidth}
       \includegraphics[width=\linewidth]{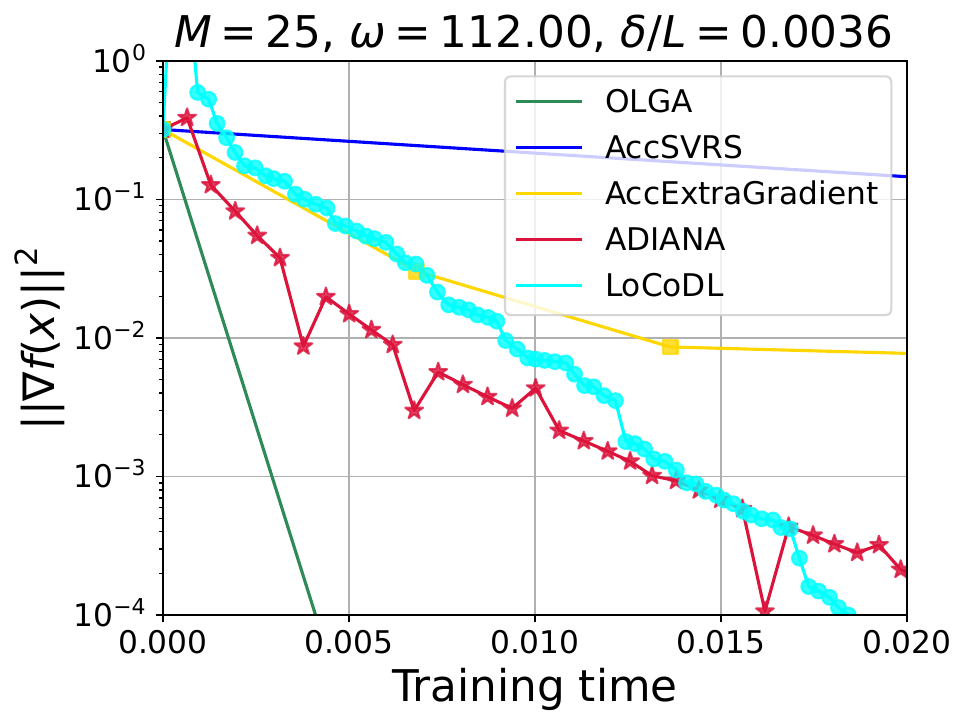}
   \end{subfigure}
   \begin{subfigure}{0.310\textwidth}
       \includegraphics[width=\linewidth]{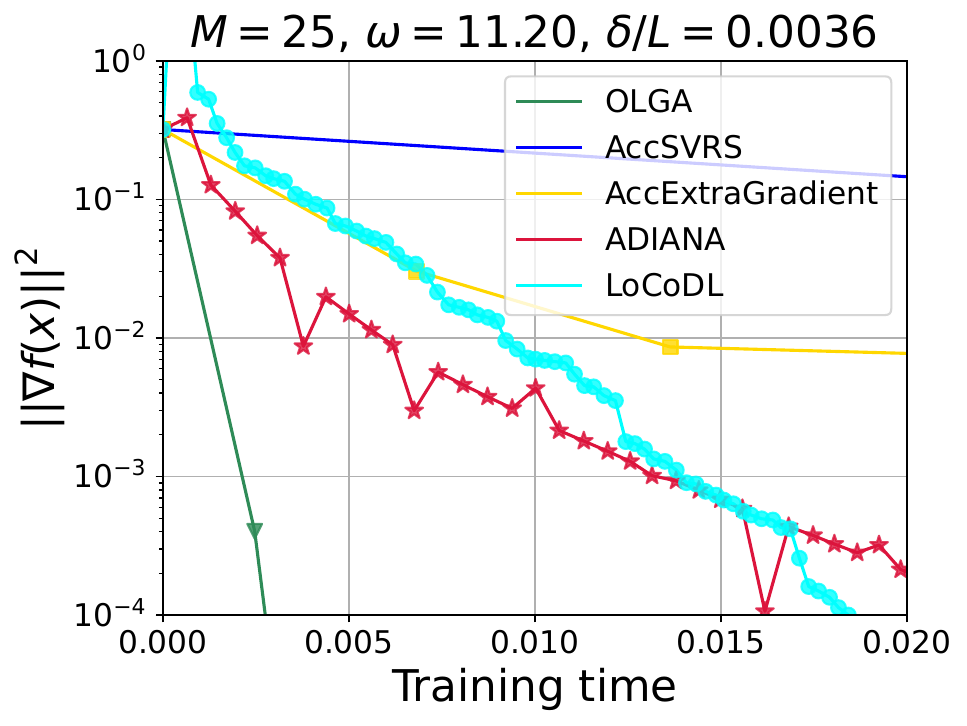}
   \end{subfigure}
   \begin{subfigure}{0.310\textwidth}
       \includegraphics[width=\linewidth]{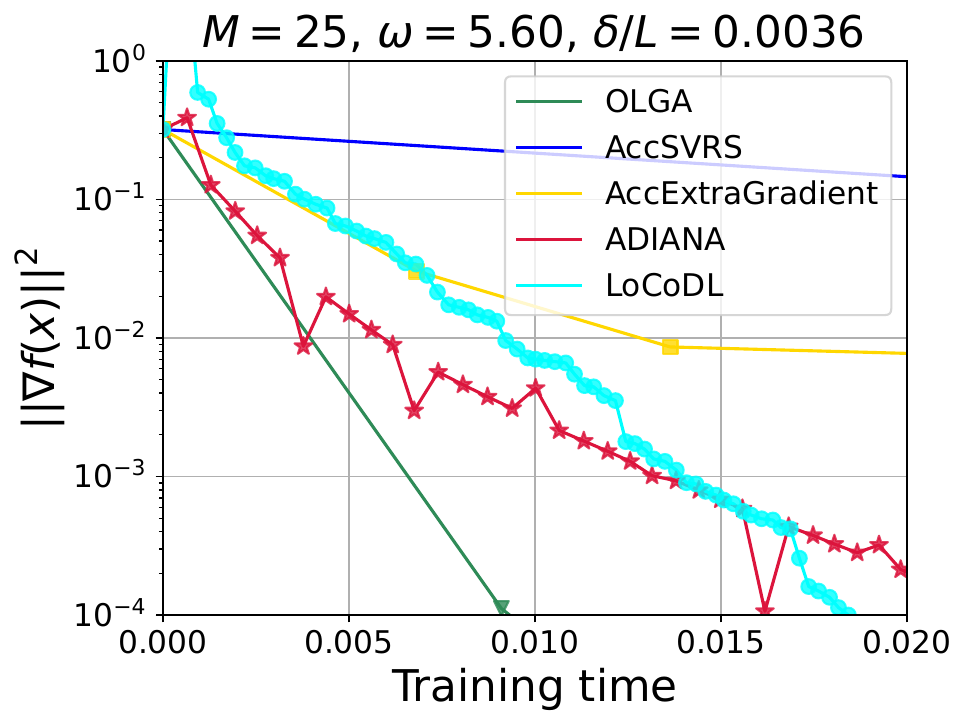}
   \end{subfigure}\\
   \begin{subfigure}{0.310\textwidth}
       \centering
       (a) $\omega = 112.00$
   \end{subfigure}
   \begin{subfigure}{0.310\textwidth}
       \centering
       (b) $\omega = 11.20$
   \end{subfigure}
   \begin{subfigure}{0.310\textwidth}
       \centering
       (c) $\omega = 5.6$
   \end{subfigure}
   \caption{Comparison of state-of-the-art distributed methods. The comparison is made on \eqref{eq:logloss} with $M=25$ and \texttt{mushrooms} dataset. The criterion is the training time on local cluster (fast connection). For methods with compression we vary the power of compression $\omega$.}
\end{figure}

\begin{figure}[h!] 
   \begin{subfigure}{0.310\textwidth}
       \includegraphics[width=\linewidth]{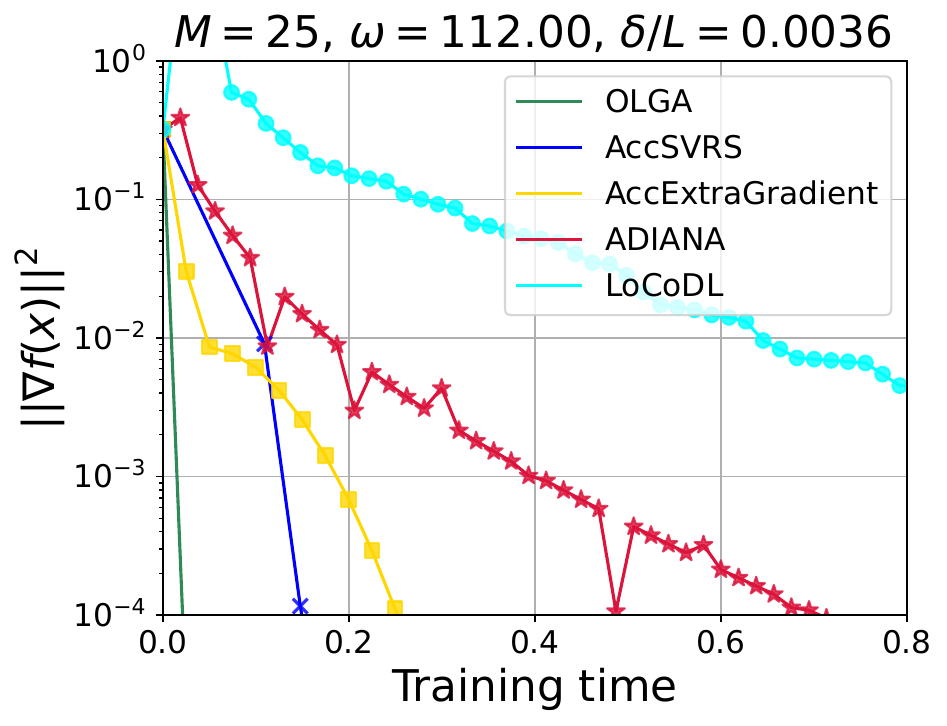}
   \end{subfigure}
   \begin{subfigure}{0.310\textwidth}
       \includegraphics[width=\linewidth]{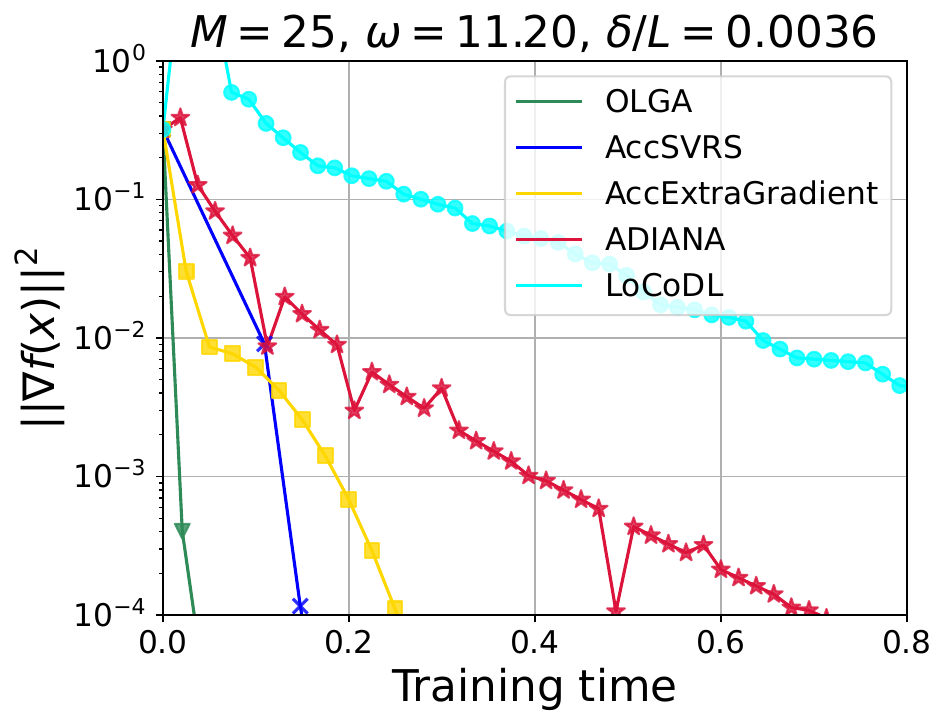}
   \end{subfigure}
   \begin{subfigure}{0.310\textwidth}
       \includegraphics[width=\linewidth]{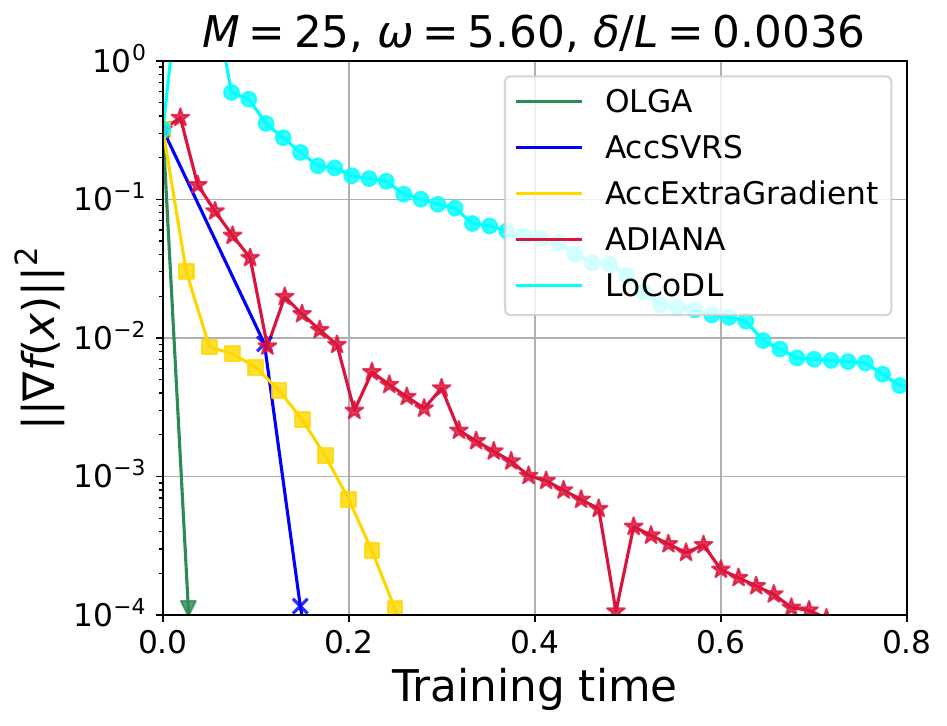}
   \end{subfigure}\\
   \begin{subfigure}{0.310\textwidth}
       \centering
       (a) $\omega = 112.00$
   \end{subfigure}
   \begin{subfigure}{0.310\textwidth}
       \centering
       (b) $\omega = 11.20$
   \end{subfigure}
   \begin{subfigure}{0.310\textwidth}
       \centering
       (c) $\omega = 5.6$
   \end{subfigure}
   \caption{Comparison of state-of-the-art distributed methods. The comparison is made on \eqref{eq:logloss} with $M=25$ and \texttt{mushrooms} dataset. The criterion is the training time on remote CPUs (slow connection). For methods with compression we vary the power of compression $\omega$.}
\end{figure}

\begin{figure}[h!] 
   \begin{subfigure}{0.310\textwidth}
       \includegraphics[width=\linewidth]{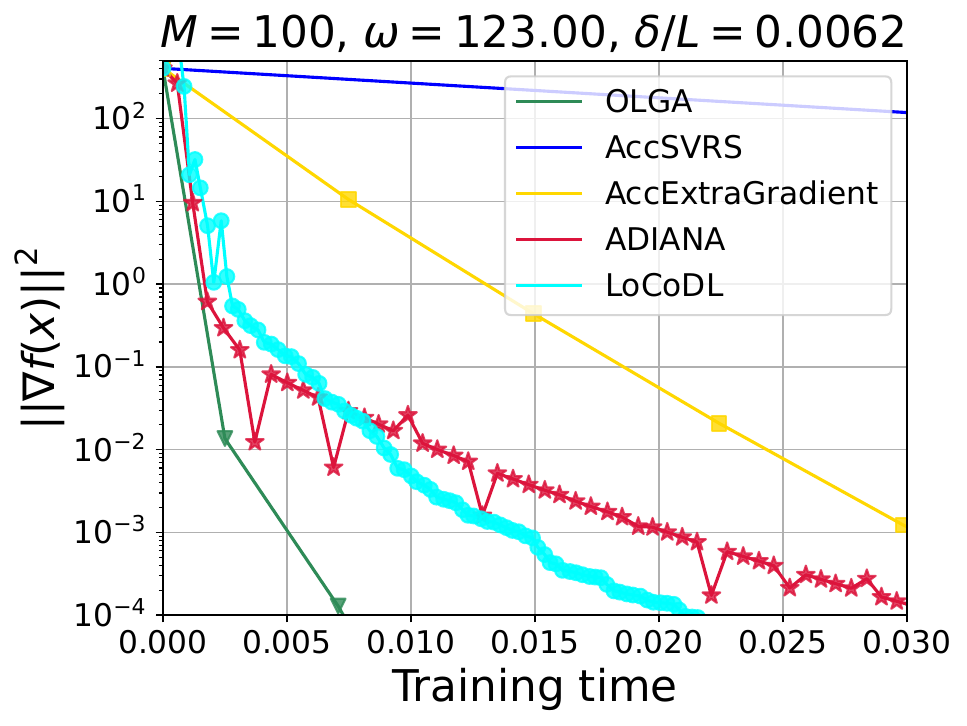}
   \end{subfigure}
   \begin{subfigure}{0.310\textwidth}
       \includegraphics[width=\linewidth]{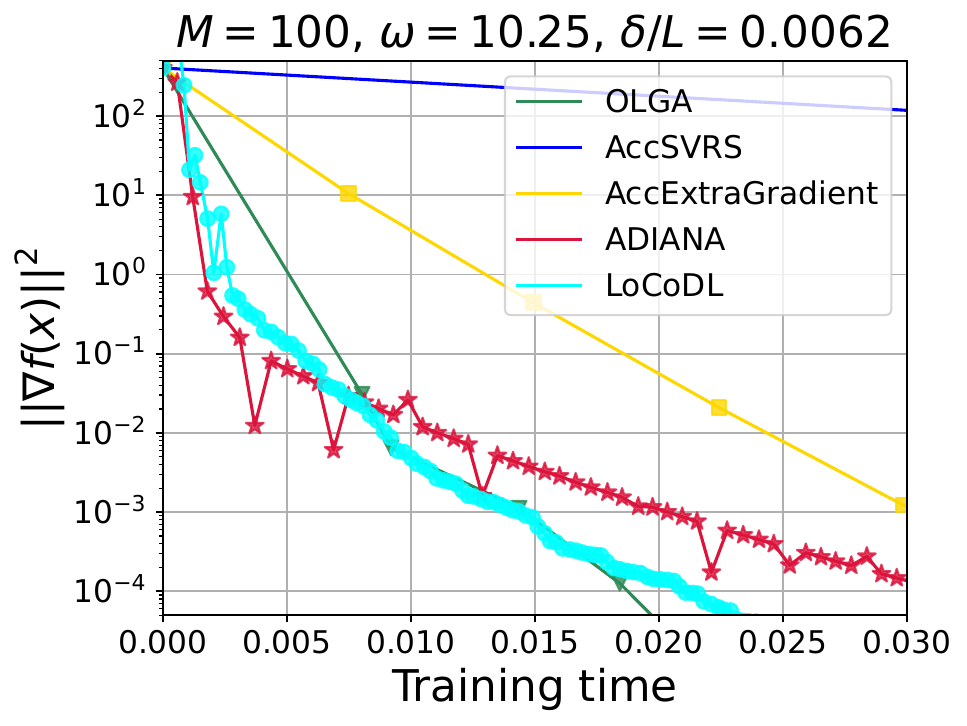}
   \end{subfigure}
   \begin{subfigure}{0.310\textwidth}
       \includegraphics[width=\linewidth]{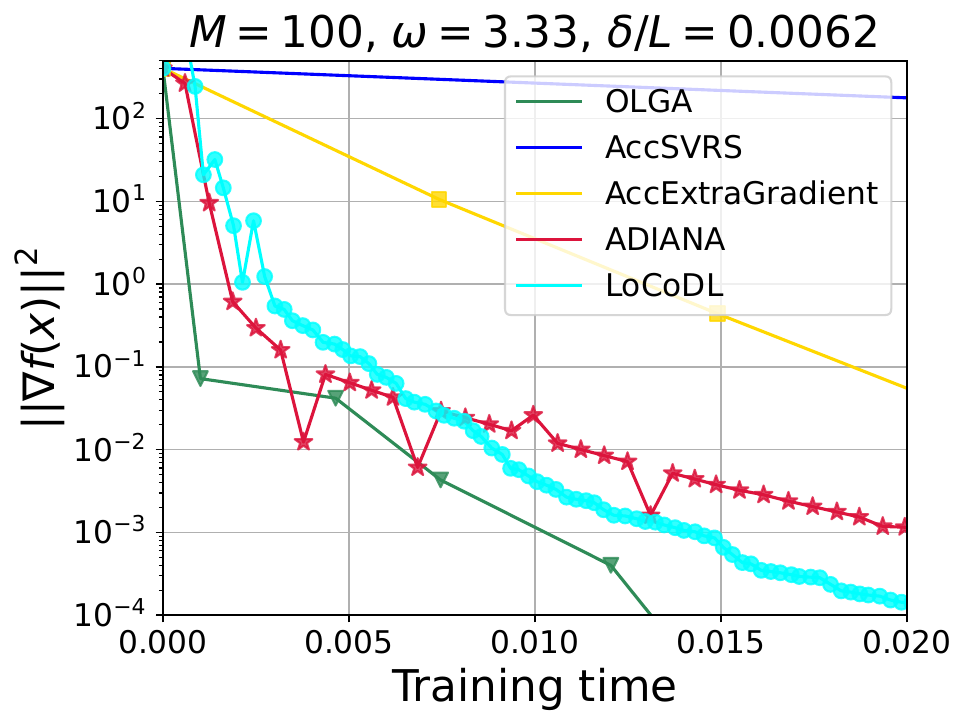}
   \end{subfigure}\\
   \begin{subfigure}{0.310\textwidth}
       \centering
       (a) $\omega = 123.00$
   \end{subfigure}
   \begin{subfigure}{0.310\textwidth}
       \centering
       (b) $\omega = 10.25$
   \end{subfigure}
   \begin{subfigure}{0.310\textwidth}
       \centering
       (c) $\omega = 3.33$
   \end{subfigure}
   \caption{Comparison of state-of-the-art distributed methods. The comparison is made on \eqref{eq:quadr} with $M=100$ and \texttt{a9a} dataset. The criterion is the training time on local cluster (fast connection). For methods with compression we vary the power of compression $\omega$.}
\end{figure}

\begin{figure}[h!] 
   \begin{subfigure}{0.310\textwidth}
       \includegraphics[width=\linewidth]{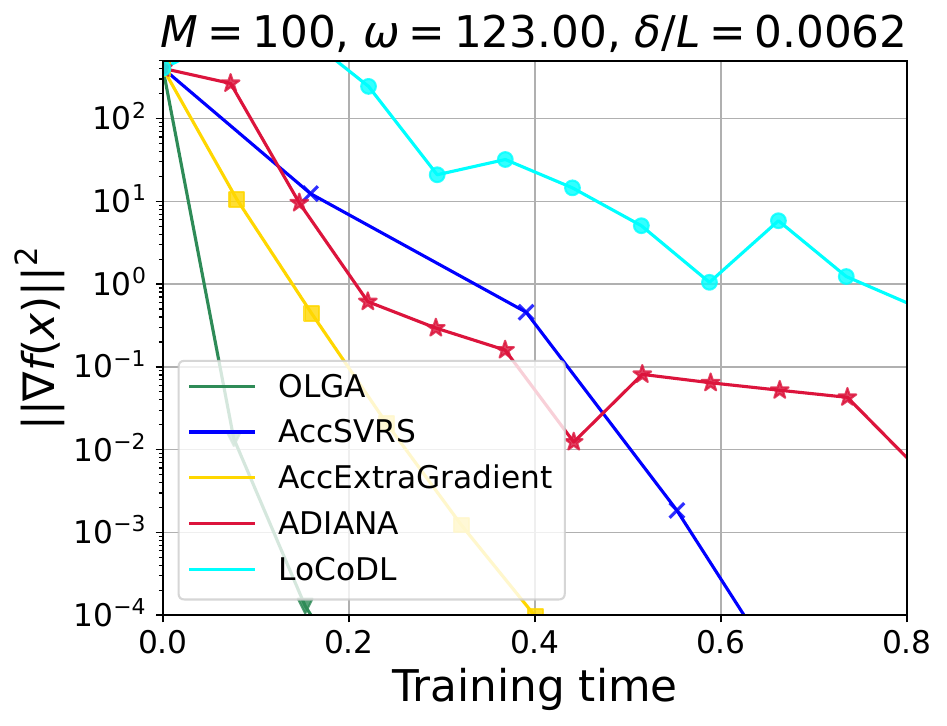}
   \end{subfigure}
   \begin{subfigure}{0.310\textwidth}
       \includegraphics[width=\linewidth]{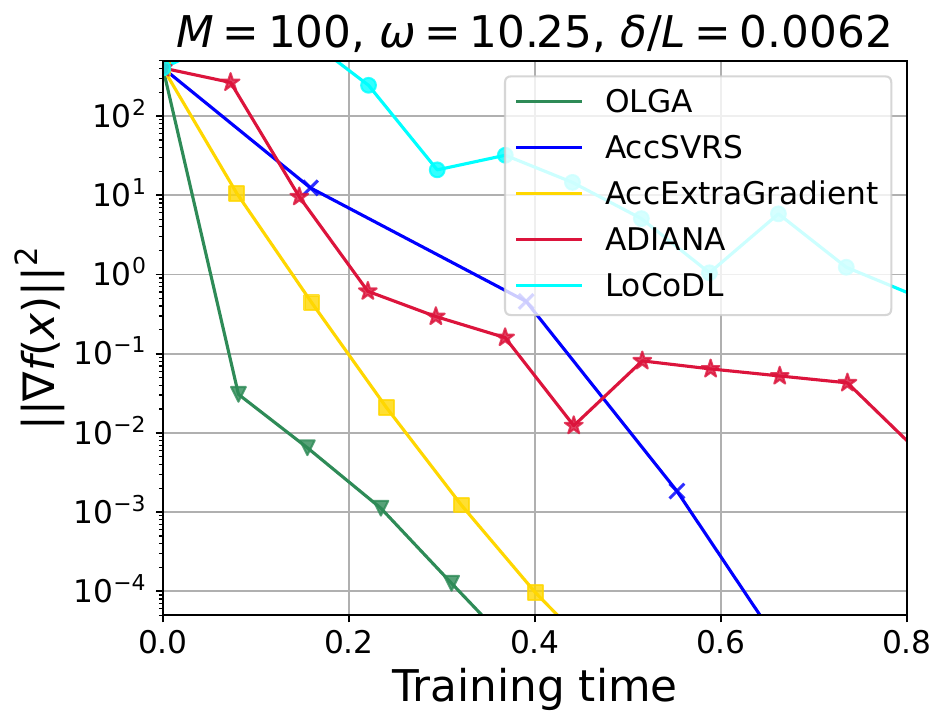}
   \end{subfigure}
   \begin{subfigure}{0.310\textwidth}
       \includegraphics[width=\linewidth]{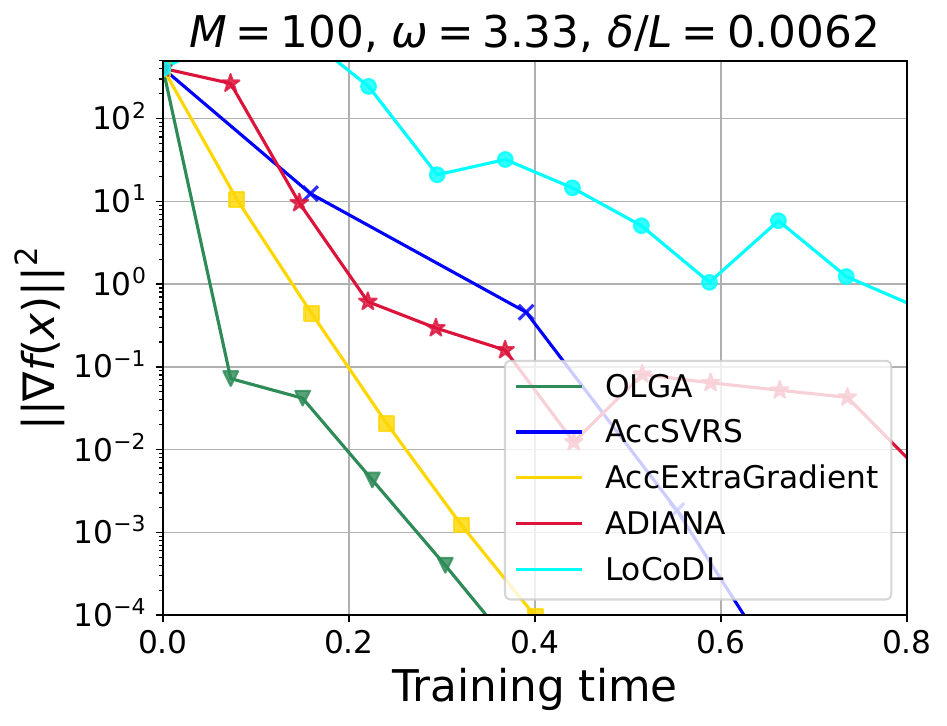}
   \end{subfigure}\\
   \begin{subfigure}{0.310\textwidth}
       \centering
       (a) $\omega = 123.00$
   \end{subfigure}
   \begin{subfigure}{0.310\textwidth}
       \centering
       (b) $\omega = 10.25$
   \end{subfigure}
   \begin{subfigure}{0.310\textwidth}
       \centering
       (c) $\omega = 3.33$
   \end{subfigure}
   \caption{Comparison of state-of-the-art distributed methods. The comparison is made on \eqref{eq:quadr} with $M=100$ and \texttt{a9a} dataset. The criterion is the training time on remote CPUs (slow connection). For methods with compression we vary the power of compression $\omega$.}
\end{figure}

\begin{figure}[h!] 
   \begin{subfigure}{0.310\textwidth}
       \includegraphics[width=\linewidth]{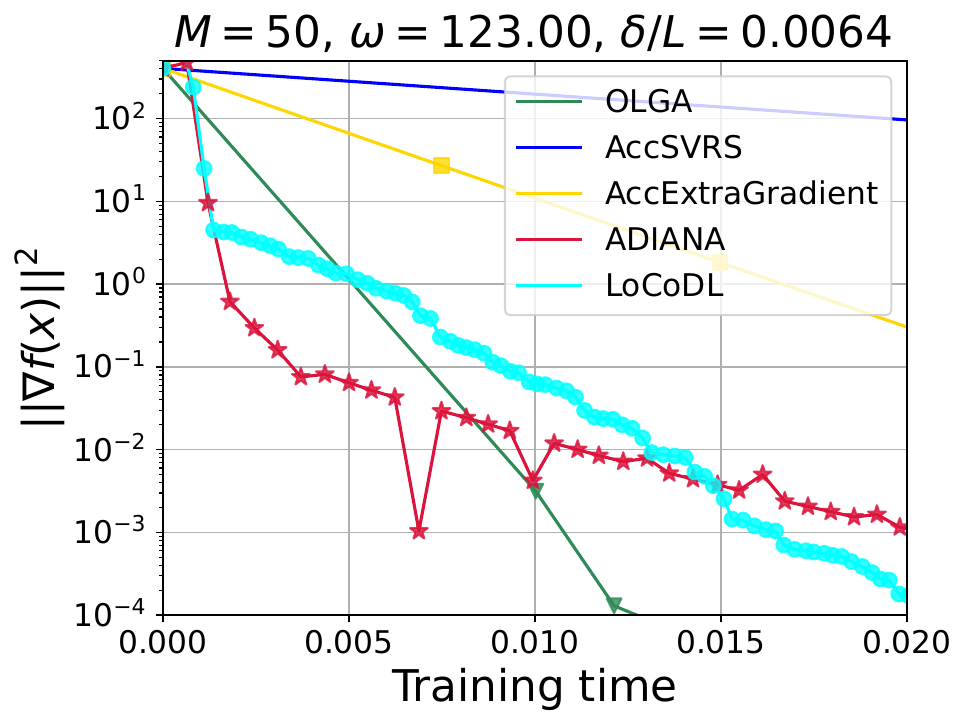}
   \end{subfigure}
   \begin{subfigure}{0.310\textwidth}
       \includegraphics[width=\linewidth]{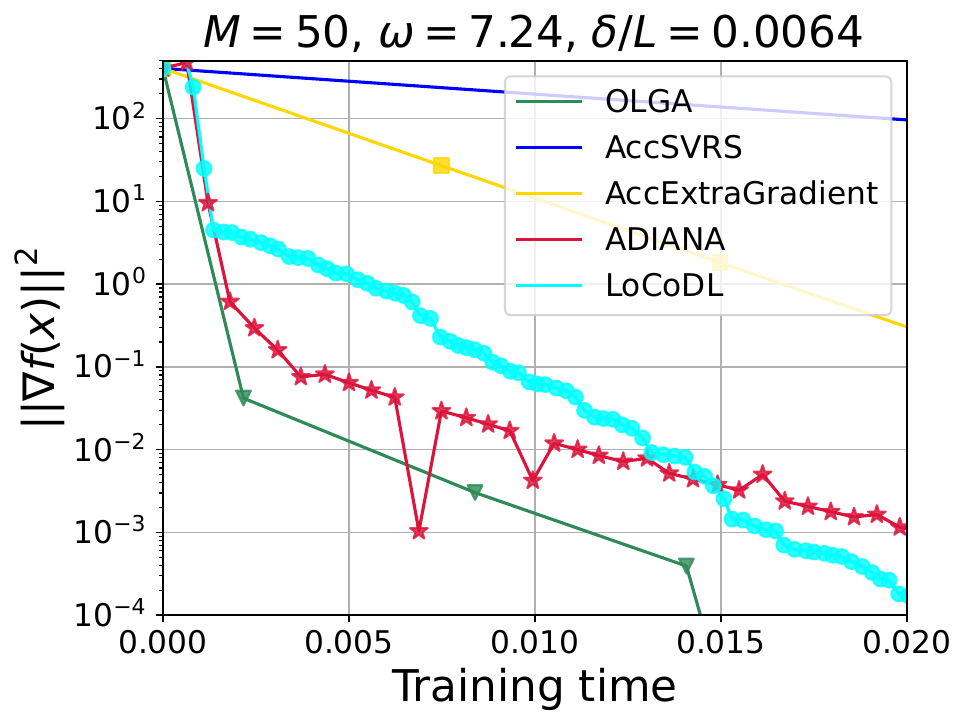}
   \end{subfigure}
   \begin{subfigure}{0.310\textwidth}
       \includegraphics[width=\linewidth]{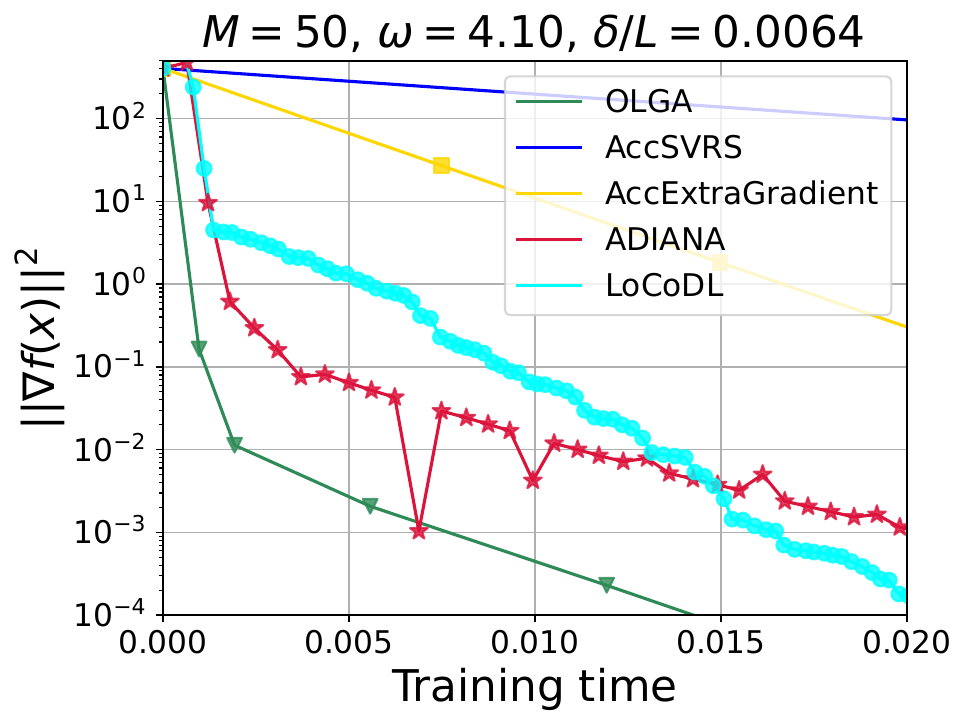}
   \end{subfigure}\\
   \begin{subfigure}{0.310\textwidth}
       \centering
       (a) $\omega = 123.00$
   \end{subfigure}
   \begin{subfigure}{0.310\textwidth}
       \centering
       (b) $\omega = 7.24$
   \end{subfigure}
   \begin{subfigure}{0.310\textwidth}
       \centering
       (c) $\omega = 4.10$
   \end{subfigure}
   \caption{Comparison of state-of-the-art distributed methods. The comparison is made on \eqref{eq:quadr} with $M=50$ and \texttt{a9a} dataset. The criterion is the training time on local cluster (fast connection). For methods with compression we vary the power of compression $\omega$.}
\end{figure}

\begin{figure}[h!] 
   \begin{subfigure}{0.310\textwidth}
       \includegraphics[width=\linewidth]{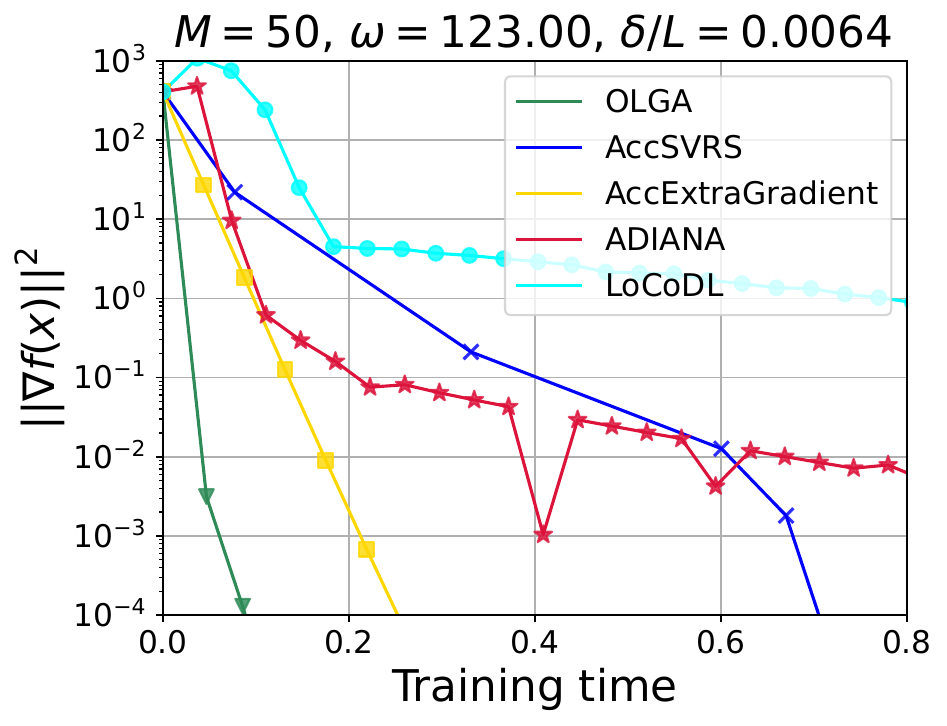}
   \end{subfigure}
   \begin{subfigure}{0.310\textwidth}
       \includegraphics[width=\linewidth]{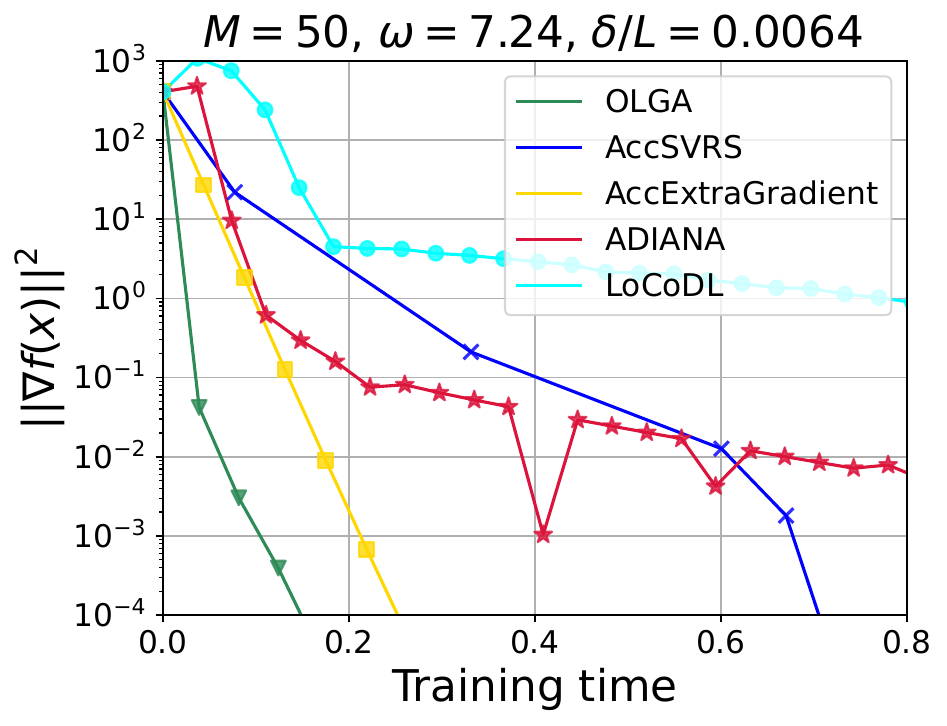}
   \end{subfigure}
   \begin{subfigure}{0.310\textwidth}
       \includegraphics[width=\linewidth]{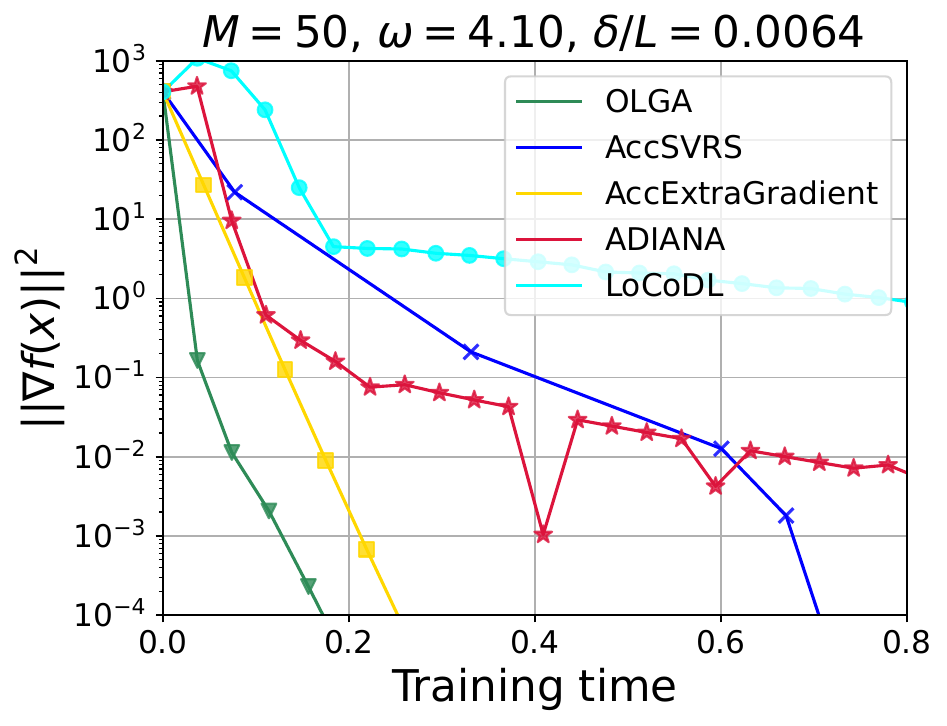}
   \end{subfigure}\\
   \begin{subfigure}{0.310\textwidth}
       \centering
       (a) $\omega = 123.00$
   \end{subfigure}
   \begin{subfigure}{0.310\textwidth}
       \centering
       (b) $\omega = 7.24$
   \end{subfigure}
   \begin{subfigure}{0.310\textwidth}
       \centering
       (c) $\omega = 4.10$
   \end{subfigure}
   \caption{Comparison of state-of-the-art distributed methods. The comparison is made on \eqref{eq:quadr} with $M=50$ and \texttt{a9a} dataset. The criterion is the training time on remote CPUs (slow connection). For methods with compression we vary the power of compression $\omega$.}
\end{figure}

\begin{figure}[h!] 
   \begin{subfigure}{0.310\textwidth}
       \includegraphics[width=\linewidth]{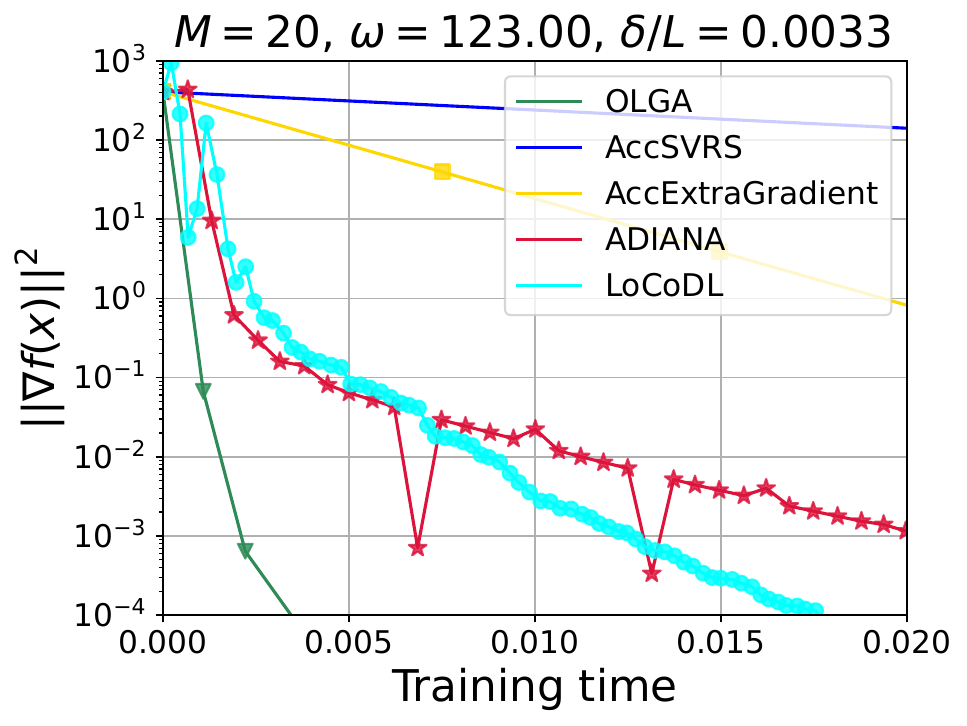}
   \end{subfigure}
   \begin{subfigure}{0.310\textwidth}
       \includegraphics[width=\linewidth]{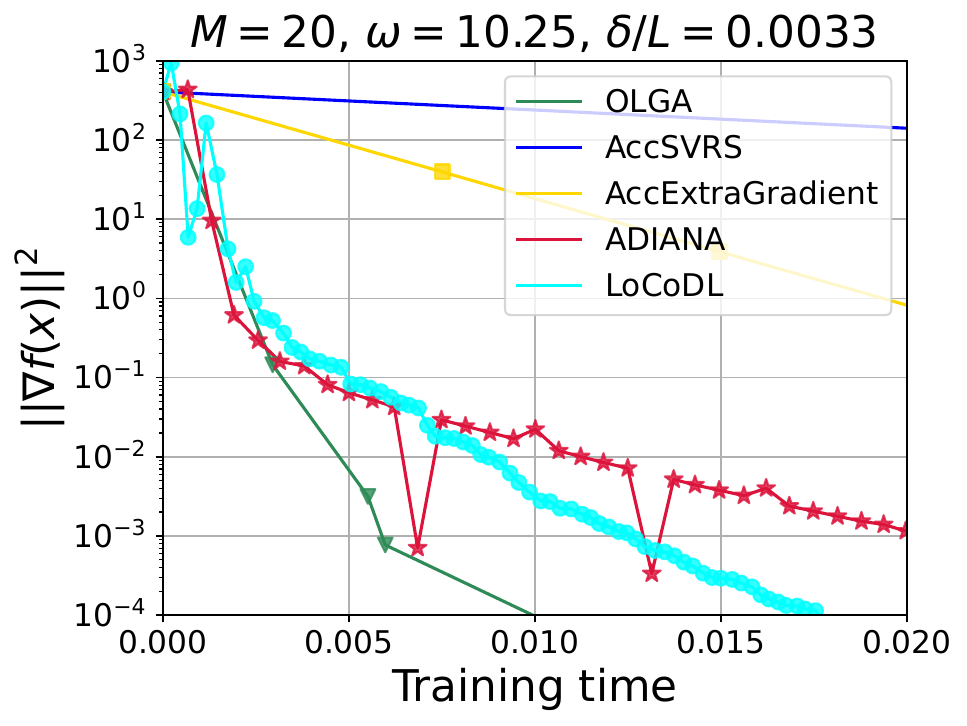}
   \end{subfigure}
   \begin{subfigure}{0.310\textwidth}
       \includegraphics[width=\linewidth]{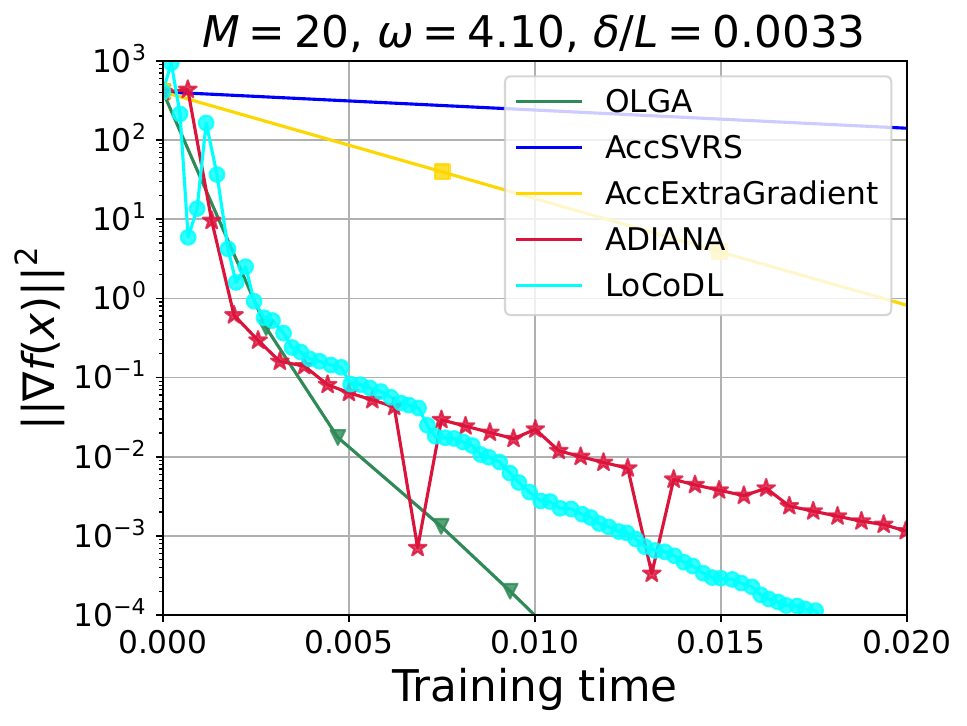}
   \end{subfigure}\\
   \begin{subfigure}{0.310\textwidth}
       \centering
       (a) $\omega = 123.00$
   \end{subfigure}
   \begin{subfigure}{0.310\textwidth}
       \centering
       (b) $\omega = 10.25$
   \end{subfigure}
   \begin{subfigure}{0.310\textwidth}
       \centering
       (c) $\omega = 4.10$
   \end{subfigure}
   \caption{Comparison of state-of-the-art distributed methods. The comparison is made on \eqref{eq:quadr} with $M=20$ and \texttt{a9a} dataset. The criterion is the training time on local cluster (fast connection). For methods with compression we vary the power of compression $\omega$.}
\end{figure}

\begin{figure}[h!] 
   \begin{subfigure}{0.310\textwidth}
       \includegraphics[width=\linewidth]{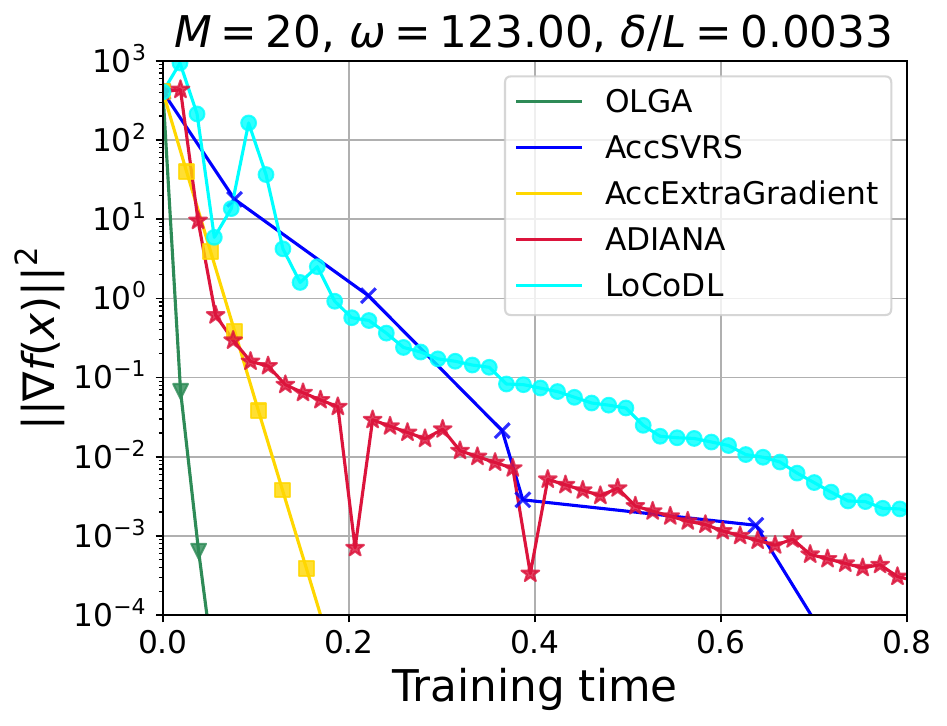}
   \end{subfigure}
   \begin{subfigure}{0.310\textwidth}
       \includegraphics[width=\linewidth]{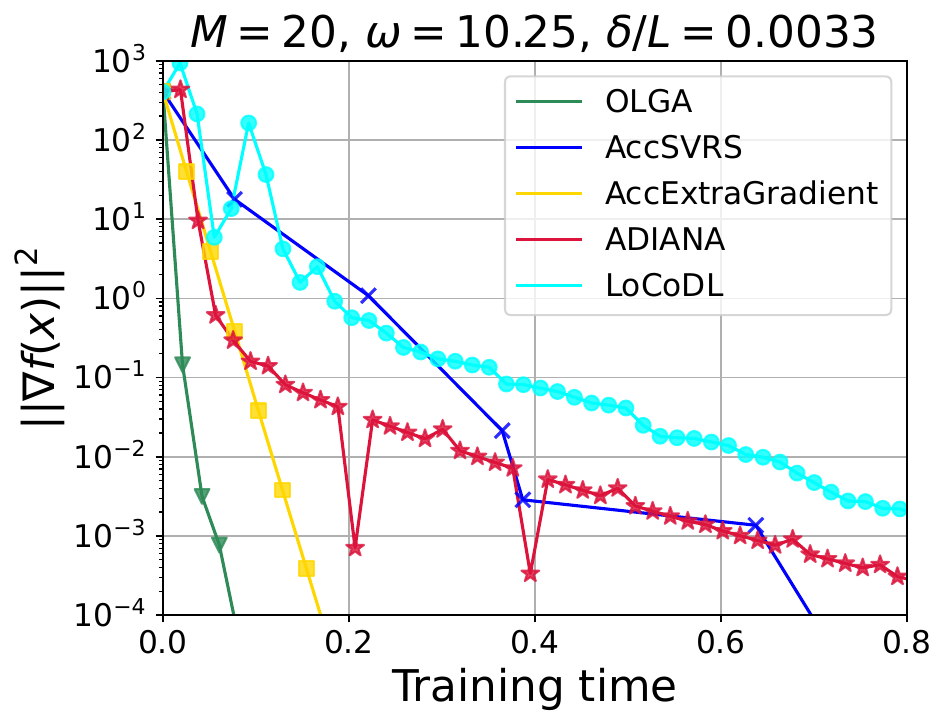}
   \end{subfigure}
   \begin{subfigure}{0.310\textwidth}
       \includegraphics[width=\linewidth]{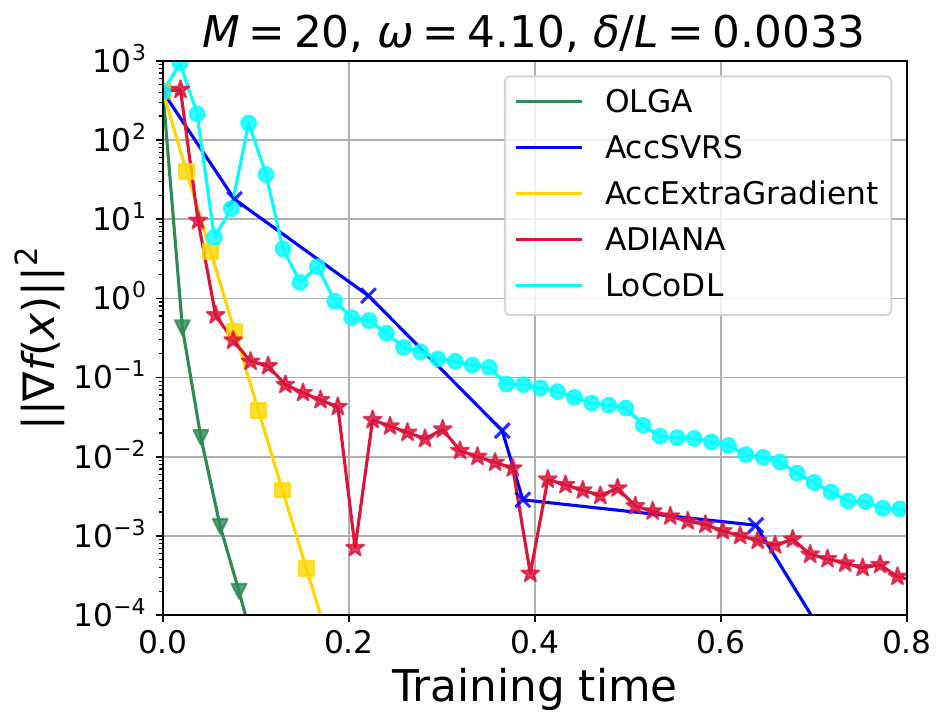}
   \end{subfigure}\\
   \begin{subfigure}{0.310\textwidth}
       \centering
       (a) $\omega = 123.00$
   \end{subfigure}
   \begin{subfigure}{0.310\textwidth}
       \centering
       (b) $\omega = 10.25$
   \end{subfigure}
   \begin{subfigure}{0.310\textwidth}
       \centering
       (c) $\omega = 4.10$
   \end{subfigure}
   \caption{Comparison of state-of-the-art distributed methods. The comparison is made on \eqref{eq:quadr} with $M=20$ and \texttt{a9a} dataset. The criterion is the training time on remote CPUs (slow connection). For methods with compression we vary the power of compression $\omega$.}
\end{figure}

\begin{figure}[h!] 
   \begin{subfigure}{0.310\textwidth}
       \includegraphics[width=\linewidth]{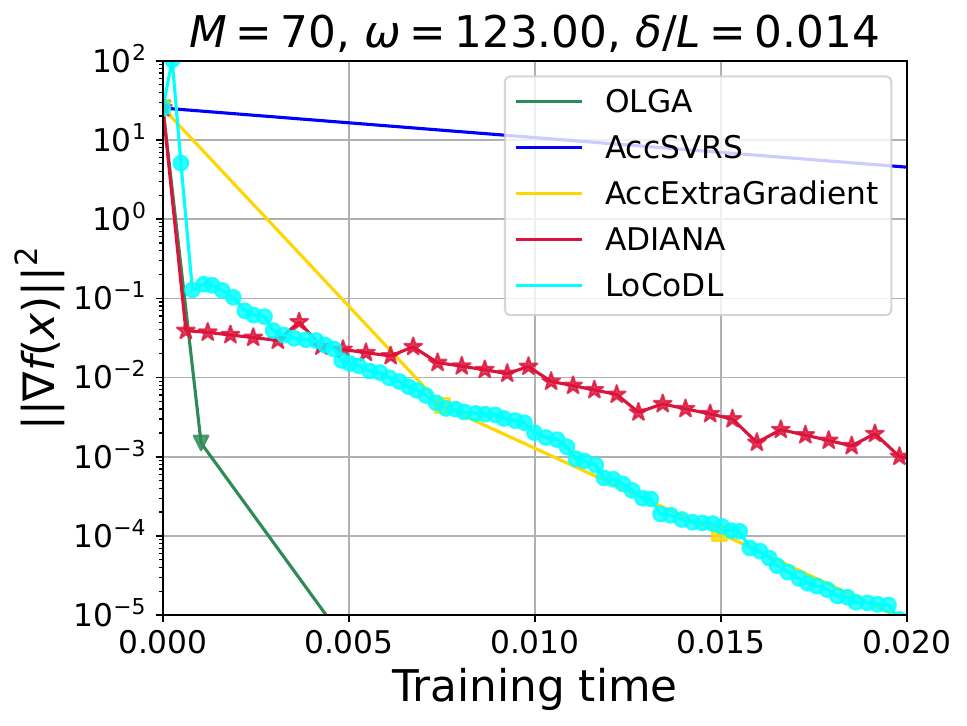}
   \end{subfigure}
   \begin{subfigure}{0.310\textwidth}
       \includegraphics[width=\linewidth]{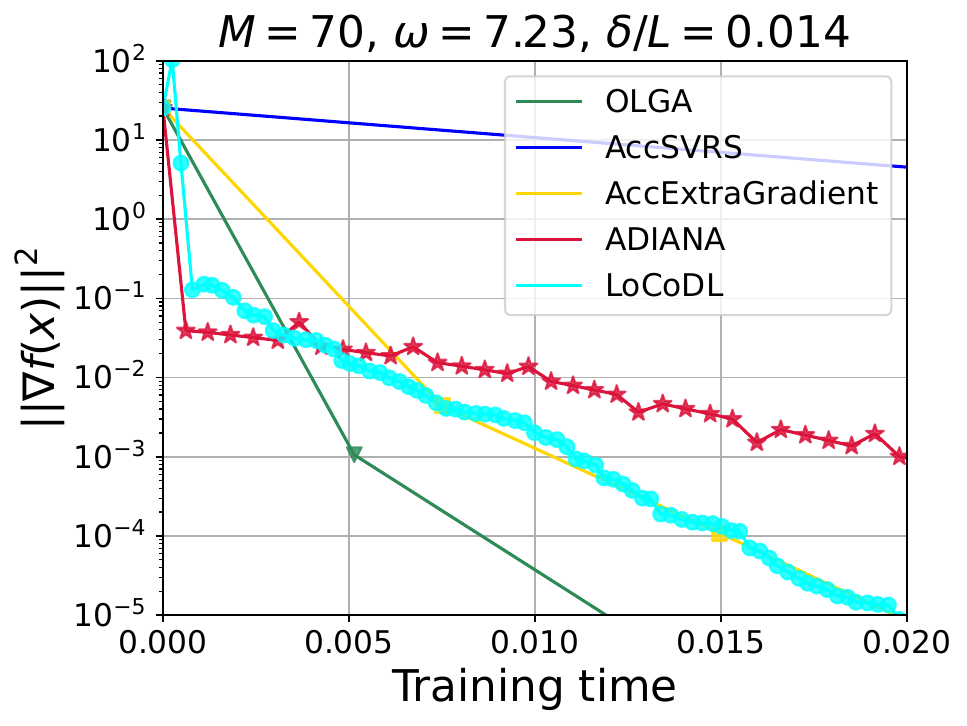}
   \end{subfigure}
   \begin{subfigure}{0.310\textwidth}
       \includegraphics[width=\linewidth]{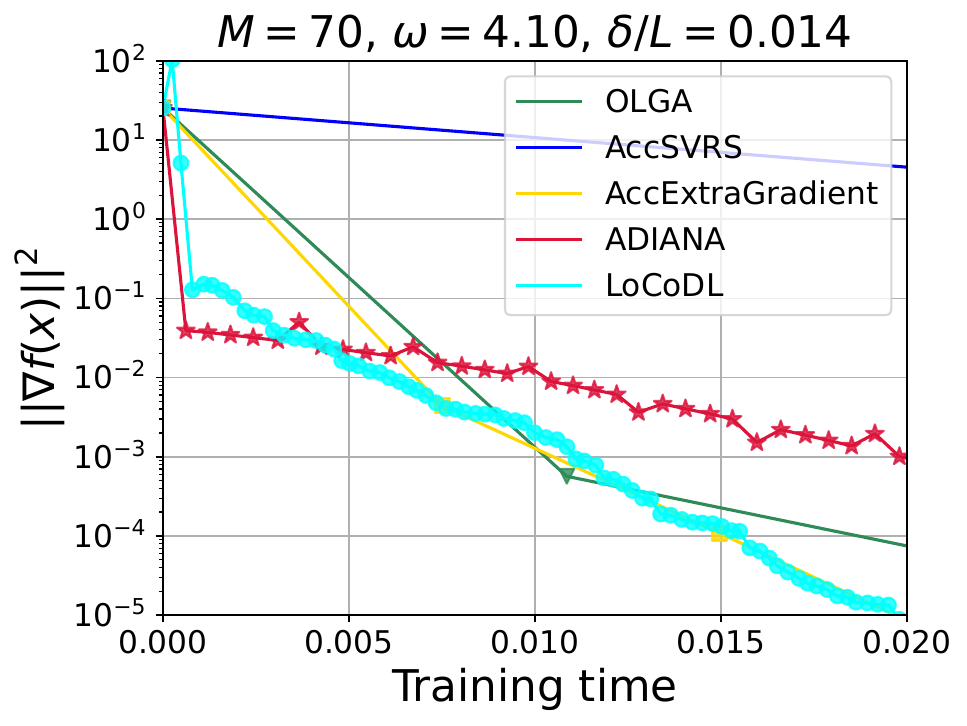}
   \end{subfigure}\\
   \begin{subfigure}{0.310\textwidth}
       \centering
       (a) $\omega = 123.00$
   \end{subfigure}
   \begin{subfigure}{0.310\textwidth}
       \centering
       (b) $\omega = 7.23$
   \end{subfigure}
   \begin{subfigure}{0.310\textwidth}
       \centering
       (c) $\omega = 4.10$
   \end{subfigure}
   \caption{Comparison of state-of-the-art distributed methods. The comparison is made on \eqref{eq:logloss} with $M=70$ and \texttt{a9a} dataset. The criterion is the training time on local cluster (fast connection). For methods with compression we vary the power of compression $\omega$.}
\end{figure}

\begin{figure}[h!] 
   \begin{subfigure}{0.310\textwidth}
       \includegraphics[width=\linewidth]{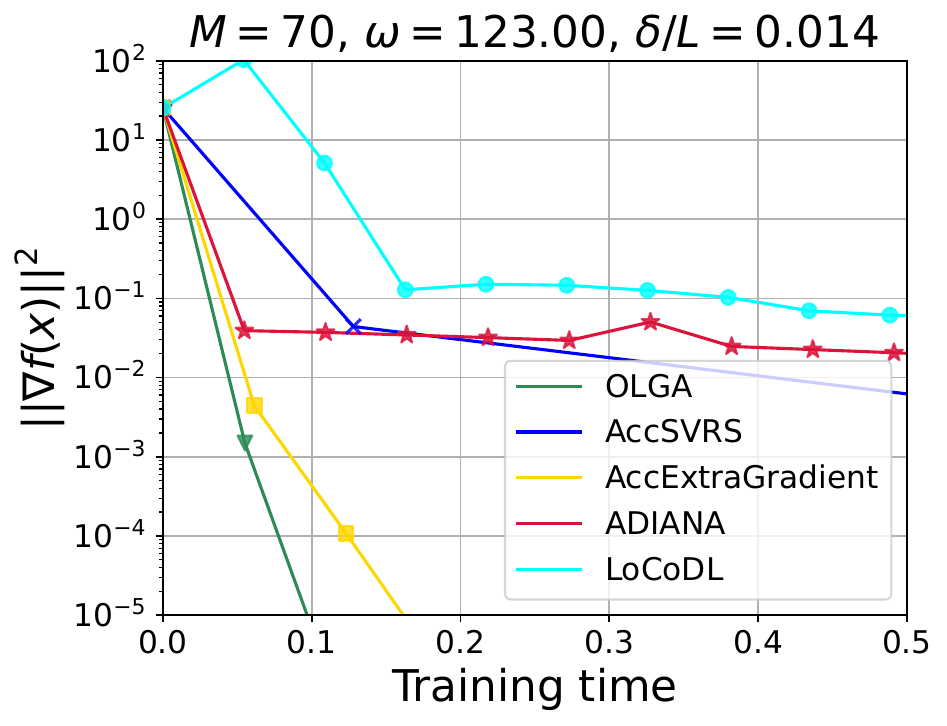}
   \end{subfigure}
   \begin{subfigure}{0.310\textwidth}
       \includegraphics[width=\linewidth]{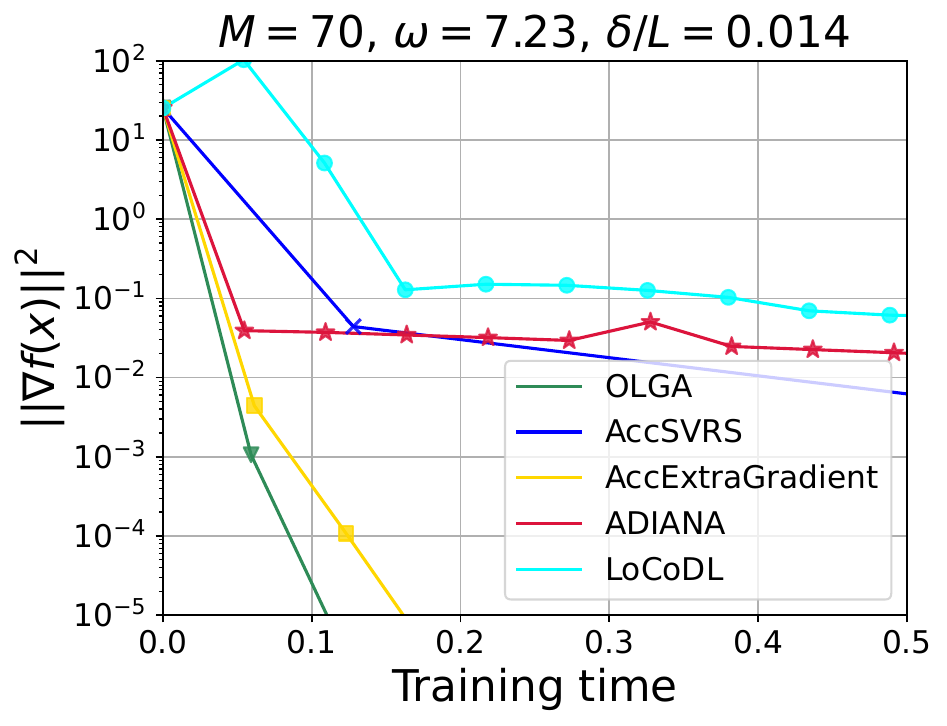}
   \end{subfigure}
   \begin{subfigure}{0.310\textwidth}
       \includegraphics[width=\linewidth]{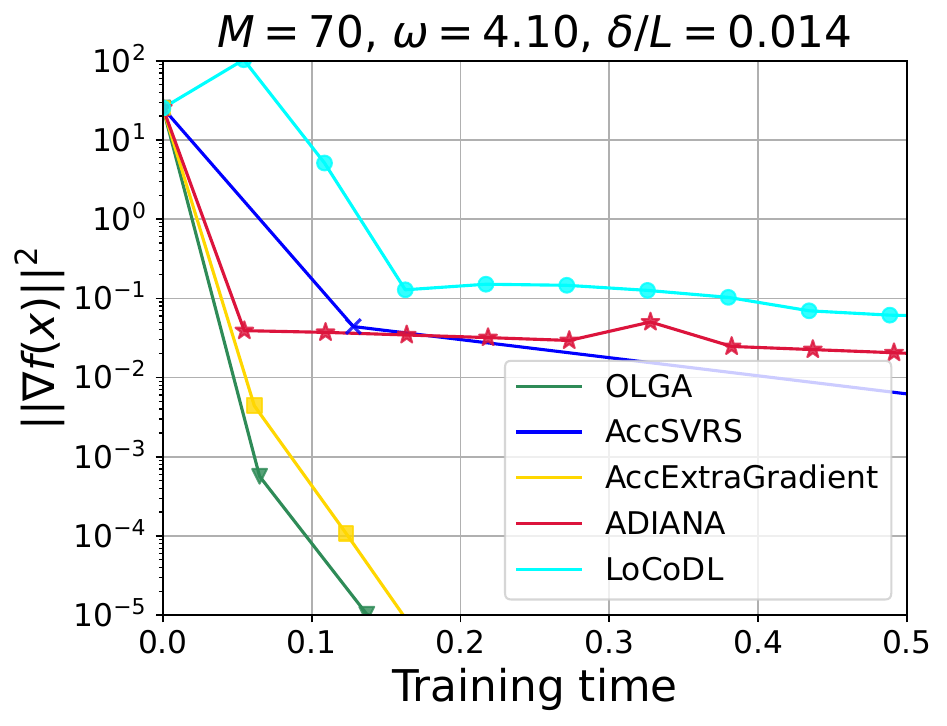}
   \end{subfigure}\\
   \begin{subfigure}{0.310\textwidth}
       \centering
       (a) $\omega = 123.00$
   \end{subfigure}
   \begin{subfigure}{0.310\textwidth}
       \centering
       (b) $\omega = 7.23$
   \end{subfigure}
   \begin{subfigure}{0.310\textwidth}
       \centering
       (c) $\omega = 4.10$
   \end{subfigure}
   \caption{Comparison of state-of-the-art distributed methods. The comparison is made on \eqref{eq:logloss} with $M=70$ and \texttt{a9a} dataset. The criterion is the training time on remote CPUs (slow connection). For methods with compression we vary the power of compression $\omega$.}
\end{figure}

\begin{figure}[h!] 
   \begin{subfigure}{0.310\textwidth}
       \includegraphics[width=\linewidth]{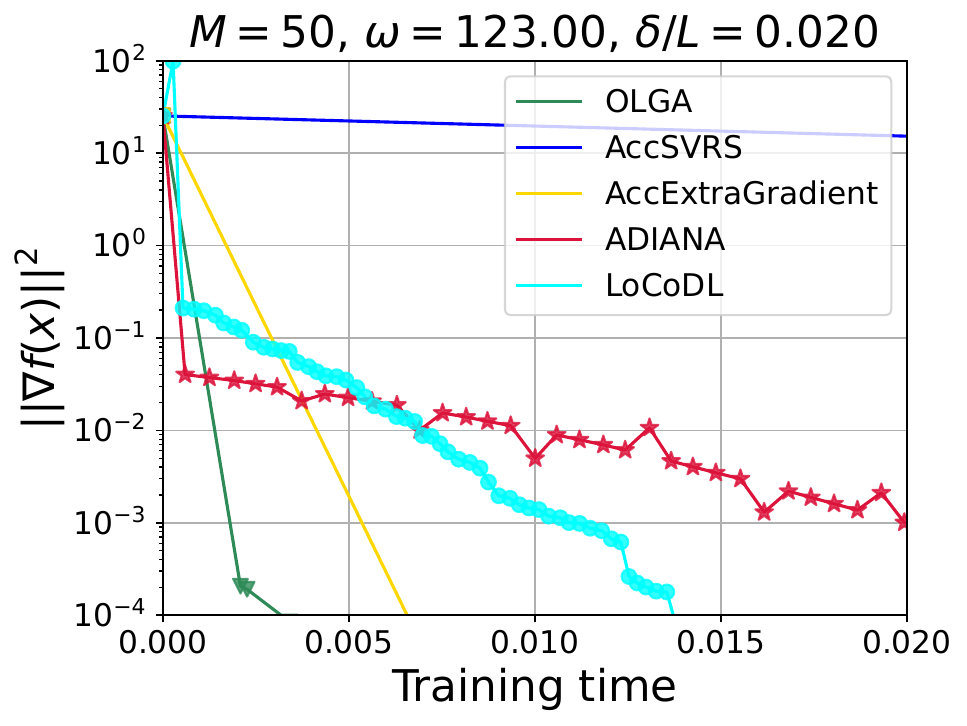}
   \end{subfigure}
   \begin{subfigure}{0.310\textwidth}
       \includegraphics[width=\linewidth]{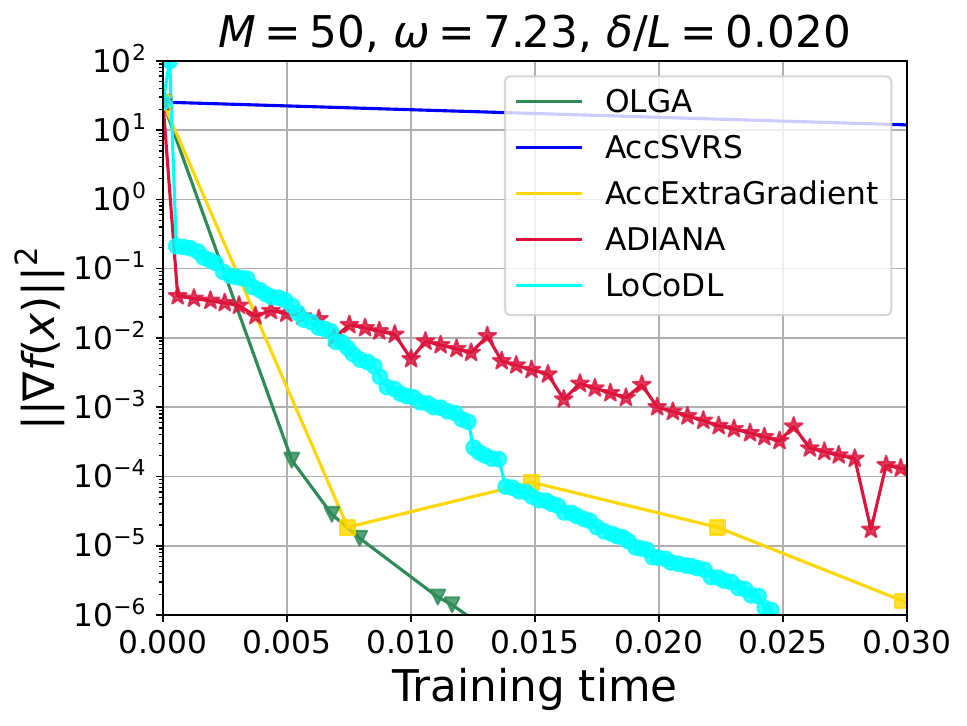}
   \end{subfigure}
   \begin{subfigure}{0.310\textwidth}
       \includegraphics[width=\linewidth]{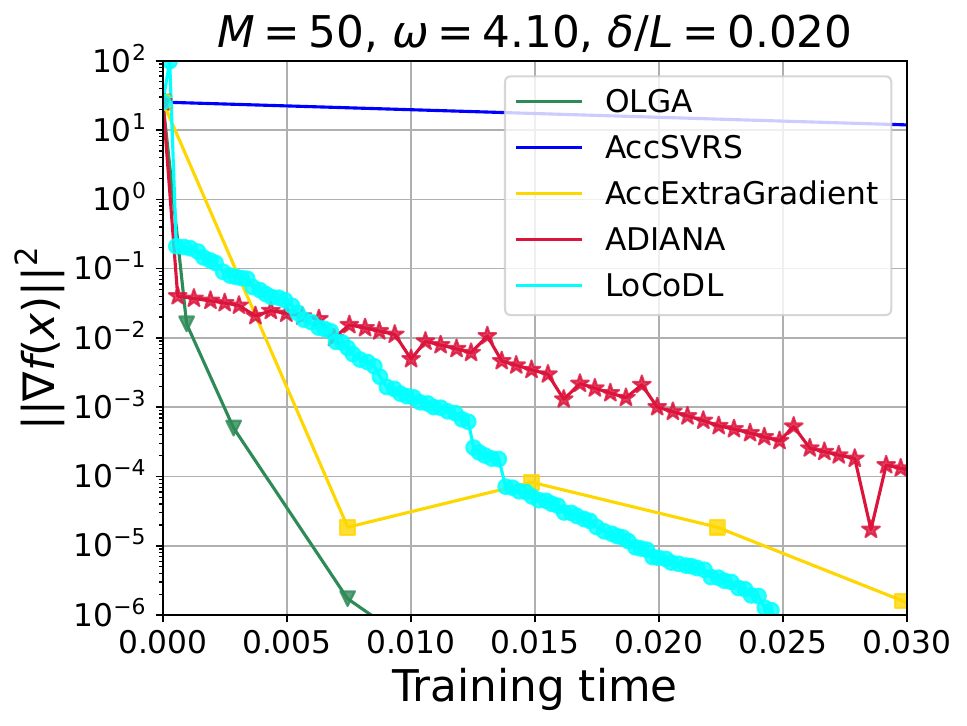}
   \end{subfigure}\\
   \begin{subfigure}{0.310\textwidth}
       \centering
       (a) $\omega = 123.00$
   \end{subfigure}
   \begin{subfigure}{0.310\textwidth}
       \centering
       (b) $\omega = 7.23$
   \end{subfigure}
   \begin{subfigure}{0.310\textwidth}
       \centering
       (c) $\omega = 4.10$
   \end{subfigure}
   \caption{Comparison of state-of-the-art distributed methods. The comparison is made on \eqref{eq:logloss} with $M=50$ and \texttt{a9a} dataset. The criterion is the training time on local cluster (fast connection). For methods with compression we vary the power of compression $\omega$.}
\end{figure}

\begin{figure}[h!] 
   \begin{subfigure}{0.310\textwidth}
       \includegraphics[width=\linewidth]{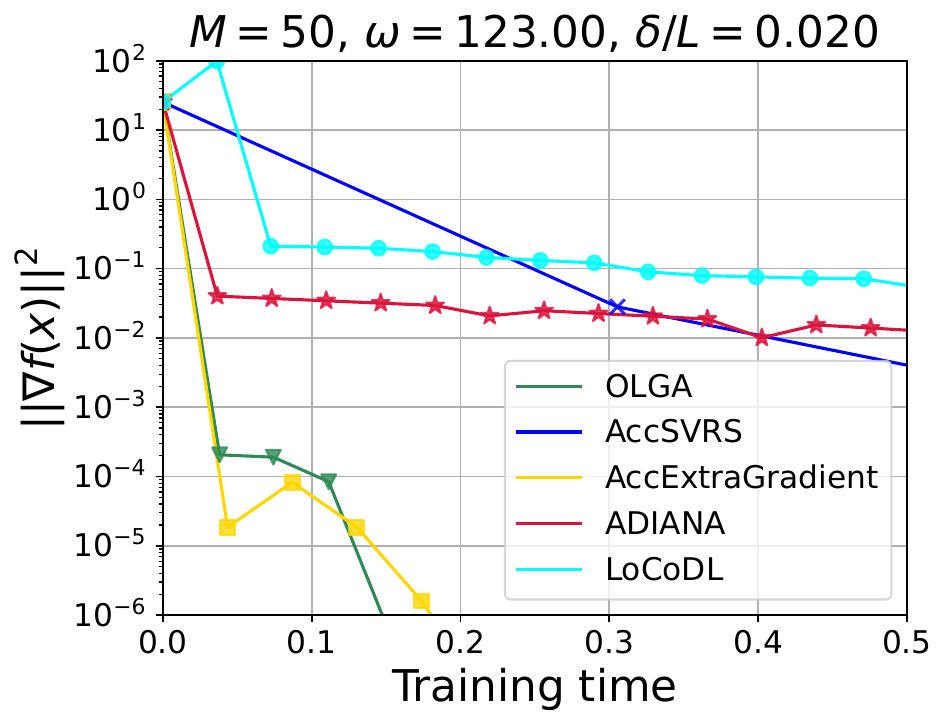}
   \end{subfigure}
   \begin{subfigure}{0.310\textwidth}
       \includegraphics[width=\linewidth]{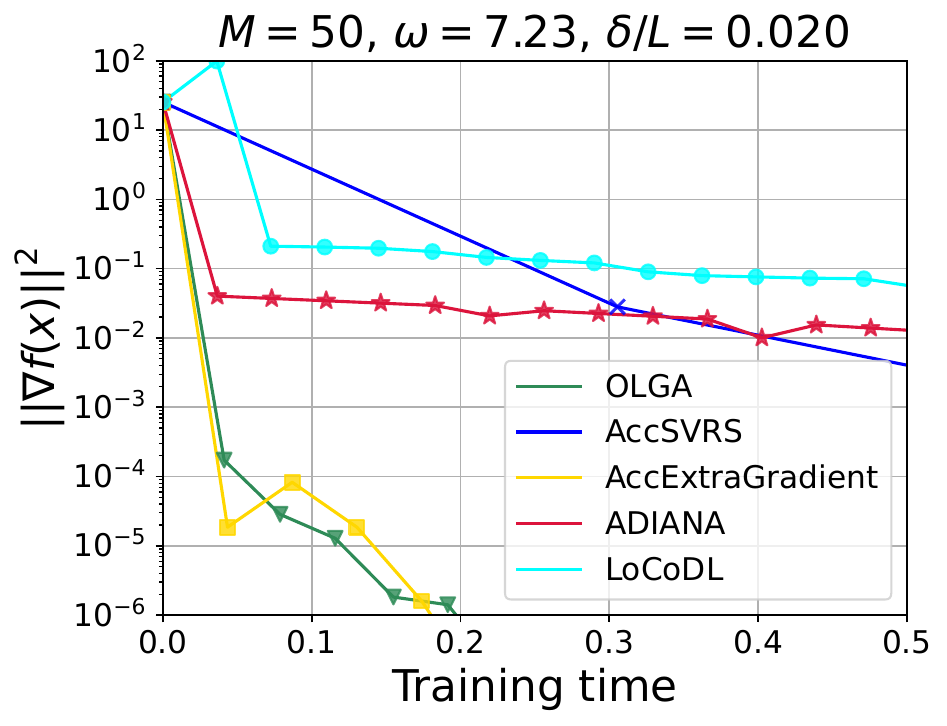}
   \end{subfigure}
   \begin{subfigure}{0.310\textwidth}
       \includegraphics[width=\linewidth]{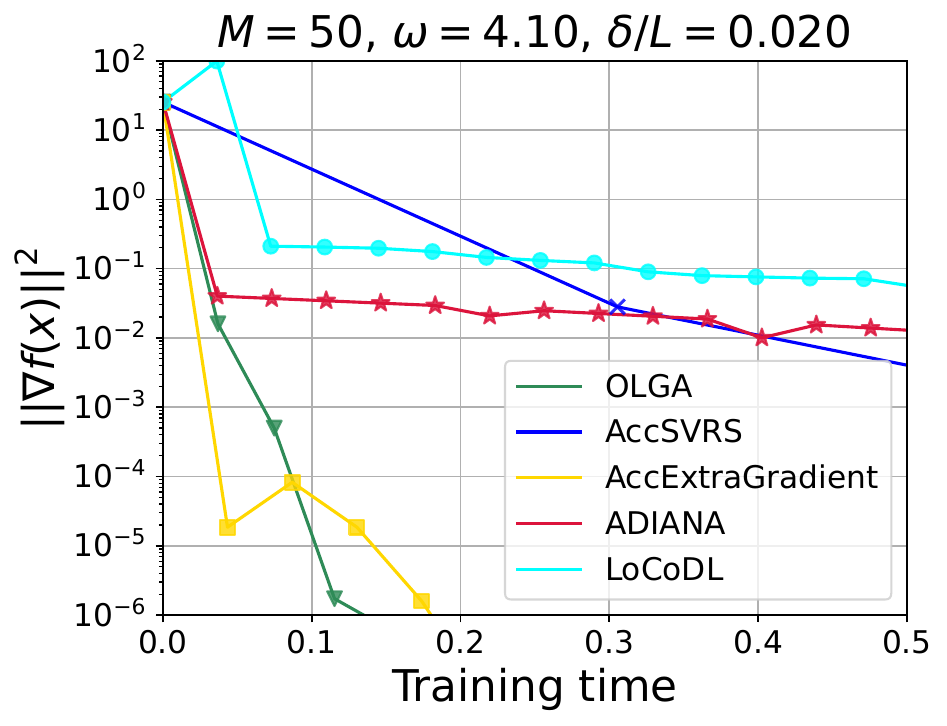}
   \end{subfigure}\\
   \begin{subfigure}{0.310\textwidth}
       \centering
       (a) $\omega = 123.00$
   \end{subfigure}
   \begin{subfigure}{0.310\textwidth}
       \centering
       (b) $\omega = 7.23$
   \end{subfigure}
   \begin{subfigure}{0.310\textwidth}
       \centering
       (c) $\omega = 4.10$
   \end{subfigure}
   \caption{Comparison of state-of-the-art distributed methods. The comparison is made on \eqref{eq:logloss} with $M=50$ and \texttt{a9a} dataset. The criterion is the training time on remote CPUs (slow connection). For methods with compression we vary the power of compression $\omega$.}
\end{figure}

\begin{figure}[h!] 
   \begin{subfigure}{0.310\textwidth}
       \includegraphics[width=\linewidth]{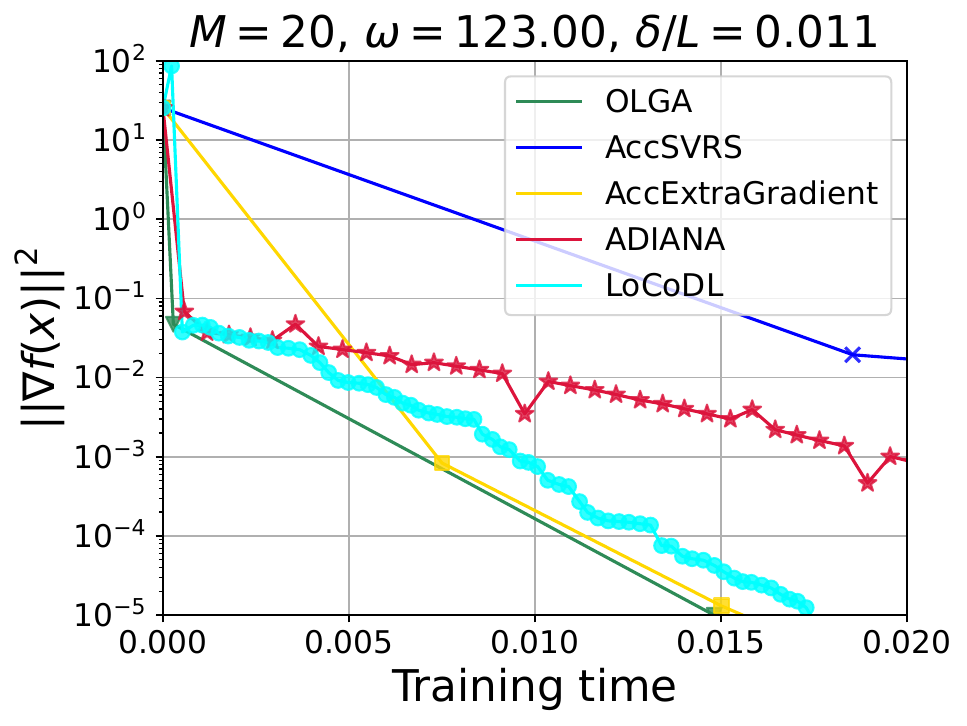}
   \end{subfigure}
   \begin{subfigure}{0.310\textwidth}
       \includegraphics[width=\linewidth]{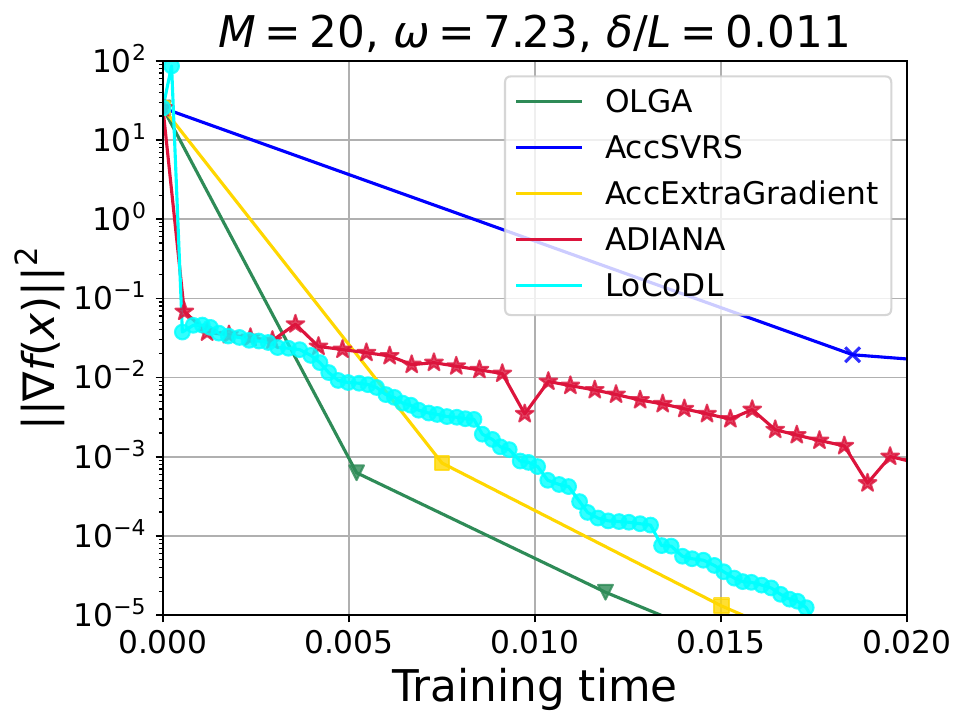}
   \end{subfigure}
   \begin{subfigure}{0.310\textwidth}
       \includegraphics[width=\linewidth]{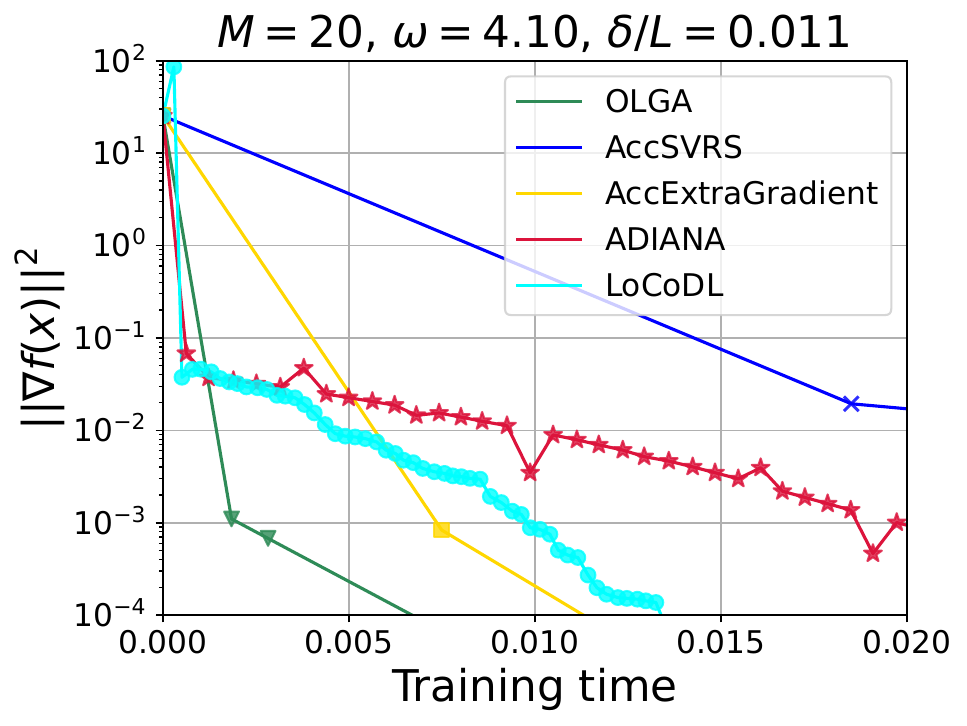}
   \end{subfigure}\\
   \begin{subfigure}{0.310\textwidth}
       \centering
       (a) $\omega = 123.00$
   \end{subfigure}
   \begin{subfigure}{0.310\textwidth}
       \centering
       (b) $\omega = 7.23$
   \end{subfigure}
   \begin{subfigure}{0.310\textwidth}
       \centering
       (c) $\omega = 4.10$
   \end{subfigure}
   \caption{Comparison of state-of-the-art distributed methods. The comparison is made on \eqref{eq:logloss} with $M=20$ and \texttt{a9a} dataset. The criterion is the training time on local cluster (fast connection). For methods with compression we vary the power of compression $\omega$.}
\end{figure}

\begin{figure}[h!] 
   \begin{subfigure}{0.310\textwidth}
       \includegraphics[width=\linewidth]{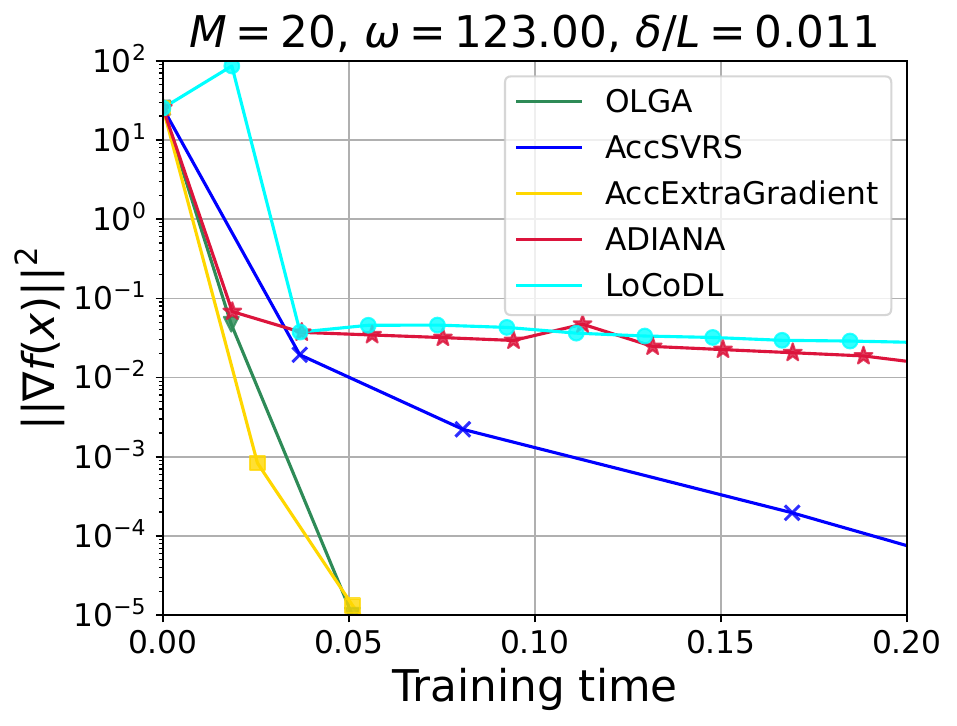}
   \end{subfigure}
   \begin{subfigure}{0.310\textwidth}
       \includegraphics[width=\linewidth]{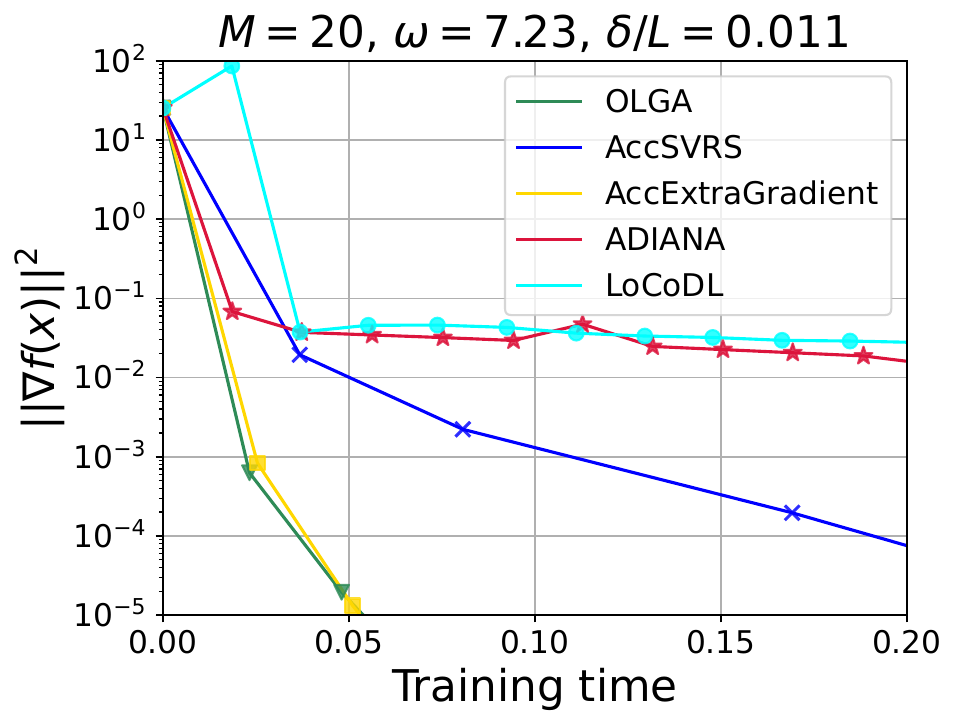}
   \end{subfigure}
   \begin{subfigure}{0.310\textwidth}
       \includegraphics[width=\linewidth]{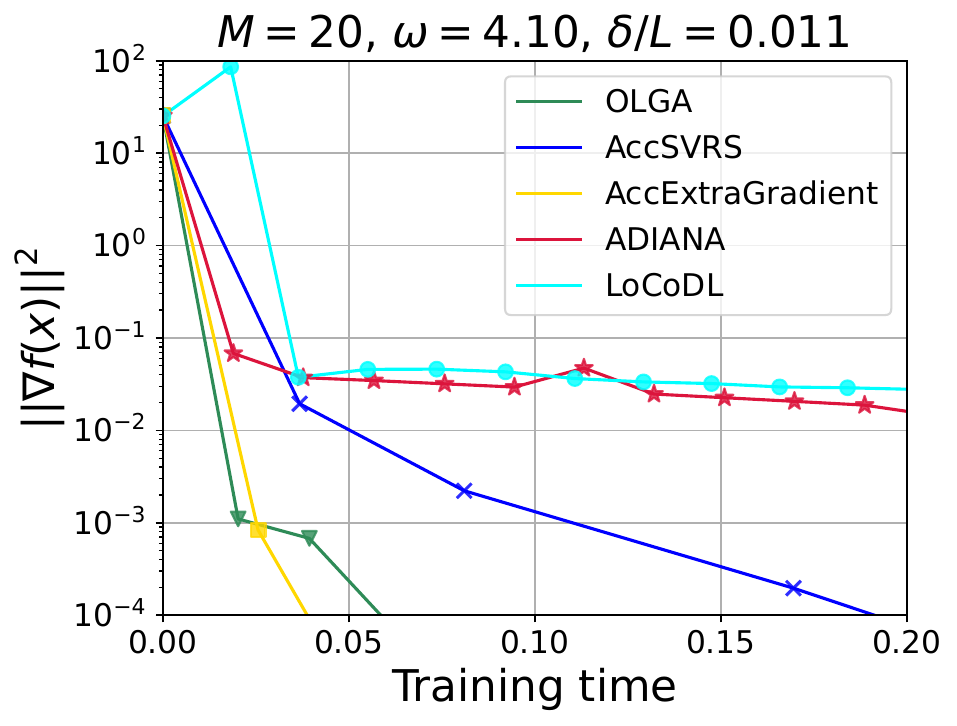}
   \end{subfigure}\\
   \begin{subfigure}{0.310\textwidth}
       \centering
       (a) $\omega = 123.00$
   \end{subfigure}
   \begin{subfigure}{0.310\textwidth}
       \centering
       (b) $\omega = 7.23$
   \end{subfigure}
   \begin{subfigure}{0.310\textwidth}
       \centering
       (c) $\omega = 4.10$
   \end{subfigure}
   \caption{Comparison of state-of-the-art distributed methods. The comparison is made on \eqref{eq:logloss} with $M=20$ and \texttt{a9a} dataset. The criterion is the training time on remote CPUs (slow connection). For methods with compression we vary the power of compression $\omega$.}
   \label{fig:end_cluster}
\end{figure}

\end{document}